\documentclass[10pt]{article}
\usepackage[affil-it]{authblk}

\usepackage{amsmath, amsthm, amssymb}
\usepackage{amssymb}
\usepackage{graphicx}
\usepackage{subfigure}

\newcommand\bi{\begin{itemize}}
\newcommand\ei{\end{itemize}}

\def\addsymbol #1: #2#3{$#1$ \> \parbox{5in}{#2 \dotfill \pageref{#3}}\\}

\newtheorem{fact}{Fact}

\newtheorem{defin}{Definition}[section]
          
          \newtheorem{teo}{Theorem}[section]
          
          \newtheorem{con}{Conjecture}
          \newtheorem{cond}{Condition}
          \newtheorem{prop}[teo]{Proposition}
          \newtheorem{lem}{Lemma}[section]
          \newtheorem{rmk}[teo]{Remark}
          \newtheorem{cor}{Corollary}[section]
          
          \newcommand{\bfact}{\begin{fact}}
          \newcommand{\efact}{\end{fact}}
          \newcommand{\Netbk}{{\cal N}^{b,\kappa}}

          \newcommand{\beq}{\begin{equation}}
          \newcommand{\eeq}{\end{equation}}
          \newcommand{\beqn}{\begin{eqnarray}}
          \newcommand{\beqnn}{\begin{eqnarray*}}
          \newcommand{\eeqn}{\end{eqnarray}}
          \newcommand{\eeqnn}{\end{eqnarray*}}
          \newcommand{\bprop}{\begin{prop}}
          \newcommand{\eprop}{\end{prop}}
          
          \newcommand{\bc}{\be\begin{array}{r@{\,}c@{\,}l}}
\newcommand{\ec}{\end{array}\ee}

          \newcommand{\bcor}{\begin{cor}}
          
          \newcommand{\ecor}{\end{cor}}
          \newcommand{\bcon}{\begin{con}}
          \newcommand{\econ}{\end{con}}
          \newcommand{\bcond}{\begin{cond}}
          
          \newcommand{\econd}{\end{cond}}
          \newcommand{\bteo}{\begin{teo}}
          \newcommand{\eteo}{\end{teo}}
          \newcommand{\brm}{\begin{rmk}}
          \newcommand{\erm}{\end{rmk}}
          \newcommand{\blem}{\begin{lem}}
          \newcommand{\elem}{\end{lem}}

          \newcommand{\ben}{\begin{enumerate}}
          \newcommand{\een}{\end{enumerate}}
          \newcommand{\bei}{\begin{itemize}}
          \newcommand{\eei}{\end{itemize}}
	 
          \newcommand{\bdf}{\begin{defin}}
          \newcommand{\edf}{\end{defin}}

          \renewcommand{\>}{&>&}

          \newcommand{\fr}{\frac}
          \renewcommand{\r}{{\mathbb R}}
          
          \newcommand{\bz}{{\bar z}}
          \newcommand{\Z}{{\mathbb Z}}
          \newcommand{\cG}{{\cal G}}

          \newcommand{\R}{{\mathbb R}}
          
          \newcommand{\E}{{\mathbb E}}
          
          \renewcommand{\P}{{\mathbb P}}

          \newcommand{\N}{{\mathbb N}}

          \newcommand{\cP}{{\cal P}}

	\newcommand{\M}{{\cal M}}
          \newcommand{\W}{{\cal W}}

          \newcommand{\cE}{{\cal E}}
          
          \newcommand{\h}{{\cal H}}

          \newcommand{\cU}{{\cal U}}

          \newcommand{\G}{\Gamma}
          
          \newcommand{\half}{\frac{1}{2}}
          
          \newcommand{\D}{{\cal D}}

\newcommand\ii{\item}

\newcommand{\btt}{\begin{theorem}}
\newcommand{\ett}{\end{theorem}}
\newcommand{\Net}{\mathcal{N}}

\newcommand{\daw}{\downarrow}
\newcommand{\uaw}{\uparrow}
\newcommand{\raw}{\rightarrow}

\newcommand{\Wl}{\mathcal{W}_l}
\newcommand{\Wr}{\mathcal{W}_r}
\newcommand{\be}{\begin{equation}}
\newcommand{\ee}{\end{equation}} 
\newcommand{\n}{\eta}

\newcommand{\Mzero}{{\cal M}}

\newcommand\no{\nonumber}
\newcommand\eps{\epsilon}

\newcommand\Q{\mathbb{Q}}

          \newcommand{\cL}{{\cal T}}
          \newcommand\sqr{\vcenter{
          \hrule height.1mm
          \hbox{\vrule width.1mm height2.2mm\kern2.18mm\vrule width.1mm}
          \hrule height.1mm}}        

\title{ Perturbations of the Voter Model in One-Dimension}
\author[1,2]{C.M. Newman}
\author[2]{ K. Ravishankar}
\author[3,4]{E. Schertzer \footnote{Corresponding author: emmanuel.schertzer@gmail.com}}
\affil[1]{Courant Institute of Mathematical Sciences, New York University, 251 Mercer St., New York, NY 10012, USA.}
\affil[2]{NYU-ECNU Institute of Mathematical Sciences at NYU Shanghai, 3663 Zhongshan Road North, Shanghai 200062, China.} 
\affil[3]{UPMC Univ. Paris 06,
Laboratoire de Probabilit\'es et Mod\`eles Al\'eatoires, CNRS UMR 7599, Paris, France.}
\affil[4]{Coll\`ege de France,
Center for Interdisciplinary Research in Biology, CNRS UMR 7241, Paris, France.}     

\begin{document}

\maketitle

\begin{abstract}
We study the scaling limit of a large class of voter model perturbations in one dimension,
including stochastic Potts models, to 
a universal limiting object, the continuum voter model perturbation.  The perturbations can 
be described in terms of bulk and boundary nucleations of new colors (opinions).
The discrete and continuum (space) models are obtained from their respective duals, the
discrete net with killing  and Brownian net with killing. These determine the color genealogy
by means of {\it reduced graphs\/}.
We focus our attention on models where the voter and boundary nucleation dynamics 
depend only on the colors of nearest neighbor sites, for which convergence of the discrete net with killing 
to its continuum analog was proved in an earlier paper by the authors. We use some detailed properties 
of the Brownian net with killing to prove voter model perturbations convergence to its continuum counterpart.  A crucial property of reduced graphs is that  even in the  continuum, they are finite almost surely.
An important issue is how vertices of  the continuum reduced graphs are strongly approximated
by their discrete analogues.
\end{abstract}

 \section{Introduction}

\subsection{General voter model perturbations}

Motivated by problems in Biology and Statistical Physics (see e.g., \cite{MDDGL99} \cite{NP00} \cite{OHLN06}), Cox, Durrett, and Perkins \cite{CDP11}  considered a class of interacting particle
systems on $\Z^d$,  called voter model
perturbations (or VMP),  whose rates of transition are close 
to the ones of a classical voter model.
Informelly, 
they considered models whose transition rates
is parametrized by $\eps$ with $\eps\to0$, so that 
the transition rates
$c^{\eps}_{x,\eta}(j)$ (the transition rate of site $x$ to state $j$ given a configuration $\eta$
for the model with parameter $\eps$)
converges to the rate of a (possibly non-nearest neighbor) zero-drift voter model.
For $d\geq3$, $q=2$ (i.e. for spin systems)
the authors show that, under mild technical assumptions,  the properly rescaled
local density 
of one of the colors converges to the solution
of an explicit reaction diffusion equation.
They are then able to combine this convergence result together with 
some percolation arguments to 
show some properties of the particle system when $\eps$ is small enough.

In this work, we consider a similar problem when $d=1$. We will construct
a natural continuous object --- the continuum voter model perturbation (CVMP)--- 
which will be seen to be the scaling limit of a certain type of
voter model perturbation.
As we shall see, such a limit can not be described in
terms of a reaction-diffusion equation anymore (contrary to the case $d\geq3$), 
but directly in terms of a duality relation with the Brownian net and Brownian net with killing. 
The former is 
a family of one-dimensional branching-coalescing
Brownian motions, as first introduced by 
Sun and Swart \cite{SS07} and then studied further by Newman, Ravishankar and Schertzer \cite{NRS08}; the latter is an extension introduced in \cite{NRS15}.
See also \cite{SSS15} for a recent review of those objects.
 
We will use some of the properties of the Brownian net 
to define continuous versions of the voter model perturbations, and
show some properties of these continuous objects. 
 We expect that a large class of VMP's in one dimension in both discrete and continuous time settings will converge to (the universal limiting object) CVMP.  Our goal in this paper is to  prove the
 convergence of a particular class of (discrete time) VMP to  CVMP.
  In the spirit of \cite{CDP11},
we hope that our results will provide some insights into
a large class of interacting particle systems which are close enough to
the voter model in dimension one.

\subsection{Decomposition. Bulk and Boundary Nucleations.}
\label{decomposition-1}

In the present work, we 
are interested in the discrete time version
of VMP considered in \cite{CDP11}.
We consider interacting particle systems
on $\Z$ such that the color at each site $x$, 
is updated at time $t+1$
according to some transition probability on $\{1,\cdots,q\}$,
that we denote $P_{x,\eta}^\eps$ (transition probability at site $x$
given a configuration $\eta$
at time $t$).
We describe below a large class of models of this type. While  our 
results
are for a more limited subclass, we believe our approach 
is applicable to the more general class. This would require 
proving convergence of the discrete net with more general jump kernels and
branching mechanisms to the Brownian net. (Such improved convergence
results are under active investigation by one or more groups
of researchers. See also Section \ref{perspective} for more discussion on this topic.)

In our particular context, we
consider a family of such models
parametrized by a number $\eps$,
i.e.  a family of models with corresponding transition probabilities
$\{P_{x,\eta}^\eps\}_{\eps>0}$ and such
that $P_{x,\eta}^\eps$ is close to the transition probabilities
of a discrete-time voter model when $\eps$ is small, i.e.  
$$
\forall i \in\{1,\cdots,q\}, \ \ \lim_{\eps\daw0} P_{x,\eta}^\eps(i) =  f_{x,\eta}(i)
$$
where 
\be\label{voter-f}
f_{x,\eta}(i) : = \sum_{y\in\Z} K(x-y)  1_{\eta(y)=i}
\ee
for a given zero mean transition kernel $K$. For many applications (see \cite{CDP11} and the Appendix where we investigate the spatial
Lotka-Volterra model, the stochastic Potts model  and  the noisy biased voter model),
one can 
decompose the perturbations part of our models
into two parts : a first part that we call the boundary noise
and a second part called the bulk noise.
The bulk noise simply 
denotes a probability distribution
on $\{1,\cdots,q\}$ --- independent of $x$ and the present configuration $\eta$
of the system. 
The boundary noise 
requires a little more explanation. Informally, one can think of 
it
as a transition probability
$B_{x,\eta}$ 
which is only non-trivial if there is some local impurity
around $x$ --- hence the terminology
of boundary noise. More formally, a boundary noise
can always be written under the form
\begin{equation}\label{B_eps}
B_{x,\eta}(i)  \ = \ \E^Q\left( g_{\eta(x+Y_1),\cdots,\eta(x+Y_N)}(i)\right)
\end{equation}
where (1) $N$ is a fixed integer, (2) for every n-uplet of colors $(c_1,\cdots,c_N)$, the quantity $g_{c_1,\cdots,c_N}$ is a probability  
distribution on $\{1,\cdots,q\}$ with $g_{c,\cdots,c}=\delta_{c}$, (3) $Q$ is a probability ditribution
on $\Z^N$. In words, 
$B_{x,\eta}$ chooses $N$ neighbors at random according to 
the probability distribution $Q$, and will align to its neighbors color if
the coloring is uniform.

The reason for distinguishing
between bulk and boundary noise 
is that in many applications,
we shall see in the Appendix that $P_{x,\eta}^\eps$
can be decomposed into three parts, i.e., 
\beqn
P_{x,\eta}^\eps  
& =   w_\eps \underbrace{ f^\eps_{x,\eta}}_{ \mbox{voter noise}} 
+ \  b_\eps \underbrace{B^\eps_{x,\eta} }_{\mbox{boundary noise} }  + \ \kappa_\eps \underbrace{p^\eps}_{\mbox{bulk noise}},
\label{voter-eps}\eeqn
with $w_\eps,b_\eps,\kappa_\eps$ being three non-negative numbers
adding up to $1$ and such that for every $x,\eta$ (1) each part 
of the decomposition defines a probability measure on $\{1,\cdots,q\}$,
(2)  $f^\eps_{x,\eta} =  \left(f_{x,\eta} + O(\kappa_\eps^2,b_\eps^3)\right)$ where $f_{x,\eta}$ is defined in (\ref{voter-f})
(3)  $B^{\eps}_{x,\eta}$ is a boundary noise converging
to a limiting boundary noise $B_{x,\eta}$,
(4) $p^\eps$ is a bulk noise converging
to a limiting bulk noise $p$.   
We  note that the case $b_\eps=0$ (no boundary nucleation) has been treated in \cite{FINR05}, and that it turns out to be considerably 
simpler than the general case treated in this paper.

Given the distributions described in the previous paragraph,
the color transitions are then most easily understood via a
two-step procedure with the first step determining which of three possible moves
to implement at the particular $(x,t)$ under consideration --- a {\it walk\/}
move, a {\it kill\/} move or a {\it branch\/} move with respective probabilities $w_\eps,\kappa_\eps,b_\eps$ --- and the second step
determining the possible change from the initial color to some new color. 
This terminology, which
may at present sound rather mysterious, is taken from the Branching-Coalescing-Killing 
random walks model,
which will be seen to be dual to the voter model perturbation. 
Once the type of move has been
chosen, here are the rules for the second step.
\bi
\item Walk: choose a color according to the probability distribution
$f^\eps_{x,\eta}$. In other words, with high probability,
choose the color of one of your neighbors chosen at random according to the kernel $K(x,\cdot)$. 
\item Kill: choose a color according to the probability distribution $p^\eps$.
\item Branch: first pick $(Y_1,\cdots,Y_N)$ according to the law $Q$, and then choose a color according to the distribution $g_{\eta(x+Y_1),\cdots,\eta(x+Y_N)}$.
\ei

\subsection{Simple VMP and convergence to the continuum VMP}
\label{cv-nn}
This work is a first step towards the understanding of the scaling limit of systems
of type (\ref{voter-eps})
when $\eps\downarrow0$.  We  will restrict ourselves 
to the case of nearest neighbor voter model perturbation.
By that, we mean that the distributions corresponding
to the voter and boundary noise
$f_{x,\eta}^\eps$ and $B_{x,\eta}^\eps$ are solely
determined by the states $\{\eta(x-1),\eta(x+1)\}$. 
Furthermore, we will assume that 
$$
f_{x,\eta} \ = \ f^\eps_{x,\eta}   \ = \  \half(\delta_{\eta(x-1)}+\delta_{\eta(x+1)}).
$$
Regarding the boundary noise,
we will assume the existence
of
a family of probability distributions $\{g_{i,j}^\eps\}_{i,j\leq q}$ on $\{1,\cdots,q\}$,
each distribution being
indexed by two elements on  $\{1,\cdots,q\}$ such that 
$$
B^\eps_{x,\eta} = g_{\eta(x-1),\eta(x+1)}^\eps
$$
with $g_{i,i}^\eps=\delta_i$. 
 Furthermore,
we will assume the existence of a limiting family 
of measures $\{g_{i,j}\}$ and $p$ such that $g_{i,j}^n := g_{i,j}^{\eps_n} \rightarrow g_{i,j}$ and 
$p^n:=p^{\eps_n} \rightarrow p$,
as the sequence  $\eps_n \downarrow0$. 
In the following, 
such a model with be refered to as a simple Voter Model Perturbation (VMP)
with boundary noise $\{g^\eps_{i,j}\}_{i,j}$ and
bulk noise $p^\eps$.

We note that in this particular case, the transition rates only depend
on the states of the two nearest neighbors, but not on
the current state of the site under consideration. 
As we shall see, this hypothesis is made 
mostly for technical reasons and is done solely
to take advantage of the known convergence result
of coalescing-branching discrete (time and space)
random walks to the Brownian net; such a 
result for the continuous time  or non-nearest neighbor version
has still to be established. 
The continuous time setting, together
with the study of more general voter model perturbations 
 in one dimension (including
the non-nearest neighbor case and more general boundary nucleation
mechanism)
will be the subject of future work.

Before stating our main theorem, we recall that
the dynamics we consider are all nearest-neighbor in that the update
probabilities of $\eta_t(x)$
at a particular lattice site $x$ depend only on $(\eta_t(x-1),\eta_t(x+1))$.
This property (which is
also the case for the standard nearest neighbor voter model) implies that the Markov chain
is reducible with complete independence between $\{\eta_t(x)\}$ on the 
even ($\Z^2_{even}=\{(x,t): x+t \, {\mbox {is even} }\}$) and odd 
($\Z^2_{odd}=\{(x,t): x+t \, {\mbox {is odd} }\}$) subsets
of discrete space-time. Thus one may restrict attention to one of these
irreducible components.

Given a realization of the odd component of the 
voter model perturbation with
parameter $\eps_n$, we
define
$\theta^{\eps}(x,t)$ to be the color
of site $(x,t)$ at time $t$. In particular,
$\theta^{\eps}(\cdot)$ defines  a 
random mapping 
from $\Z_{odd}^2 \cap \{(x,t)\in\Z^2 \ : \ t\geq0 \}$   
to $\{1,\cdots,q\}$. 
Analogously, 
we will define the CVMP as a mapping from 
$\R\times\R^{+}$ to the set of subsets of
$\{1,\cdots,q\}$,
allowing one space-time point to take several colors.
The next theorem enumerates some of the properties of the
CVMP. 
The precise construction and description will be carried out in Section \ref{coloring}.

\bteo\label{teo1}
Let $\lambda,\{g_{i,j}\},p$ be a family of probability distributions on $\{1,\cdots,q\}$, with $g_{i,i}=\delta_i$.  
Let $b,\kappa \geq0$.
There exists a random mapping 
$(x,t)\raw \theta(x,t)$ from $\R\times\R^{+}$
to subsets of $\{1,\cdots,q\}$, 
with
the following properties :
\begin{enumerate}
\ii For every deterministic $(x,t)$, $|\theta(x,t)|=1$  a.s.. 
\ii (Coarsening) For every deterministic $t$, $|\theta(x,t)|\leq 2$ a.s.. Moreover,
the set
 $\{x \ : \ |\theta_{(x,t)}|=2\}$ is locally finite 
and  partitions
the line into intervals of uniform color, i.e.
the color of $x\to\theta(x,t)$ between two consecutive points of this set
remains constant. 
\ii (Scaling Limit) 
Let $\{\eps_n\}$ be a sequence of positive real numbers converging to $0$.
Let $\theta^{n}(\cdot,\cdot):=\theta^{\eps_n}(\cdot,\cdot)$ be the simple voter model perturbation
characterized by the boundary noise and bulk noise $(\{g_{i,j}^n\}_{1\leq i, j\leq q},p^n)$,
with
branching and killing parameters $(b_n,\kappa_n)$ such that
\begin{enumerate} 
\item[(i)] there exists 
$b,\kappa\geq0$ such that 
$$
b_n / \eps_n \to b, \ \ \mbox{and } \kappa_n / \eps_n^{2} \to \kappa \ \  \mbox{as $n\to\infty$,}
$$
  \item[(ii)] there exists $(\{g_{i,j}\}_{1\leq i, j\leq q},p)$
a family of probability distributions on $\{1,\cdots,q\}$ such that 
$$
\forall i,j, \ g_{i,j}^{n} \to g_{i,j},  \ \ \mbox{and } \ \ p^n \rightarrow p \ \ \mbox{ as $n\to \infty$},
$$
\item[(iii)] the initial distribution of the particle system is given
by a product measure $\nu$    (not depending on $n$)    with
one-dimensional marginal given by $\lambda$.
\end{enumerate}
For every $(x,t)\in\R^2$, let $S_\eps((x,t)):= (\eps x, \eps^2 t)$. 
   For $i=1,\cdots,k$, let $z^i$ be a deterministic point in $\R^2$ and
let $\{z^i_n\}_{n\geq1}$ be a deterministic sequence in $S_{\eps_n}(\Z_{odd}^2)\cap\{(x,t): t>0\}$
such that
$
 \ \ \lim_{n\to\infty} \ 
z^i_n \ = \  z^i.
$
Then
\beq
\label{conv;colors}
\P\left( \ \theta^{n}(  z^1_n ) ,\cdots, \theta^n (z^k_n) \in\cdot \right) 
\to \P\left(  \theta( z^1),\cdots, \theta(z^k)   \in\cdot \right).
\eeq
\end{enumerate}
\eteo

\brm
In the Appendix, we show that the stochastics Potts model in one-dimension at low temperature -- i.e., high $\beta$ -- is a VMP. Further,
we shall see that the scaling in Theorem \ref{teo1}(3)(i) emerges naturally in this setting. See Remark \ref{rmk2} in the Appendix.
\erm

\subsection{Duality to branching-coalescing-killling random walks}
\label{coloring-algo}


In this section,
we start by introducing a
percolation model on $\Z_{odd}^2$ (already introduced in \cite{MNR13}) and we
show
how a random 
coloring algorithm of this percolation
configuration
is dual to simple voter model perturbations.

{\bf An oriented percolation model.}
Let $b,\kappa \geq0$ with $b+\kappa \leq 1$.  
Each site $v=(x,t)\in  \Z_{odd}^2$ --- where $x$ is interpreted as a space coordinate
and $t$ as a time coordinate ---  
has two nearest neighbors
with higher time coordinates
--- $v_r=(x+1,t+1)$ and $v_l=(x-1,t+1)$ ---
 and 
$v$
is randomly (and independently for different $v$'s)
connected 
to a subset of its neighbors
$v_r$ and $v_l$
according
to the following distribution.
\bi
\item With probability $\fr{(1-b-\kappa)}{2}$, draw the
 arrow
$(v \raw v_l)$ (i.e. starting at $v$ and anding at $v_l$); 
\item with probability $\fr{(1-b-\kappa)}{2}$
draw the arrow 
 $(v\raw v_r)$;  
\item with probability $b$, draw the
two arrows $(v\raw v_r)$ and $(v\raw v_l)$;
\item finally do not draw any arrow with 
probability $\kappa$. See Fig. \ref{Tux}.
\ei
If we denote by $\cE^{b,\kappa}$ the resulting random
arrow configuration,
the random graph $\cG^{b,\kappa}=(\Z_{odd}^2,\cE^{b,\kappa})$
defines a certain type of
$1+1$ dimensional percolation model
oriented forward in the $t$-direction.
In this percolation model,
the vertices 
with two outgoing arrows will
be referred to as branching points
whereas points with no arrow 
will be referred to as killing points.

By definition, a path $\pi$ along $\cE^{b,\kappa}$
will denote
a path starting from any site of $\Z_{odd}^2$
and following the random arrow configuration
until getting killed or reaching $\infty$. More precisely, 
a path $\pi$ is the graph of a function
defined on an  interval in $\R$
of the form
  $[\sigma_\pi,e_\pi]$, with $\sigma_\pi, e_\pi\in\Z\cup\{\infty\}$   such that
$e_{\pi}=\infty$ or
$(\pi(e_\pi),e_\pi)$ is a killing
point, for every $t\in[\sigma_\pi,e_\pi)$,
$(\pi(t),t)$ connects to $(\pi(t+1),t+1)$
and $\pi$ is linear between $t$ and $t+1$.

Considering the set of all
the paths along  $\cE^{b,\kappa}$,
one generates
an infinite family --- denoted by $\cU^{b,\kappa}$ ---
that can loosely be described
as a collection of graphs of 
one dimensional coalescing
simple random walks, that branch
with probability $b$ and get killed with probability $\kappa$.
Roughly speaking,
a walk at space-time site
$v$ can create two new walks
(starting respectively at $v_l$
and $v_r$)
with probability $b$ and can be killed
with probability $\kappa$; two walks 
move independently when they are apart
but become perfectly correlated (i.e., they coalesce)
upon meeting at a space-time point. 
In the following, $\cU^{b,\kappa}$
will be referred to
as a system of branching-coalescing-killing random
walks (or in short, BCK) with
parameters $(b,\kappa)$, or equivalently, and in analogy with their continuum counterpart, 
as a discrete net with killing (in analogy with
the Brownian net with killing as introduced in \cite{NRS15}).

\bigskip

{\bf Backward discrete net.} Our percolation model is oriented forward in time. 
We define the backward BCK --- denoted by $\hat \cU^{b,\kappa}$ ---   as the backward in time 
percolation model obtained from the BCK
by a reflection about the $x$-axis. See Fig. \ref{Tux2}.

\begin{figure}[h!]
\begin{minipage}[b]{0.5\linewidth}
\centering
\includegraphics[scale=.3]{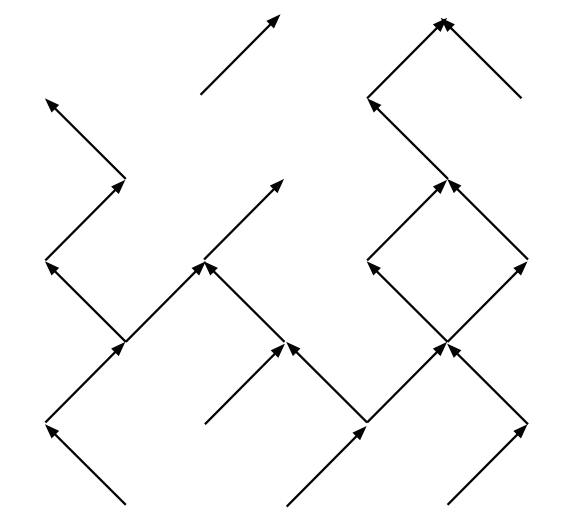}
\caption{Forward discrete net with killing}
\label{Tux}
\end{minipage}
\hspace{0.5cm}
\begin{minipage}[b]{0.5\linewidth}
\centering
\includegraphics[scale=.3]{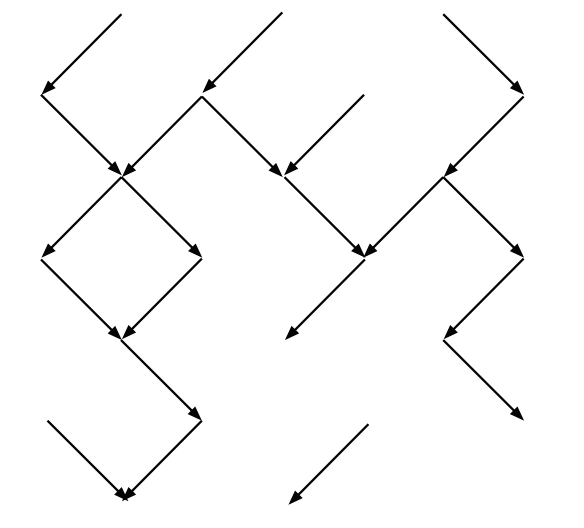}
\caption{Backward discrete net with killing}
\label{Tux2}
\end{minipage}
\end{figure}
\bigskip

{\bf A coloring algorithm.} We now describe how
simple voter model perturbations 
(as described in Section \ref{cv-nn}) are exactly dual to the backward BCK (i.e., $\hat \cU^{b,\kappa}$). 
To describe this duality relation, we consider 
the subgraph of $\hat {\cal G}^{b,\kappa}$ (the backward percolation model) 
whose set of vertices is given by ${\mathbb Z}_{odd}^2\cap\{(x,t) \  : \ t\geq0\}$. Then, 
we equip each vertex $z\in {\mathbb Z}_{odd}^2\cap\{(x,t) \  : \ t\geq0\}$
 with an independent
uniform (on $[0,1]$) random variable $U_z$.

We start by coloring the leaves
as follows (see Fig. \ref{config4}). The color assigned to  a point $z$ with  time coordinate equal to $0$ ---  denoted by $\theta^\eps(z)$ --- 
is the only integer 
in $\{1,\cdots,q\}$
such that
$$
U_z\in[\sum_{j\leq \theta^\eps(z)-1} \lambda(j), \sum_{j\leq \theta^\eps(z)}\lambda(j)).
$$
(With the convention that $\lambda(0)=0$.)
The color assigned to  a killing point is the integer in $\{1,\cdots,q\}$
such that
$$
U_z\in[\sum_{j\leq \theta^\eps(z)-1} p^\eps(j), \sum_{j\leq \theta^\eps(z)} p^\eps(j))
$$
where we recall that $p^\eps(\cdot)$ is the probability distribution determining the transition
when a bulk nucleation occurs (with the convention that $p^\eps(0)=0$).

For every other vertex $z\in \Z_{odd}^2\cap\{(x,t)\ : \ t\geq0\}$,
we consider 
the component
of our (backward) oriented percolation model
originated from $z$ and restricted to the upper half plane. 
This defines a certain acyclic directed graph $\hat G^z=(\hat V^z, \hat E^z)$
whose vertices have out-degree 
at most $2$  -- see Fig. \ref{config5}.
For every point in $\hat V^z$, we assign a color to each nodes sequentially 
from the leafs to
the root by applying the following algorithm, where 
$\hat V_n^z$ will denote the set of color-assigned vertices at step $n$
of the algorithm.
\bi
\ii {\bf Step 1} $\hat V_1^z$ is the set of leaves, whose colors have already been assigned. 
\ii {\bf Step n$>1$} If $\hat V^z\setminus \hat V_n^z=\emptyset$ stop. 
Otherwise,
pick a vertex $z$ in $\hat V_n^z$ such
that $z$ is connected to $\hat V_n^z$, i.e., all its nearest (directed) neighbor(s) belong
to $\hat V_n^z$. 
Next, if $z$ connects to a single vertex or
if the color of its two neighbors match, assign to it the color of its neighbor(s).
If the colors of the two neighbors disagree,
with two colors $k\neq l$,
then the site is assigned
  a color $i$   
such that 
\be\label{fghj}
U_z\in[\sum_{i\leq\theta^\eps(z)-1}g^\eps_{k,l}(i),\sum_{i\leq\theta^\eps(z)}g^\eps_{k,l}(i+1)).
\ee
\ei
See Figs \ref{config5}-\ref{config666} for a concrete example in the case $q=3$ (white, grey, black corresponding respectively to $1,2,3$)
and $p^\eps, \lambda$, and $g_{i,j}^\eps$ for $i\neq j$  
are uniform distributions on $\{1,2,3\}$. We note that
this specific choice  for the bulk and boundary noises correspond to the discrete time stochastic
Potts model considered in the Appendix (see Section \ref{Potts-sect}).
\brm
At step $n$ of the previous algorithm, there can be multiple $z'\in V^z\setminus V_{n}^z$, such that all of its (directed) neighbors are in 
$\hat V_{n}^z$. (In Fig \ref{config555}, there are two such vertices -- one connecting to the top white circle and one connecting to the top black circle).
However, we let convince herself that the final coloring of the graph is independent of the choice of $z'$ at step $n$.
\erm

\brm
Our
coloring is consistent in the sense that if $z'\in \hat V^z$, the color
assigned to $z'$ in the coloring algorithm applied to $\hat V^z$
is identical to $\theta^\eps(z')$ -- i.e., the coloring algorithm applied to $\hat V^{z'}$.
\erm

\begin{figure}[h!]
\begin{minipage}[t]{0.4\linewidth}
\centering
\includegraphics[width=\linewidth]{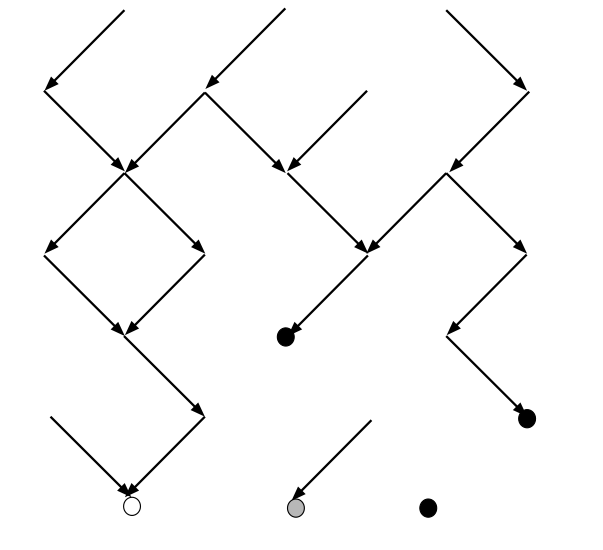}
\caption{Realization of the coloring ($q=3$ colors (white, grey, black)) 
for the leaves of the backward percolation model restricted to the upper half plane.}
\label{config4}
\end{minipage} \ \ \ \ \ \ \ 
\begin{minipage}[t]{0.45\linewidth}
\centering
\includegraphics[width=\linewidth]{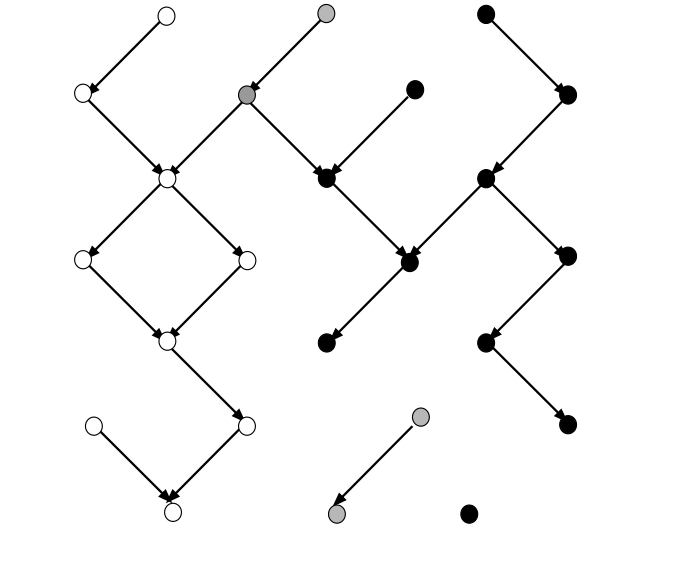}
\caption{Final coloring of a section of $\Z^2_{odd}\cap \{(x,t) : t\geq0\}$ after applying the coloring algorithm to every vertex.}
\label{config444}
\end{minipage}
\end{figure}

\begin{figure}[h!]
\begin{minipage}[t]{0.4\linewidth}
\centering
\includegraphics[width=\linewidth]{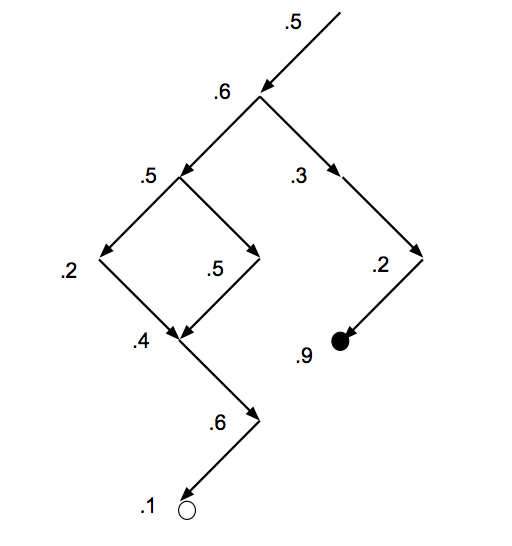}
\caption{subgraph originated from the top-middle point of Fig. \ref{config4}
decorated with fictive variables $U_z$.}
\label{config5}
\end{minipage}  \ \ \ \ \ \ \ 
\begin{minipage}[t]{0.4\linewidth}
\centering
\includegraphics[width=\linewidth]{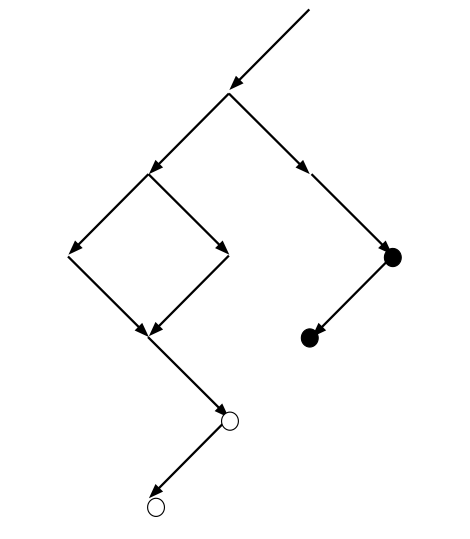}
\caption{Coloring at Step 3 of the algorithm.}
\label{config555}
\end{minipage}
\\
\begin{minipage}[t]{0.4\linewidth}
\centering
\includegraphics[width=\linewidth]{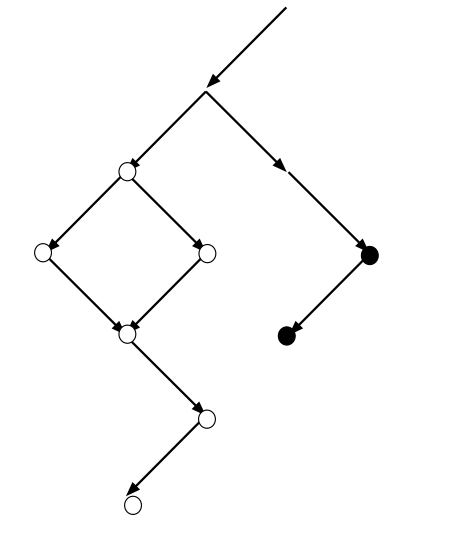}
\caption{Coloring at Step 7 of the algorithm.}
\label{config66}
\end{minipage} \ \ \ \ \ \ \ 
\begin{minipage}[t]{0.4\linewidth}
\centering
\includegraphics[width=\linewidth]{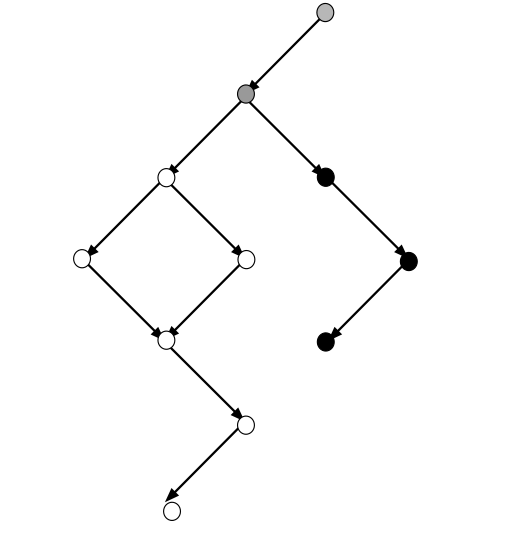}
\caption{Final coloring of the subgraph.}
\label{config666}
\end{minipage}
\end{figure}

\bprop
\label{pottsdual}
Let us consider the simple VMP
with characteristic distributions $(\{\frac{1}{2}(\delta_i+\delta_j)\}, \{g^\eps_{i,j}\}, p^\eps)$,
branching-killing parameters $b_\eps,\kappa_\eps\geq 0$, 
starting with an initial  $\nu$-coloring of the line, where $\nu$ 
is   an i.i.d product probability measure    on $\{1,\cdots,q \}^\Z$.

For every $z^1,\cdots,z^k\in\Z_{odd}^2$, the n-tuple $(\theta^\eps(z^1),\cdots,\theta^\eps(z^k))$, as described in the coloring algorithm above, coincides with the
$k$-dimensional distribution of this VMP with initial distribution $\nu$.
\eprop

\begin{proof}
Since coalescing random walks are the dual of the standard voter model \cite{HL75} \cite{L04}, it is clear that in the absence of killing and branching the color genealogy of the VMP is given by coalescing random walks . At a bulk nucleation site (killing point) a color independent of the previous color evolution is assigned to the site. This leads to the conclusion that if a path in the color genealogy evolution hits a bulk nucleation point then it should be killed at that point. The color at a boundary nucleation point is determined by the colors of the two adjacent sites. The colors of these two sites are determined by following the color genealogy evolution starting at these sites until reaching time zero or hitting a killing point as given by the dual percolation model starting at these two sites. The uniform random variables at the bulk and boundary nucleation sites along with the rules for determining the color at these sites given earlier ensure that the colors at these sites are chosen with the distribution $p^\epsilon$ and $g_{k,l}^\epsilon$ respectively.
\end{proof}

\subsection{Scaling limit}
\label{scaling-limnit}

In \cite{NRS15}, we introduced the
Brownian net with killing (or $\Net^{b,\kappa}$ 
where $b$ and $\kappa$
are continuum branching and killing parameters)
that was shown to emerge as the scaling limit of a BCK system
with small branching and killing parameters (see Theorem \ref{onewebteo} below for a precise statement).

In light of the previous section, 
it is natural to construct the CVPM 
as a process
dual
to the Brownian net with killing; the duality
relation between the two models 
being described in terms
of the random coloring of
the ``continuum graphs''
induced by the Brownian net.

The main difficulty in defining
the CVMP
lies in the fact that
the set of vertices of this ``continuum graph''
is dense in $\R\times\R^+$.
In order to deal with this 
extra difficulty, we
now discuss some 
remarkable features of the coloring algorithm
described in the previous section.

\bigskip

\bdf[Relevant Separation Point and Reduced Graph -- Fig. \ref{config7}-\ref{config8}]
\label{reduced-graph-df}
Let $G=(V,E)$ be a finite acyclic oriented graph. 
We will say that an intermediate node $z$ (i.e. a node which is neither a root nor a leaf)
of the graph is relevant iff there exists
two paths of $G$    passing through $z$ and reaching 
 the leaves such that
they do not meet between $z$ and the leaves. Any other intermediary point 
will be called irrelevant.

We define the reduced version of 
$G$, denoted by $\tilde G=(\tilde V,\tilde E)$, as the oriented graph of $G$ obtained 
after ``skipping" all irrelevant
points in $V$. More precisely,
the vertices
of $\tilde V$ are obtained from
$V$ by removing
the set of irrelevant separation points and
placing a directed edge between
two points $z,z'$ in $\tilde V$
iff
there exists a path $\pi = (z_1,\cdots,z_n)$ in $G$
with $z_1=z$ and $z_n=z'$
and such that $z_i$ is irrelevant for $i\neq 1,n$.

\edf

\begin{figure}[h!]
\begin{minipage}[b]{0.5\linewidth}
\centering
\includegraphics[scale=.3]{bck22}
\caption{Colored graph $\hat G^z$ originated from a point $z$.}
\label{config7}
\end{minipage}
\hspace{0.5cm}
\begin{minipage}[b]{0.5\linewidth}
\centering
\includegraphics[scale=.3]{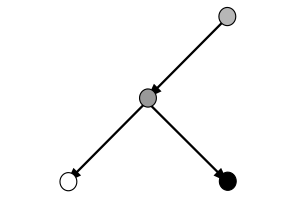}
\caption{Reduced version of the graph presented in the left-panel. }
\label{config8}
\end{minipage}
\end{figure}

\bprop\label{reduced-graph-df2}
Let $G$ be a finite rooted 
acyclic oriented graph whose vertices have at most out degree $2$ 
and let us assume that the vertices $z$ of the graph
are equipped with i.i.d. random variables $\{U_z\}$,
uniformly
distributed on $[0,1]$.
For a given coloring of the leaves,
our coloring algorithm (as described in Section \ref{coloring-algo})
applied separately  to $G$ and to the reduced graph $\tilde G$
induces the same coloring of the verties $\tilde V$ (using the same set of i.i.d. uniform random variables for the reduced graph).
\eprop
\begin{proof}
If a vertex $z$ connects to two edges 
with the same color, then $z$ must align to this color (by definition of the boundary nucleation).
The same holds if $z$ only connects to a single point.
It is then easy to see that any path in the graph $G$ that does not contain any relevant separation points  
is uniformly colored,
and as consequence, in our coloring algorithm, irrelevant separation points can be ``skipped'' to deduce the color of the vertices $\tilde V$ (see  Figs \ref{config5}-\ref{config666} for a concrete example). This will be achieved in Section \ref{coloring}.

\end{proof}



Significantly, 
even if the graph induced by the Brownian net paths starting from
a point
is infinite, its reduced version is finite (see Fig. \ref{reduced-graph} below).
This will allow
us to apply our coloring
algorithm
at the continuum level and 
to define the CVMP
by
imitating the duality 
relation described in Section \ref{coloring-algo}.. 

\subsection{Perspectives}
\label{perspective}

The discrete VMP's considered in Section \ref{cv-nn} and Theorem \ref{teo1}
are dual to a system of coalescing-branching nearest neighbor random walks,
that are known to converge (under proper rescaling)
to the Brownian net with killing \cite{NRS15}.

Let us now   reconsider the discrete VMP alluded to    in Section \ref{decomposition-1}, 
and whose transition probabilities are
defined in (\ref{voter-eps}). As for the nearest neighbor case, at least when when $f^\eps_{x,\eta}=f_{x,\eta}$, it is not hard to see that such systems are dual 
to branching-coalescing random walks. However, the transition probabilities of such walks 
are more complex: (1) walks move according to the transition kernel $K$  as described in (\ref{voter-f}); (2) they are killed with probability $\kappa_\eps$; and finally (3) they branch with probability $b_\eps$, and 
upon branching, the location of the $N$ new particles are 
chosen according $Q$ (shifted by the current location of the particle under consideration).

It is plausible that (under some restrictions on the moments of the kernels $K$ and $Q$) such systems of particles also converge 
to the Brownian net with killing. If so, the limiting object alluded to in Theorem \ref{teo1} -- and whose construction will be carried in Section \ref{coloring} --
should also be the scaling limit of a large class of discrete VMP's. However, it remains unclear what should be the limiting 
branching and killing mechanism at the continuum. Both the limit of generalized branching-coalescing random walks, and their relation 
to the CVMP are challenging questions, and would certainly require    techniques    that go beyond the scope of this paper.

Finally, it would be natural to also consider the continuous time analog of the models alluded above (as in \cite{CDP11} for higher dimensions). 
In particular, some recent work by Etheridge, Freeman and  Straulino \cite{EFS15} where they considered the paths generated by the genealogy of a spatial Fleming-Viot process in one dimension could be relevant to make progress in this direction.  

\subsection{Outline of the rest of the paper}

The rest of the paper is organized as follows.
In  Section \ref{BNK}, we recall the definition of the scaling limit of the BCK:
the Brownian net with killing introduced in \cite{NRS15},
which can be loosely described as an 
infinite family of one dimensional coalescing-branching-killing Brownian motions.
In Section \ref{coloring}, this 
is used to construct the continuum 
voter model perturbation (CVMP) by mimicking at the continuum level
the discrete 
duality relation of Section \ref{coloring-algo}. Properties
(1)-(2) of Theorem
\ref{teo1} will be shown in Proposition \ref{prop1-3}.
Finally, we show the convergence result of Theorem \ref{teo1} (Property (3))
in Section \ref{Invariance::Principle}.

\section{The Brownian net with killing}
\label{BNK}

In Proposition \ref{pottsdual},
the nearest-neighbors discrete voter model perturbations
were seen to be dual to an infinite family
of $BCK$ random walks. Hence, 
before defining a continuous analog
of the voter model perturbation,
it is natural to start with the construction
of the BCK scaling limit : the Brownian net with killing (as introduced in \cite{NRS15}),
denoted by $\Net^{b,\kappa}$ --- where $b$ and $\kappa$
are two non-negative numbers playing a role analogous to
the branching and killing parameters in the discrete setting.

\subsection{The space $({\cal H}, d_{\cal H})$}
\label{the-space}

As in \cite{FINR04}, we will define the Brownian net with killing 
as a random compact set of paths. In this section, we
briefly 
outline the construction of the space of compact sets.
For more details, the interested reader may refer 
to \cite{FINR04} and \cite{NRS15}.

First define $(\bar\R^2,\rho)$ to be 
the compactification of $\R^2$
with
$$
\rho((x_1,t_1),(x_2,t_2)) \ = \ | \phi(x_1,t_1) - \phi(x_2,t_2)| \vee |\psi(t_1)-\psi(t_2)|,
$$
where
$$
\phi(x,t) = \frac{\tanh(x)}{1+|t|}  \ \ \ \mbox{and}  \ \ \  \psi(t) = \tanh(t).
$$
In particular, we note that the mapping
$(x,t)\rightarrow(\phi(x,t),\psi(t))$
maps $\bar\R^2$ onto a compact subset of $\R^2$.

Next, let $C[t_0,t_1]$
denote the set of continuous functions from 
$[t_0,t_1]$ to $[-\infty,+\infty]$. From there, we define the set 
of continuous paths in $\R^2$ (with a prescribed
starting and ending point) as
$$
\Pi := \cup_{t_0\leq t_1} C[t_0,t_1]\times \{t_0\}\times\{t_1\}.
$$
Finally, 
we equip this set of paths with a metric
$d$, defined as the
maximum of the sup norm of the distance between
two paths, the distance between their respective starting points and 
the distance between their ending
points. (In particular, when no killing occurs,
as in the forward Brownian web, the ending
point of each path is $\{\infty\}$).
More precisely,
if for any path $\pi$,
we denote by $\sigma_\pi$
the starting time of $\pi$
and by $e_\pi$
its ending time,
we have
\beqnn\label{diistance}
d(\pi_1,\pi_2) \ = \
 | \psi(\sigma_{\pi_1}) - \psi(\sigma_{\pi_2})  | 
\vee
| \psi(e_{\pi_1}) - \psi(e_{\pi_2})  | 
\vee 
\max_{t} \  \rho((\bar\pi_1(t),t),(\bar\pi_2(t),t)),
\eeqnn
where $\bar \pi$ is the extension of $\pi$
into a path from $-\infty$ to $+\infty$
by setting
$$
\bar \pi(t) = \left\{ \begin{array}{c} 
\pi(\sigma_\pi) \ \ \mbox{for $t<\sigma_\pi$},  \\
\pi(e_\pi) \ \ \mbox{for $t>e_\pi$.}  \end{array} \right.
$$
  Finally, let ${\cal H}$ denote the set of compact subsets of $(\Pi,d )$  and let $d_{\cal H}$
denote
the
Hausdorff
metric,
\beqnn\label{dh}
d_{\cal H}(K_1,K_2) = \max_{\pi_1\in K_1} \min_{\pi_2\in K_2} d (\pi_1,\pi_2) \vee  
\max_{\pi_2\in K_2} \min_{\pi_1 \in K_1} d (\pi_1,\pi_2).  
\eeqnn   
In \cite{FINR04}, it is proved that $({\cal H},d_{\cal H})$ is Polish. In the following, 
we will construct the Brownian net with killing as a random element of this space.

\subsection{The Brownian web }
\label{Forward_Web}
As mentioned in the introduction, 
the Brownian web (\cite{TW98} \cite{FINR04} \cite{SSS15}) is the scaling limit of the discrete web
under diffusive space-time scaling
and is defined as an element
of $({\cal H},d_{\cal H})$.
The next theorem, taken from~\cite{FINR04}, gives some of the key
properties of the BW.

          \bteo
          \label{teo:char}
          There is an \( ({\cal H},{\cal F}_{{\cal H}}) \)-valued random variable
          \(
          {\W} \)
          whose distribution is uniquely determined by the following three
          properties.
          \begin{itemize}
                    \item[(o)]  from any deterministic point \( (x,t) \) in
          $\r^{2}$,
                    there is almost surely a unique path \( {B}_{(x,t)} \)
          starting
                    from \( (x,t) \).

                    \item[(i)]  for any deterministic, dense countable subset
        \(
          {\cal
                    D} \) of \( \r^{2} \), almost surely, \( {\W} \) is the
          closure in
                    \( (\Pi, d) \) of \( \{ {B}_{(x,t)}: (x,t)\in
                    {\cal D} \}. \)
  \item[(ii)]  for any deterministic  $n$ and \((x_{1}, t_{1}),
        \ldots,
                    (x_{n}, t_{n}) \), the joint distribution of \(
                    {B}_{(x_{1},t_{1})}, \ldots, {B}_{(x_{n},t_{n})} \) is that
                    of coalescing Brownian motions from those starting points (with unit diffusion
        constant).

          \end{itemize}
          \eteo

Note that (i) provides a practical construction of the Brownian web. For $\cal D$ as defined above,  construct coalescing Brownian motion paths starting from $\cal D$. This defines a {\it skeleton} for the Brownian web that is denoted by $\W({\cal D})$. $\W$ is simply defined as the closure of this precompact set of paths.

\subsection{The standard Brownian net} 
\label{standard-net}
\subsubsection{Idea of the construction}

When $\kappa =0$,  Sun and Swart give a construction of the 
Brownian net which is based on the 
construction of two coupled drifted 
Brownian webs $(\Wl,\Wr)$ that interact in a sticky way. For instance,
for every point $(x,t)$, there is a unique pair $(l,r)\in(\Wl,\Wr)$ starting from $(x,t)$, with $l\leq r$, and where
where $l$ is distributed as the path of a Brownian motion with drift $-b$, and $r$
is a Brownian motion with drift $+b$. The difference between the two paths is a Brownian motion
with sticky reflection at $0$
(see \cite{SS07} and a recent review paper \cite{SSS15} for more details).

The Brownian net $\Net^b$
is then constructed by
concatenating 
paths of the right and left webs. 
More precisely, given two paths $\pi_1, \pi_2\in\Pi$, let $\sigma_{\pi_1}$ and $\sigma_{\pi_2}$ be the starting times of those paths.
For $t >
\sigma_{\pi_1} \vee \sigma_{\pi_2}$ (note the strict inequality),
 $t$ is called
an intersection time of $\pi_1$ and $\pi_2$ if $\pi_1(t)=\pi_2(t)$. By
hopping from $\pi_1$ to $\pi_2$, we mean the construction of a new
path by concatenating together the piece of $\pi_1$ before and the
piece of $\pi_2$ after an intersection time. Given the left-right
Brownian web $(\Wl, \Wr)$, let $H(\Wl\cup\Wr)$ denote the set of paths
constructed by hopping a finite number of times between paths in
$\Wl\bigcup\Wr$.
$\Net^b$
is then defined as the closure
of $H(\Wl\bigcup \Wr)$. After taking this limit,
$\Wl$ and $\Wr$ delimit the net paths from the left and from the right, in the sense that 
paths in $\Net^b$ can not cross any element of $\Wl$ (resp., $\Wr$) from the left (resp., from the right).

Finally, we cite from \cite{SS07} a crucial property of the Brownian net that will be useful for the rest of this paper.

\bprop\label{bc-set}
Let $S<T$ and define the branching-coalescing point set
$$
\xi^S(T) \ := \ \{x\in\R \ : \ \exists \pi\in \Net^{b} \ \mbox{s.t.} \ \sigma_{\pi}=S, \ \pi(T)=x \}.
$$
For almost every realization of the Brownian net, the set $\xi^S(T)$ is locally finite.
\eprop

\subsubsection{Special points of the standard Brownian net}
\label{special-force}

In this section ,
we recall the classification of the special points
of the Brownian net, as described in \cite{SSS08a}. 
As in the Brownian web,
the classification of special points will be based on 
the local geometry of the Brownian net.
Of special interest
to us will be the points with a deterministic
time coordinate.

\bdf[Equivalent Ingoing and Outgoing Paths]
\label{eq:in-out}
Two paths $\pi,\pi'\in(\Pi,d)$
are said to be equivalent paths entering a point $z\in\R^2$, or in short 
$\pi \sim^{z}_{in} \pi' $, iff there exists a sequence $\{t_n\}$
converging to $t$ such that
$t_n<t$ and
$
\pi(t_n) = \pi'(t_n)
$
for every $n$.
Equivalent paths exiting a point $z$, 
denoted by
$\pi \sim^{z}_{out} \pi'$, 
are defined 
analogously by finding a sequence $t_n>t$
with $\{t_n\}$ converging to $t$ and $\pi(t_n)=\pi'(t_n)$. 
\edf

Despite the notation,
these are not in general equivalence relations on the spaces of all
paths entering resp. leaving a point.
However, in \cite{SSS08a},
it is shown that
that
if $(\Wl,\Wr)$ is a left-right Brownian web, then a.s. for all
$z\in\R^2$, the relations $\sim_{in}^z$
and
$\sim_{out}^z$
actually
define equivalence relations on the set of paths in $\Wl \cup \Wr$
entering (resp., leaving) $z$, and the equivalence classes of paths in $\Wl \cup \Wr$
entering (resp., leaving)
$z$ are naturally ordered from left to right. Moreover,
the authors 
gave a complete classification of points
$z\in\R^2$ according to the structure of the equivalence classes in $\Wl \cup \Wr$ entering (resp., leaving)
$z$.

\begin{figure}[h!]
\centering
\includegraphics[scale=.2]{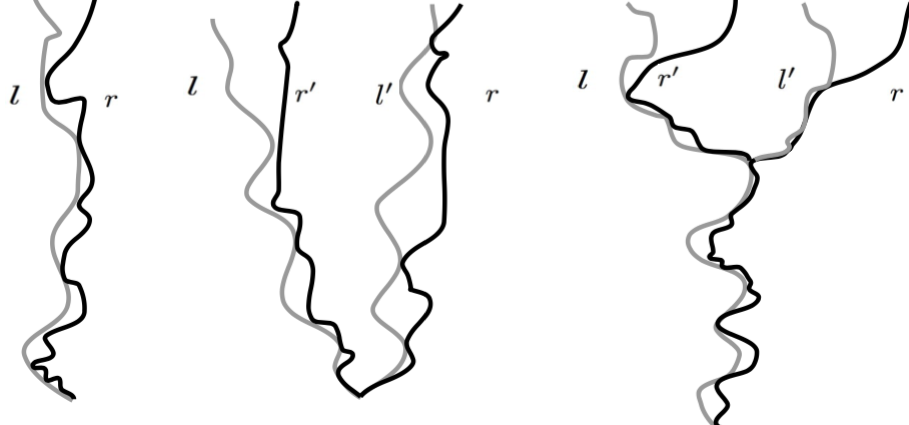}
\caption{Structure of the left (grey) and right (black) webs $\Wl,\Wr$ at $(o,p)$, $(o,pp)$ and $(p,pp)$ points
of the net $\Net^b$. In particular, for $(o,pp)$ and $(p,pp)$ points, there are two equivalent outgoing pairs $(l,r')$ and $(l',r)$ (ordered from left to right) exiting the point $z$.}
\label{1-2}
\end{figure}

In general, such an equivalence class may be of three types. If it contains only
paths in $\Wl$ then we say it is of type $l$, 
if it contains only paths in $\Wr$ then we say it is of
type $r$, and if it contains both paths in $\Wl$ and $\Wr$
then we say it is of type $p$, standing for
pair. To denote the type of a point $z\in\R^2$ in a Brownian net $\Net^b$, 
we list the incoming
equivalence classes in $\Wl \cup \Wr$
from left to right (see Fig. \ref{1-2}), and then, separated by a comma, the outgoing equivalence classes of paths from left to right. 
For example, a point of type $(p,pp)$ (as in  Fig. \ref{1-2}) has exactly one equivalent class of paths entering
the point, and two distinct equivalent classes exiting the point. Further, all of those classes 
contain  at least a path in $\Wr$ and path in $\Wl$. For a point of type $(o,p)$,
there is no path entering the point, and one equivalent class exiting the point. 
We cite the following results from 
Theorems 1.7 and 1.12 in \cite{SSS08a}.

\bprop[Geometry of the Net at Deterministic Times]
\label{special-net}
Let $\Net^b$ be a Brownian net and let $(\Wl,\Wr)$ be the left-right Brownian web associated with $\Net^b$.
Let $t$ be a deterministic time.
\ben
\ii For every deterministic point $x$, the point $(x,t)$ is of type $(o,p)$.
\ii Each point $(x,t)$ (with $x$ deterministic or random)
is either of type
(o, p), (p, p) or (o, pp), 
and all of these types occur.
\ii Every path $\pi\in\Net^b$ starting from the line $\R\times\{t\}$,
is squeezed between an equivalent pair of right-most and left-most paths,
i.e.
there exists $l\in \Wl$ and $r\in \Wr$ so that
$l\sim_{out}^z r$ such that 
$l \leq \pi \leq r$ on $[t,\infty)$.
\ii  Any point $(x,t)$ entered by a path $\pi\in\Net^b$ with $\sigma_\pi<t$
is of type $(p,p)$. Moreover, $\pi$
is squeezed between an equivalent pair of right-most and left-most paths, i.e.
there exist
 $l\in\Wl, r \in \Wr$ with $r\sim_{in}^z l$
and $\eps > 0$ and such that 
$l \leq \pi \leq r$ on $[t - \eps,t]$.
\een
\eprop

In the Brownian net, a separation point
(or branching point)
will refer to a point $z=(x,t)$ such that
there exist two paths
$\pi_r$ and $\pi_l\in\Net^{b}$
with $\sigma_{\pi_l},\sigma_{\pi_r}<t$ 
so that (1) $\pi_l(t)=\pi_r(t)=x$ and (2) there exists 
$\eps>0$
so that $\pi_l(s)<\pi_r(s)$ for every $s\in(t,t+\eps]$.

\bprop[Geometry of the Net at a Separation Point]
\label{geo-net-sp}
If $z=(x,t)$ is a separation point in the standard Brownian net
$\Net^{b}$ then
\ben
\ii $z$ is of the type (p,pp).
\ii Every path $\pi$ starting from $z$  is squeezed between some
pair of equivalent right-most and left-most paths ---
i.e.
there exists $l\in \Wl$ and $r\in \Wr$ (which will depend on $\pi$) with $l\sim_{out}^z r$ such that 
$l \leq \pi \leq r$ on $[t,\infty)$.
\ii Every path $\pi$ entering the point $z$ is squeezed between  a pair of equivalent 
right-most and left-most paths ---
i.e.,
there exist
$l\in\Wl, r \in \Wr$, $l\sim_{in}^z r $,
entering
the point $z$
and $\eps > 0$ such that 
$l \leq \pi \leq r$ on $[t - \eps,t]$. 
\een
\eprop

\subsection{The Brownian net with killing}
\label{onenet}

{\bf Construction.} Having the standard Brownian net at hand,
we now need to turn on the killing
mechanism. To do that, we recall the definition of
the time length measure of the Brownian net as introduced in \cite{NRS15}.

\bdf[Time Length Measure of the Brownian Net]
\label{tlm-net}
For almost every realization of the standard Brownian net,
$\cL$
is the unique $\sigma$-finite measure such that for every 
Borel set $E\subset\R^2$,
\be
\cL(E) = \int_\R |\{x : (x,t)\in E\cap(p,p)\}| dt \label{ex-cl}.
\ee
where $(p,p)$ refers to the set of $(p,p)$ points for the net $\Net^{b}$.
\edf

Given a realization of the standard Brownian net $\Net^b$,
we define the set of the killing marks $\Mzero^\kappa$ as
a Poisson point process with intensity measure
$\kappa \cL$. Finally, we define
the Brownian net with killing as the union
of 
(1) all the paths
$\pi\in\Net^b$ killed
at 
\be\label{str-e}
e_{\pi}:=\inf\{t>\sigma_{\pi} : (\pi(t),t)\in\M^\kappa \}
\ee
and (2) for every $z\in\R^2$, the trivial path whose starting point
and ending point coincide with $z$. In the following,
this construction
will be denoted as $\Net^{b,\kappa}$.

Let $E\subseteq \R^2$. We will denote by $\Net^{b,\kappa}(E)$ the subset of paths in $\Net^{b,\kappa}(E)$
with starting points in $E$. Finally, ${\cal M}^\kappa(S)$ will denote the set of killing points that are attained by $\pi\in\Net^{b,\kappa}$
with $\sigma_\pi= S$. We cite the following result from \cite{NRS15} (Proposition 2.14 therein).

\bprop\label{p214}
For every $S\in\R$, ${\cal M}^\kappa(S)$ is locally finite a.s..
\eprop

{\bf Coupling between the net and the net with killing.}
Our construction of $\Net^{b,\kappa}$ induces a natural coupling between $\Net^b$ and $\Net^{b,\kappa}$, and in the following,
$\Net^b$ will refer to the net used to construct the net with killing $\Net^{b,\kappa}$.
Let ${\cal U}^{b_\eps,\kappa_\eps}$ denote the discrete BCK (or discrete net with killing)  with branching and killing
parameters $(b_\eps,\kappa_\eps)$, and $\cU^{b_\eps}:=\cU^{b_\eps,0}$. Like in the continuum, there is a natural 
and analogous coupling between the discrete objects $\cU^b$ and $\cU^{b,\kappa}$, and again,
 $\cU^b$ will always refer to the discrete net coupled with $\cU^{b,\kappa}$. (This coupling is obtained by starting from $\cU^b$ and 
 by removing all the outgoing arrows independently at every vertex with probability $\kappa$).

{\bf Convergence.} In \cite{NRS15}, we show that the killed Brownian net is the scaling limit of 
the BCK system for small values of the branching and killing parameters. In the following,
for every element $U\in{\cal H}$, $S_\eps(U)$ is the set paths obtained 
after scaling the $x$-axis by $\eps$ and the $t$-axis by $\eps^2$.

\bteo(Invariance Principle)
\label{onewebteo}
Let $b,\kappa \geq0$ and
let $\{b_n\}_{n>0}$ and $\{\kappa_n\}_{n>0}$
two sequences of non-negative numbers such that $b_n+\kappa_n\leq 1$ and
such that 
$\lim_{n\to\infty} b_n \eps_n^{-1} = b$
and 
$\lim_{n\to\infty} \kappa_n \eps_n^{-2} = \kappa$.
Then, as $n\to\infty$,
$$
(S_{\eps_n}(\cU^{b_n}), S_{\eps_n}( {\cal U}^{b_{n},\kappa_{n}} ))\ \raw \  (\Net^{b}, \Net^{b,\kappa})
\ \ \ \ \text{in law.}$$ 
\eteo
\begin{proof}
In \cite{NRS15}, we  showed that $S_{\eps_n}( {\cal U}^{b_n,\kappa_{n}} )$ converges to $\Net^{b,\kappa}$ in law. However,
a closer look at the proof shows that
the result was shown by constructing a coupling 
between 
$
\{(S_{\eps_n}(\cU^{b_n}), S_{\eps_n}( {\cal U}^{b_{n},\kappa_{n}} ))\}_{n>0}$
and
$(\Net^{b}, \Net^{b,\kappa})$
such that  $S_{\eps_n}( {\cal U}^{b_{n}} )$
converges to $\Net^b$ a.s., and such that under this coupling,
the convergence of the second discrete coordinate to its continuum counterpart holds in probability. 
This allows to extend the invariance principle derived in \cite{NRS15} to Theorem \ref{onewebteo}.
\end{proof}

\section{The continuum voter model perturbation (CVMP)}
\label{coloring}

As described in the introduction, 
we aim at constructing the CVMP as a map
from $\R\times\R^{+}$ onto finite subsets of 
$\{1, . . . , q\}$.
This ``quasi-coloring" will be defined in terms
of the 
duality relation given in Section \ref{coloring-algo},
i.e,  
we see 
the CVMP
as dual to the 
Brownian net with killing. The key
observation is that even if the 
continuous graph
starting from a ``generic" point $z$
is infinite (what we mean by generic
will be discussed below in more details),
the reduced version
of this graph is locally finite. (Recall 
that
a reduced graph is constructed 
after removal of
all the non-relevant separation points in the original
graph --- see Definition \ref{reduced-graph-df}).

We now give an outline of the present section. 
In Section \ref{section:relevant}, 
we introduce the notion of relevant separation
points and of reduced graph for the Brownian net with killling and prove
some of their properties.
Those results are simple extensions of the results already known in the 
case $\kappa =0$ \cite{SSS08a, ScheSS08b}. 
In Section \ref{f:CSPM},
we construct the CVMP. 
In Section \ref{properties}, we  prove 
properties 1 and 2 listed in Theorem \ref{teo1}.

\subsection{The reduced graph}
\label{section:relevant}

We first define a notion of graph isomorphism adapted to 
our problem.
\bdf\label{asy-iso}
Let $G=(V,E)$ and $G'=(V',E')$ be two finite 
directed acyclic graphs. 
We will say that $G$ and 
$G'$ are isomorphic if there exists a 
bijection $\psi$ from $V$ to $V'$ such that 
for every $z_1,z_2\in V$
\beqn
\label{connectivity}
(z_1,z_2) \in E \ \ \mbox{iff}  \ (\psi(z_1),\psi(z_2))\in E'.
\eeqn
\edf
In other words, two graphs will be
isomorphic   if they   
coincide modulo 
some relabelling of the nodes.

Let us consider the set of all finite directed acyclic
graphs whose vertices are labelled by points 
in $\R^2$. The graph isomorphism property is an equivalence relation partitioning this set of graphs into equivalence classes. In the following, we will denote by $(\G,\cal G)$ the set of all equivalence classes endowed with its natural  $\sigma$ field.  In this section, we  construct a natural 
reduced graph generated by the Brownian net with killing  
paths
(starting from a point $z$ until a certain time horizon $T$)  as an element 
of $(\G,\cal G)$. 

\bigskip

As in the standard Brownian net, in the Brownian net with killing, a separation point
(or branching point)
will refer to a point $z=(x,t)$ such that
there exist two paths
$\pi_r$ and $\pi_l\in\Net^{b,\kappa}$
with $\sigma_{\pi_l},\sigma_{\pi_r}<t$ and $e_{\pi_l},e_{\pi_r}>t$
so that (1) $\pi_l(t)=\pi_r(t)=x$ and (2) there exists 
$\eps>0$
so that $\pi_l(s)<\pi_r(s)$ for every $s\in(t,\min(t+\eps,e_r,e_l)]$.
We note that in the coupled pair $(\Net^{b}, \Net^{b,\kappa})$ the separation points of the killed Brownian
net and the ones of the standard Brownian net coincide a.s..

\bdf[Relevant Separation Point] 
A point $z=(x,t)\in\R^2$ with $t \in (S,T)$    is
said to be an $(S,T)$-relevant
separation point of the Brownian net $\Net^{b,\kappa}$, 
if and only if
there exist two paths $\pi_1,\pi_2\in{\cal N}^{b,\kappa}$
with starting time $\sigma_{\pi_1}=\sigma_{\pi_2}=S$
and ending times $e_{\pi_1},e_{\pi_2}>t$, 
and such that
\be
\pi_1(t) \ = \ \pi_2(t) \ \ \ \text{but} \ \ \pi_1(u) \neq \pi_2(u) \ \ \text{for} \ \ u \in \ \left( t, \ \inf(T,  e_1, e_2) \right).
\ee  
\edf

In the following, we will denote
by $R^\kappa({S,T})$ the set 
of $(S,T)$-relevant separation points in the 
Brownian net with killing $\Net^{b,\kappa}$ (with no $b$
superscript in $R^\kappa(S,T)$ to ease the notation). Finally,
$R({S,T})$ will denote 
the same quantity when $\kappa=0$ (Proposition 6.1. in \cite{ScheSS08b}).

\begin{prop}\label{finite-SP}
For every $S<T$, the set $R^\kappa (S,T)$ is locally finite a.s.
\end{prop}

\begin{proof}
As already mentioned, this property for $\kappa =0$ has been established in \cite{ScheSS08b}.
It remains to extend the property for $\kappa \neq0$. 
Let $E\subset\R^2$ be a bounded set
and let us show that $R^{\kappa}(S,T)\cap E$ is finite a.s.. 

Let $L>0$ and define $B_L$
to be the set of realizations such that 
$$
\{\pi\in\Net^b \ :  \ \sigma_\pi=S, \ trace(\pi)\cap E \neq \emptyset \} \ \subset \ \{\pi\in\Net^b \ :  \ \sigma_\pi=S, \  \forall t\in[S,T], \  -L<\pi(t)<L  \}.
$$
In particular, the equicontinuity of the paths in $\Net^b$
implies that $\P(B_L)$ goes to $1$ as $L\to\infty$.

Next, we inductively define $\theta^0\equiv\theta^{0,L}=S$ and a
sequence
of killing times $\{\theta^i\equiv\theta^{i,L}\}_{i\geq1}$
as follows: 
\beq\label{thetai}
\forall i \geq 0, \  \theta^{i+1} \ = \ \inf\{t > \theta^i : \ \exists (x,t) \ \in \ \M^\kappa \left(\theta^i\right)  \ \text{and}  \ |x|<L\}.
\eeq
where $\M^\kappa (t)$ is the set of killing points attained by the
set of paths $\pi\in\Net^{b,\kappa}$ with $\sigma_{\pi}= t$. 
From the very definitions of the $\theta^i$'s, it is not hard to see that
\be\label{qe1}
\mbox{conditionaly on $B_L$,   $R^{\kappa}(S,T)\cap E \ \subset \ \cup_{i\geq0} R(\theta^{i}\wedge T, \ \theta^{i+1}\wedge T)$}. 
\ee
(See also Fig. \ref{desin}.)
\begin{figure}[h!]
\centering
\includegraphics[scale=.25]{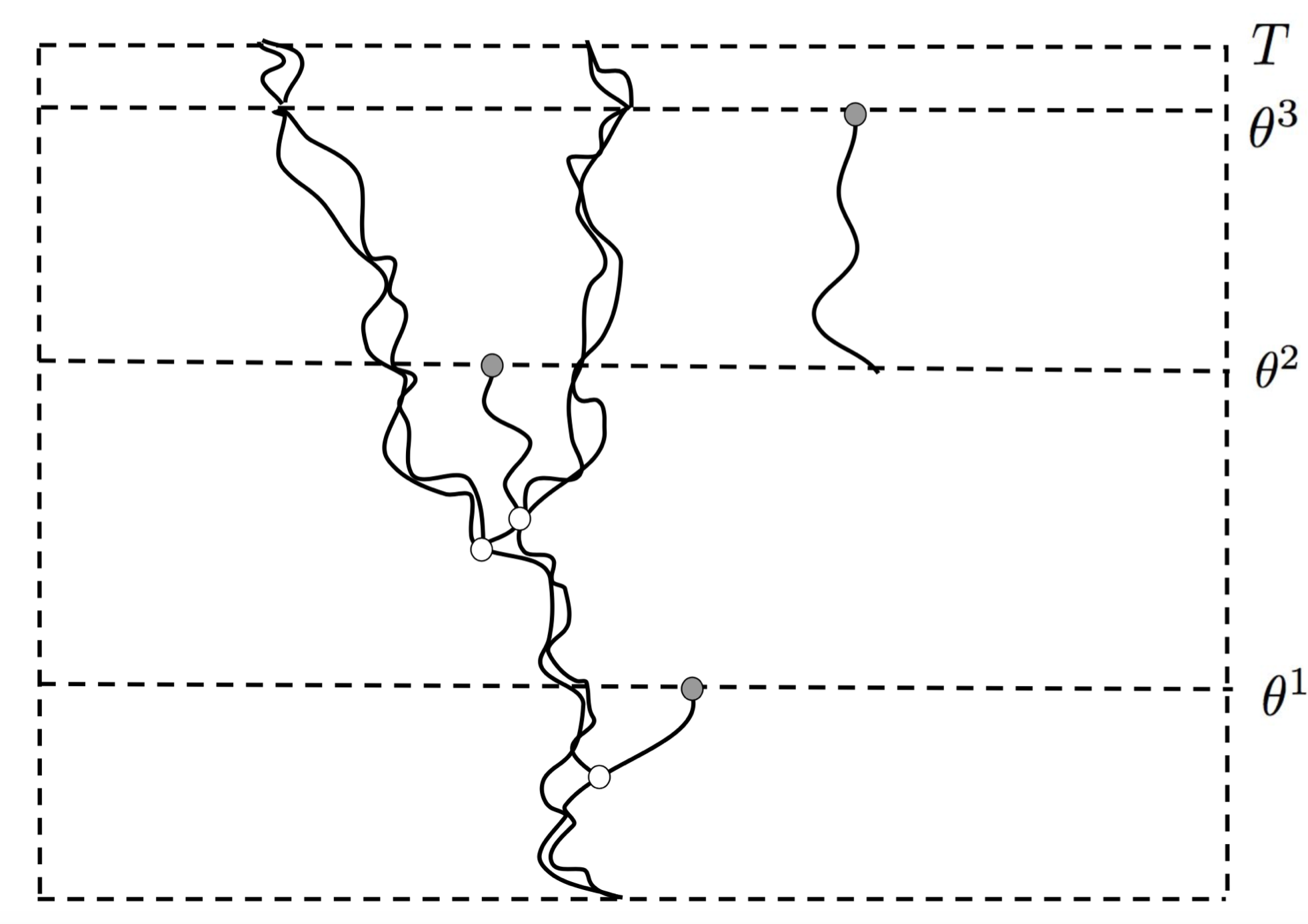}
\caption{
Grey circles represent points of the form $(x,\theta^i)\in{\cal M}^\kappa$.
The $(S,T)$-relevant separation points (white circles) hit by a path starting from a given point in $(-L,L)\times\{S\}$ 
and staying in the box $(-L,L)\times[S,T]$ 
belong to
$\cup_{i\geq0} R(\theta^{i}\wedge T, \ \theta^{i+1}\wedge T)$.
}
\label{desin}
\end{figure} 
\bigskip

We now claim that
\be\label{qe2}
\mbox{for every $i$, $R({\theta^{i}\wedge T, \theta^{i+1}\wedge T})$ is locally finite.}
\ee
In order to prove (\ref{qe2}), it is enough to prove that 
(a) when $\theta^i<T$, $R({\theta^i,T})$ is locally finite, and (b) $R({\theta^i,\theta^{i+1}})$ is also locally finite.
For property (a), we first note that $R(U,T)$ is a.s. locally finite
for every deterministic $U$ with $U<T$,  implying
that
the set of times $U$
for which $R(U,T)$ is locally finite 
has full Lebesgue measure.
Combining this with the definition of the random intensity 
measure used in our Poisson marking
(see Definition \ref{tlm-net}) and using Fubini's theorem  
show that property (a) must hold.

For Property (b), the strong Markov Property 
and stationarity in time 
show that  it is enough to prove 
the result for the special value $i=0$. 
For any deterministic times $U$,
we know that
$R(S,U)$ is locally finite a.s.. Reasoning as in the previous paragraph
completes the proof of (b), and thus (\ref{qe2}).

We are now ready to complete the proof of Proposition \ref{finite-SP}.
First of all, $\M^\kappa (S)$ is locally finite by Proposition \ref{p214}. Thus
we get 
that $\theta^1-\theta^0>0$ a.s.. Furthermore, using
the strong Markov property of the standard Brownian net
and its stationarity in time,
the sequence
$(\theta^{i+1} - \theta^{i})$ 
is a sequence
of i.i.d. strictly positive random variables.
As a consequence,
for every
realization
of $\Net^{b,\kappa}$,
the sequence $\{R({\theta^i\wedge T, \theta^{i+1}\wedge T})\}_{i\geq0}$
is empty after some finite value of $i$. 
Our proposition
then immediately follows from (\ref{qe1}) and (\ref{qe2}) and the fact that $\P(B_L)$
goes to $1$ as $L$ goes to $\infty$.

\end{proof}

Define
\begin{eqnarray}
L^\kappa ({S,T}) & = & \{(x,T) \ : \  
\exists \pi\in\Net^{b,\kappa}, \ \mbox{s.t.} \ \sigma_{\pi}=S, \ e_\pi\geq T, \ \pi(T)=x \}, \nonumber \\
K^\kappa ({S,T}) & = & {\cal M}^\kappa (S) \bigcap \R\times[S,T]. \label{def-kl}
\end{eqnarray}

\bdf[Continuous Reduced Graph]\label{def:reduced-graph}
Let $S<T$ be deterministic or random. Let $z=(x,S)$ and let $V^z(T)$
be the union of the singleton $z$ with 
the subset of vertices $L^\kappa ({S,T}) \bigcup K^\kappa ({S,T}) \ \bigcup \ R^\kappa ({S,T})$ reached by a path in $\Net^{b,\kappa}$  starting 
from
$z$.
\ben
\ii A point $z=(x,S)$ will be said to be a $T$-finite point iff $|V^z(T)|<\infty$.
\ii For every $T$-finite point $z=(x,S)$,
the reduced graph starting from $z$ on 
the time interval $[S,T]$ --- denoted by $G^z(T)$ --- 
is the random directed graph
whose vertices are given by $V^z(T)$ and 
whose edges are generated by the 
Brownian net with killing $\Net^{b,\kappa}$:
two points $z'=(x',t'),z''=(x'',t'')\in V^z(T)$, will
be connected by a (directed) edge
iff $t'<t''$  and there exists
$\pi \in \Net^{b,\kappa }$
such that 
$\sigma_\pi = t', \pi(t')=x',\pi(t'')=x''$ and 
$$
\forall u\in(t',t''), \ (\pi(u),u) \notin L^\kappa ({S,T}) \bigcup K^\kappa ({S,T}) \ \bigcup \ R^\kappa ({S,T}). 
$$
\een
\edf

\bprop\label{finite-3.3}
Let $S<T$ be two deterministic times. Almost surely for every realization of the Brownian net with killing,
every point (deterministic or random) in $\R\times \{S\}$ is $T$-finite. 
\eprop

\begin{proof}
By equicontinuity of the paths in $\Net^{b,\kappa}$, we only need to show that 
$$
L^\kappa(S,T) \bigcup K^\kappa({S,T}) \ \bigcup \ R^\kappa (S,T)
$$
is locally finite a.s.. First, $R^\kappa (S,T)$ is locally finite using 
Proposition \ref{finite-SP}, and  ${\cal M}^\kappa (S)$ is also locally finite by Proposition \ref{p214}. Finally, when $\kappa =0$, 
$L^0({S,T})$ is locally finite by Proposition \ref{bc-set}. Since $L^\kappa (S,T)$ is stochastically
dominated by $L^0({S,T})$, this ends the proof of Proposition \ref{finite-3.3}.

\end{proof}

\subsection{Construction of the continuum voter model perturbation}
\label{f:CSPM}
In this section, we construct the CVMP alluded to in 
Theorem \ref{teo1}. It will depend on following quantities
(1) the branching and killing parameters $b,\kappa $,
(2) the boundary nucleation mechanism $\{g_{i,j}\}_{i,j\leq q}$,
(3) two probability distributions $p$ and $\lambda$ on $\{1,\cdots,q\}$,
where $p$ is the bulk nucleation mechanism and $\lambda$
is the analog of the one-dimensional marginal of the initial
distribution for the discrete voter model perturbation.

Recall that
in order to determine the color 
of a vertex $(x,t)$ in
our discrete
voter model perturbation
$\{\theta^{\eps}(x,t)\}_{(x,t)\in \Z^2_{odd}}$,
we start 
with a random coloring of the 
leaves of the graph induced by 
the backward BCK, starting from $(x,t)$ as a root.
The color of $(x,t)$
is then determined
by applying the 
algorithm described in Section \ref{coloring-algo} to the reduced graph rooted at $(x,t)$ (see Proposition \ref{reduced-graph-df2}).

\bigskip

{\bf The backward Brownian net with killing.} As in the discrete case, the backward Brownian net with killing $\hat \Net^{b,\kappa}$
is obtained from the Brownian net with killing by a reflection about the $x$-axis.
In the following, we will use identical notation 
for the backward Brownian net with killing, adding only a hat sign
to indicate that we are dealing with the backward object. For instance $\hat R^\kappa (S,T)$
will refer to the (S,T)--relevant separation points for the backward net between times $S$ and $T$.

\bigskip

{\bf Step 0.  Assignment of i.i.d. random variables.} 
In the continuum, 
we consider a realization 
of the backward Brownian net with killling $\hat \Net^{b,k}$
and consider
the random set of points
consisting of the union of
\begin{enumerate}
\ii all the separation points in $\hat \Net^{b,\kappa}$,
\ii all the killing points in $\hat \Net^{b,\kappa}$,
\ii the set $\{(x,0) \ : \ \exists \hat\pi \in \hat \Net^{b,\kappa} \ \text{s.t.} \ \sigma_{\hat \pi} > 0  \ \text{and} \ {\hat \pi}(0)=x \}$
\end{enumerate} 
Those three sets are countable 
for almost every realization of the Brownian net with killing for the following reasons:
\begin{enumerate}
\item First, for any separation point $z$, 
there exists $p<q\in\Q$ such that $z\in \hat R^\kappa_{p,q}$.
From Proposition \ref{finite-SP}, $\hat R^{\kappa}_{p,q}$ is locally finite a.s., which implies that 
the set of separation points is countable a.s.
\item Since, conditional on ${\cal N}^b$, the set of killing points is defined as a Poisson Point Process, this set is also countable a.s..
\item  Finally,
for the third set, we note that
 \beqnn
 \{(x,0) \ : \ \exists \hat\pi \in \hat \Net^{b,\kappa} \ \text{s.t.} \ \sigma_{\hat \pi} > 0  \ \text{and} \ {\hat \pi}(0)=x \} \ = \
 \cup_{s>0, s\in\Q}  \ \hat \xi^s(0) 
 \eeqnn
where $\hat \xi^s(0)$ is the backward version of the branching-coalescing point set as defined in Proposition \ref{bc-set}. 
Since  $\hat \xi^s(0)$ is locally finite for every deterministic $s$, the third item is also countable.
\end{enumerate}
Finally, since the random set consisting of the union of the sets 1, 2 and 3 above is a.s. countable,
given
a realization
of $\hat \Net^{b,\kappa}$,
i.i.d. uniformly distributed random
variables $U_z$ in $[0,1]$ can be assigned to each of the points $z$
in the three sets described above.

\bigskip

{\bf Step 1. Pre-coloring of the leaves.}
We consider the set of leaves, consisting of the killing points 
and the intersection points of ${\hat \Net}^{b,\kappa}$ with $\{t=0\}$. For each such point $z$, we deduce
the pre-coloring of $z$, denoted by $\bar\theta(z)$,  from the variable $U_z$ according to the following rule.
\begin{itemize}
\item If the leaf $z$ is an intersection point between a path from $\hat \Net^{b,\kappa}$ and $\{t=0\}$ (i.e. a point from set 3), 
$\bar \theta(z)$
is the unique integer in $\{1,\cdots,q\}$ such that   
$$
U_z \in [\sum_{i\leq \bar \theta(z)-1 }\lambda(i), \sum_{i\leq \bar \theta(z)} \lambda(i) )
$$
with the convention that $\lambda(0)=0$.
\item If the leaf $z$ is a killing point in $\hat \Net^{b,\kappa}$ (i.e. from set 2), $\bar \theta(z)$
is the unique integer in $\{1,\cdots,q\}$ such that 
$$
U_z\in [\sum_{i\leq \bar \theta(z)-1 }p(i), \sum_{i\leq \bar \theta(z)} p(i) )
$$
with the convention that $p(0)=0$.
\end{itemize}

{\bf Step 2 : Pre-coloring of the ``nice" points.} 
Given a point $(x,S)$ and a time $T>S$,
recall the definition of the reduced graph
$G^{(x,S)}(T)$ in
the forward Brownian net with killing $\Net^{b,\kappa}$ (see Definition \ref{def:reduced-graph}). Similarly, for any $S>0$,
one can define  a reduced graph $\hat G^{(x,S)}(0)$
whenever the point $(x,S)$ is $0$-finite
in the backward Brownian net with killing.

\bdf
A point is ``nice'' if it is a $0$-finite point
for the backward Brownian net ${\hat \Net}^{b,\kappa}$ . 
\edf

In particular, if we consider ${\cal D}$, a countable deterministic dense set of
$\R\times\R^+$, Proposition \ref{finite-3.3} (in its backward formulation)
ensures that  all the points in ${\cal D}$ are nice 
for almost every realization of the Brownian net with killing. 

\bdf
A graph is said to be simply rooted if the out-degree of the root is $1$.
\edf

For nice points with a simply rooted graph,
we assign the color determined 
by the coloring algorithm described in Section \ref{coloring-algo} using the family 
of distributions $\{g_{i,j}\}$ and the $U_z$'s at separation points  (as in (\ref{fghj})). 
For every such point,
this defines a color that we denote by $\bar \theta(z)$.

For graphs with more complicated root structure (i.e. with the root connected to more than one point),
we decompose 
the finite graph into simply rooted graphs. (For instance, if the root connects to two distinct points $z'$ and $z''$, the first component is
the subgraph whose set of vertices is  the union of the root with the vertices that can be accessed from $z'$; the second component
is defined analogously.) 
$\bar \theta(z)$ is then defined as the union
of the coloring induced by each of the simply rooted subgraphs in the decomposition.

\brm\label{rem1}
\label{n-colors}
By construction, $|\bar \theta(z)|$ is less or equal to the out-degree 
of the root.
\erm



{\bf Step 3 : Horizontal Limits.}
The coloring
of the half-plane $\R\times\R^{+}$ is generated by taking horizontal limits. Our procedure is based on the following
lemma.

\blem\label{more-1}
Almost surely for every realization of the backward Brownian net with killing, 
for every time $t$ (deterministic or not) the set $\{x \ : \ (x,t) \ \text{is nice}\}$
is dense in $\R$ . 
\elem
\begin{proof}
Let us consider a dense countable deterministic set ${\cal D}\subset\R\times\R^+$. By Proposition \ref{finite-3.3} in its backward formulation,
every point in ${\cal D}$ is nice a.s.. Furthermore, any point $z'$ belonging to
the trace of $\hat\Net^{b,k}({\cal D})$ (the backward Brownian net with killing paths starting from ${\cal D}$),
must also have a finite graph representation, 
since every relevant separation point    for $z'$
is also relevant for some $z\in {\cal D}$, and every leaf  for $z'$ is also a leaf for  $z$.    Hence, for every every
realization
of the Brownian net, for every $t$ (random or deterministic)
it is sufficient to prove that 
$\hat \Net^{b,\kappa}({\D})$ intersects the set $\R\times\{t\}$ at a dense set of points. 

For $\kappa=0$, the property directly follows from a property of the Brownian web: the so-called skeleton $\hat {\cal W}_r({\cal D})$ (and also $\hat {\cal W}_l({\cal D})$) 
intersect every  horizontal time $\R\times\{t\}$ at a dense set of point. (The contrary
would easily contradict the equicontinuity of the web paths.)
A fortiori, the net with no killing $\hat \Net^{b}$ enjoys the same property. Finally, this extends to the case $\kappa>0$, 
by noting that our Poisson construction of the killing points can not charge  any horizontal line $\R\times\{t\}$
with more than one point.

\end{proof}

For every point $(x,t)\in\R\times\R^+$ (including the nice points described previously),
we define the colors at point $(x,t)$ as
\beqn
\label{di-0}
\theta(x,t) \ = \ \{c\in\{1,\cdots,q\} \ : \ \exists \{(x_n,t)\} \ \text{``nice" s.t.} \
 \  x_n \rightarrow x \ \text{and } \ \forall n, \ c\in \bar \theta(x_n,t)\}
\eeqn 
From Lemma \ref{more-1}, $\theta(x,t)$ is well defined. Further,
for every nice point $z\in\R\times\R^+$, we must have
$\bar \theta(z)\subseteq\theta(z)$.

\brm
It is natural to ask whether one could define 
the CVMP by taking limits in any directions. It is not hard to show that 
for every deterministic $(x,t)$, and every color $c$,
there exists a sequence $(x,t_n)$ with $t_n\to t$
such that $\bar \theta(x,t_n)=c$. As a consequence, 
if we allowed for vertical limits, every deterministic point
will be trivially colored with $\{1,\cdots,q\}$. This property was established 
in the case where the boundary noise is absent (Theorem 1.1.(7) in \cite{FINR05})
\erm

\subsection{Properties of the CVMP}
\label{properties}

We will now prove some basic properties of our mapping (Theorem \ref{teo1}(1)--(2)).
In order to do so, we first elaborate on the properties of the reduced graph at 
deterministic times.

\subsubsection{Reduced graphs at deterministic times}

In this subsection, we derive some properties of the reduced graph at deterministic times. We 
formulate these results ``forward in time'' -- i.e., for the original Brownian net, and not its backward version. When needed (in the forthcoming sections),
these results will be used in 
their backward formulation.

Before stating the main proposition of this section (see Proposition \ref{reduced-graph-det-times} below), 
we start with some definitions.
Let $z_0$ and $z_1$ be two $T$-finite points. We will say 
that $G^{z_0}(T)$ and $G^{z_1}(T)$ only differ by their root iff
they are identical up to relabeling of their root, i.e., when 
the mapping $\psi$, defined by
$$
\forall z \in \ V^{z_0}(T), \ \ \psi(z) = \left\{ 
  \begin{array}{l l}
     z & \quad \text{if $z\neq z_0$ }\\
    z_1 & \quad \text{otherwise}\\
  \end{array} \right.,
$$
is a graph isomorphism from $G^{z_0}(T)$ to $G^{z_1}(T)$ (as in Definition \ref{asy-iso}).

\bigskip

Before stating our next result, we recall from Proposition \ref{special-net} that for every deterministic $S$, a.s. for every realization of the Brownian net with killing,
every point $(x,S)$ is either of type $(o,p), (p,p)$ or $(o,pp)$.

\bprop\label{reduced-graph-det-times}(Reduced Graph at Deterministic Times)
Let $S,T$ be a pair of deterministic times with $S<T$.  For almost every realization of the Brownian net with killing,  every $z_0=(x_0,S)$ is ``nice'' (with $x_0$ random or deterministic) and 
the following properties hold:
\ben
\ii If $z_0=(x_0,S)$ is of type $(o,p)$ or $(p,p)$ then $G^{z_0}(T)$
is a simply rooted graph. In particular,
for any deterministic point $z_0$, $G^{z_0}(T)$ is simply
rooted. 
\ii If $z_0 = (x_0, S)$ is of type (o, pp), then  the root $z_0$ of the graph
$G^{z_0}(T)$ is either of out-degree 1 or 2. 
\ii Every intermediate node in $G^{z_0}(T)$ has out-degree 2.
\ii The set
\beq
S_2(S,T) : =  \{x_0 \ : \ G^{(x_0,S)}(T) \ \ \text{is not simply rooted} \ \}
\eeq
is a locally finite set and is made of $(o,pp)$ points. Further, between two consecutive points of $S_{2}(S,T)$,
the graphs only differ by their root.
 I.e., if  $x,y$ are  such that
$$
x_1< x \leq y <y_1,
$$ 
where $x_1$ and $y_1$ are two consecutive
points in $S_2(S,T)$,
then $G^{(x,S)}(T)$ 
and $G^{(y,S)}(T)$ 
only differ by their root.
\een
\eprop

The proof of this result relies on some properties of the special point of the Brownian net.
Let $(r,l)\in (\Wr, \Wl)$ and let $z,z'$ be two vertices 
in the graph $G^{z_0}(T)$. We will write 
$$
z: =(x,t)\to_{l,r}z':=(x',t')
$$ 
iff $t<t'$, $l \sim_{out}^{z} r$ and  $l \sim_{in}^{z'} r$, and 
$z$ and $z'$ are connected by the paths $l$ and $r$, in the sense that $l$ and $r$ do not encounter 
any element of $V^{z_0}(T)$ in the time interval $(t,t')$. 
Among other things, the following technical lemma will allow to relate the out-degree of
a root to its type (see Corollary \ref{cor-eq}).

\blem\label{lem-lem}
Let $S<T$ be two deterministic times. For almost every realization of the Brownian net with killing,  for every $z_0=(x_0,S)$ (with $x_0$ random or deterministic),
the following properties hold.
\begin{enumerate}
\item Suppose $z\in V^{z_0}(T)$  is not a leaf.
For every equivalent outgoing  pair $(l,r)$ from $z$,
there exists a unique $z'\in V^{z_0}(T)$
such that  $z\rightarrow_{l,r}z'$. 
\item For every $z,z'\in V^{z_0}(T)$ such that $z\to z'$,
there exists a unique pair $(l,r)\in(\Wr, \Wl)$ such that 
$z=(x,t)\to_{l,r} z'=(x',t')$. 
\item Let $z_-=(x_-,S)$ and $z_+=(x_+,S)$ with $x_-\leq x_+$ such that there exist two distinct equivalent outgoing paths $(l,r')$
and $(l',r)$  starting respectively from $z_-$ and $z_+$, with the convention that
when $x_-=x_+$,  $(l,r')$ and $(r',l)$ are ordered from left to right. Finally, let 
$z_l$ and $z_r$ such that 
$$
z_-\to_{l,r'} z_l \ \ \mbox{and} \  \ z_+\to_{l',r} z_r. 
$$
Then $z_l\neq z_r$ if and only if $l$ and $r$ do not meet
on $(S,U_{l,r'}\wedge U_{l',r}]$,
where
\begin{equation}\label{Ulr}
U_{l,r'}:=T \ \wedge \ \inf\{u\geq S \ : \ \exists \pi\in\Net^{b,\kappa} \ \ \mbox{s.t.} \  \sigma_{\pi}=S, \ e_\pi = u, \ \mbox{and}  \  l \leq \pi \leq r' \}
\end{equation}
and $U_{l',r}$ is defined analogously using $l'$ and $r$.
\end{enumerate}
\elem

\begin{proof}

We start with the proof of item 1. 
Let $z:=(x,t)\in V^{z_0}(T)$, 
for every equivalent outgoing  pair $(l,r)$ at $z$, 
define $U_{l,r}$ analogously to (\ref{Ulr}), i.e.,
$$
U_{l,r} \ = \ T \ \wedge \ \inf\{u \geq t \ : \ \exists \pi\in\Net^{b,k} \ \ \mbox{s.t.} \  \sigma_{\pi}=t, \ e_\pi = u, \ \mbox{and}  \  l \leq \pi \leq r' \}.
$$
Finally, define
\be\label{taulr}
\tau_{l,r} := \sup\{u \in [t,U_{l,r}] : l(u) = r(u)\}
\ee
to be the last time $l$ and $r$ separate
before time $U_{l,r}$. We first claim that  the point
\begin{equation}\label{zprime}
z':=(l(\tau_{l,r}),\tau_{l,r}) = (r(\tau_{l,r}),\tau_{l,r}) 
\end{equation}
belongs to $V^{z_0}(T)$ and that  $r \sim^{z'}_{in} l$.
In order to prove this claim, we distinguish two cases. 

Let us first assume that $\tau_{l,r} = U_{l,r}$. There are only two possibilities: If $U_{l,r}=T$, then $z'$ is a leaf and must must be of type $(p,p)$ by Proposition \ref{special-net}(4); 
if $U_{l,r}\neq T$, the point $z'$ must be a killing point (ands thus a leaf) which is also
of type $(p,p)$ (this follows directly from the structure of the intensity measure (\ref{ex-cl}) used in the Poissonian construction
of the killing points).
Hence, when $\tau_{l,r} = U_{l,r}$, the point $z'$ (as defined in (\ref{zprime})) is always a leaf of $V^{z_0}(T)$ and is always of type $(p,p)$, 
which implies that $l \sim^{z'}_{in} r$. 

\begin{figure}[h!]
\centering
\includegraphics[scale=.2]{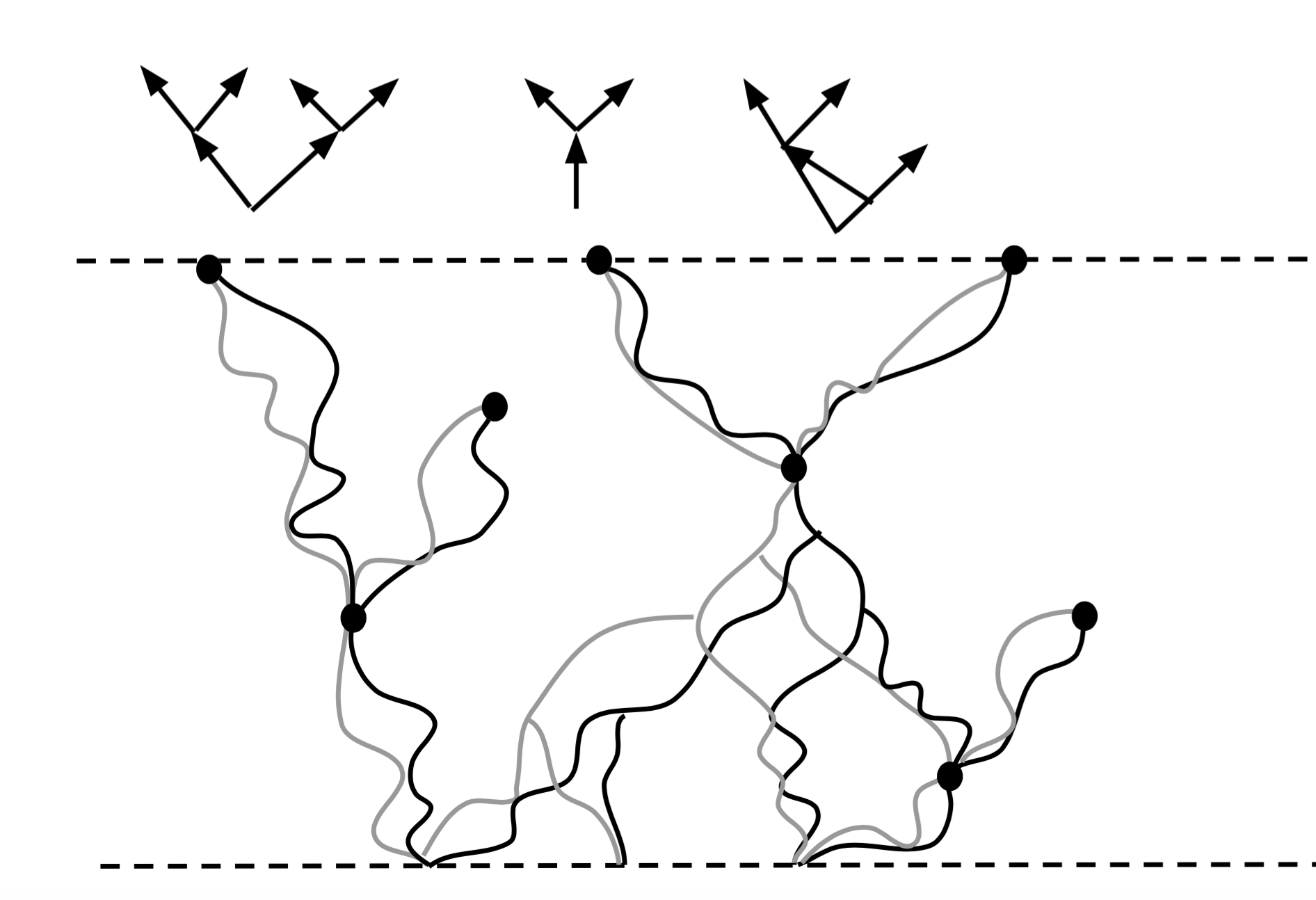}
\caption{Three reduced graphs generated from the net paths starting from points on the line $\R\times\{S\}$ (with time horizon $T$). The leftmost and rightmost graphs are generated from two consecutive points in $S_2(S,T)$. Grey (resp., black) paths represent paths from $\Wl$ (resp., $\Wr$).}
\label{reduced-graph}
\end{figure}

Let us now assume that $\tau_{l,r} < U_{l,r}$. It is not hard to show that
$z'$ is an (S,T)-relevant separation point in $G^{z_0}(T)$.
Since separation points 
are of type $(p,pp)$ (by Proposition \ref{geo-net-sp}(1)), 
$z'$ is a point of $V^{z_0}(T)$ (a separation point)
with the property
$l \sim^{z'}_{in} r$.

In the two previous paragraphs, we showed that the point $z'=(l(\tau_{l,r}),\tau_{l,r})$ belongs to $V^{z_0}(T)$ and that 
$l\sim^{z'}_{in} r$. Furthermore,
the definition of $\tau_{l,r}$
directly implies that $r$ and $l$
do not encounter any relevant separation point or any killing point
before entering the point $z'$. Combining those two facts implies that $z\to_{l,r} (l(\tau_{l,r}),\tau_{l,r})$. 
This ends the proof of item 1.



\medskip

We now proceed with the proof of the second item of Lemma \ref{lem-lem}. Let $z,z'\in V^{z_0}(T)$ be such that $z\to z'$.
Let $\pi\in\Net^{b}$ be a path connecting $z$ and $z'$.  $z$ is not a leaf and thus it can either be (1) a separation point, or (2) a root.
In case (1), Proposition \ref{geo-net-sp} implies that $\pi$ is squeezed between a unique equivalent pair $(l,r)$ such that $l \sim^{z}_{out} r$.
In case (2), since we are considering the case of a reduced graph rooted at a deterministic time, 
the same property must hold by Proposition \ref{special-net}.  
Let $\tau_{l,r}$ be defined as in (\ref{taulr}). By definition of $\tau_{l,r}$,
$\pi$ can not access any point of $V^{z_0}(T)$ on the interval $(t,\tau_{l,r})$. Since we proved that $(l(\tau_{l,r}),\tau_{l,r})\in V^{z_0}(T)$,
$z'$ must coincide with $(l(\tau_{l,r}),\tau_{l,r})$. Finally, since we 
showed that $z\to_{l,r} (l(\tau_{l,r}),\tau_{l,r})$, this ends the proof of item 2.

\bigskip

We now proceed with the proof of the third item of Lemma \ref{lem-lem}.
The existence of $z_l$ and $z_r$ is guaranteed by
item 1 of Lemma \ref{lem-lem}.
Further, from what we showed earlier, 
$z_l=(l(\tau_{l,r'}), \tau_{l,r'})$ and $z_r=(r(\tau_{l',r}), \tau_{l',r})$ and thus
$$
z_l = z_r \ \ \Longleftrightarrow  \ \ \tau_{l,r'}=\tau_{l',r}, \ \mbox{and} \ \ l(\tau_{l,r'})=r(\tau_{l',r}).
$$ 
Using the non-crossing
property of the Brownian web paths (in $\Wl$ or $\Wr$),  
if $l$ and $r$ meet before time $U_{l,r'}\wedge U_{l',r}$, then 
the condition on the RHS of the former equivalence is satisfied.
Conversely, if  the condition on the RHS of the former equivalence is satisfied,
it is clear that $l$ and $r$ meet before time $U_{l,r'}\wedge U_{l',r}$.
This ends the proof of Lemma \ref{lem-lem}.





\end{proof}

As a direct corollary of Lemma \ref{lem-lem}(1) and (2), we deduce the following result.

\bcor\label{cor-eq}
Let $S<T$ be two deterministic times. For almost every realization of the Brownian net with killing,  for every $z_0=(x_0,S)$ (with $x_0$ random or deterministic), the out-degree of a point $z\in V^{z_0}(T)$ is at most equal to the number of pairs $(l,r)\in (\Wl,\Wr)$ such that $l \sim_{out}^z r$.
\ecor

\begin{proof}[Proof of Proposition \ref{reduced-graph-det-times}]

The fact that $z_0$ is ``nice'' is the content of Proposition \ref{finite-3.3}. We now turn to the proof of item 1. 
By Proposition \ref{special-net}, for any point of type $(o,p)$ or $(p,p)$,
there exists a unique pair $l,r$ such that 
$l\sim^{z_0}_{out}r$. By Corollary \ref{cor-eq},
this proves that 
the degree of $z_0$ is at most to 1. Since $z_0$ is a.s. not a killing point,
the out-degree is at least $1$, thus showing that the finite graph representation at any point of type $(o,p)$ or $(p,p)$
is simply rooted. Finally, since every deterministic point is a.s. of type $(o,p)$ (by Proposition \ref{special-net}),
the same property must hold at any deterministic point a.s.. This ends the proof of item 1.

\medskip

We now turn to the proof of items 2 and 3. 
We first claim that
any node
of type $(*,pp)$ --- where * can be either be $o$ (for roots of type (o,pp)) or 1 (corresponding to separation point
by Proposition \ref{special-net} and Proposition \ref{geo-net-sp}) ---
connects to one or two points. 
By definition of a point of type $(*,pp)$,
there exists exactly two distinct pairs 
$(l,r')$ and $(l',r)$ in $(\Wl,\Wr)$
starting from $z$
such that
$l\sim_{out}^z r'$ and $r'\sim_{out}^z l$. 
By Corollary \ref{cor-eq},
the out-degree is at most two, thus proving item 2. Further, by definition of a relevant separation
points, 
the out-degree for such a point is at least two, and thus the out-degree 
of a relevant separation point 
is exactly equal to $2$. This ends the proof of item 3.

\medskip

It remains to show item 4.
The fact that $S_2(S,T)$
is a set of $(o,pp)$ points follows from 
Corollary \ref{cor-eq} and the fact that at deterministic times,
points are either of type $(o,p), (p,p)$ or $(o,pp)$.
Next, we show  that $S_2(t,T)$
is locally finite.
Lemma \ref{lem-lem}(3) (in the special case $x_-=x_+$) easily implies that 
if $x$ is in $S_{2}(S,T)$
then there exists a time $t_x$ --- either a killing
time for a path starting from time $S$,
or $t_x=T$ --- 
such that the rightmost and leftmost paths $l,r$ starting from $(x,S)$ 
do not meet before time $t_x$.

Let $L>0$. Since ${\cal M}^{\kappa}(S)$ is locally finite (by Proposition \ref{p214}),
the previous argument and 
the equicontinuity of the net paths 
imply that there exists a rational number $q>0$  such that $[-L,L]\cap S_2(S,T)$ is a subset of   
\be\label{3.20}
 \{x \ :  \ \mbox{the leftmost and rightmost path  in $\Net^b$ starting from $(x,S)$ do not meet on $(S,S+q)$} \};
\ee
e.g., take $q<\inf\{s>S \ : \ (y,s)\in{\cal M}^\kappa \mbox{ and } \exists \pi\in\Net^b, \ \sigma_\pi=S, \ \pi(S)\in[-L,L], \pi(s)=y\}$.
Since a.s. for every realization of the net, the  set  (\ref{3.20}) is locally finite for every deterministic $q$,
this proves that $[-L,L]\cap S_2(S,T)$ is finite a.s. and thus, $S_2(S,T)$ is locally finite.   
(This can easily be seen as a direct consequence from the fact that the dual branching-coalescing point set (the dual branching-coalescing point set being defined from the dual Brownian net as defined in \cite{SS07}) starting from $t+q$ is finite at time $t$ by Proposition \ref{bc-set}.)


To conclude the proof of item 4, it remains to show
that $G^{z_-}(T)$ and $G^{z_+}(T)$
are identical (up to some relabeling of the roots)
if $z_-=(x_-,S)$ and $z_+=(x_+,S)$
are located (strictly) between two consecutive 
elements of $S_2(S,T)$. We give an argument by contradiction
and assume that the roots of 
$G^{z_-}(T)$ and $G^{z_+}(T)$ do not connect to the same point (i.e., $z_l\neq z_r$).
We now show that this implies 
the
existence of a point $z_0=(x_0,S)$
between $z_-$ and $z_+$ (i.e. $x_-\leq x_0 \leq x_+$)
that
connects to two distinct points. 

Since $[x_-,x_+]\cap S_2(S,T)=\emptyset$, we can partition the segment 
$[x_-,x_+]$ into equivalence classes, where  $x_1,x_2$ are equivalent
iff they connect to the same vertex (i.e., to the same element of $L^\kappa(S,T)\cup R^\kappa(S,T) \cup K^\kappa({S,T})$ as defined in (\ref{def-kl})).
Since $L^\kappa(S,T)\cup R^\kappa(S,T) \cup K^\kappa({S,T})$ is locally finite, the equicontinuity of the net paths implies that 
the number of classes is finite. From there, it is straightforward to check that there exists $L\in[x_-,x_+]$, 
$z_1$ and $z_2$ two elements in $L^\kappa(S,T)\cup R^\kappa(S,T) \cup K^\kappa({S,T})$, 
and two sequences of net paths $\{\pi_1^n\}$
and $\{\pi_2^n\}$ with $\sigma_{\pi_i^n}=S,i=1,2$ such that 
$$
(\pi_1^n(S),S) \mbox{ connects to } z_1, \ (\pi_2^n(S),S) \mbox{ connects to }  z_2 \ \ \mbox{and} \ \ \lim_{n\to\infty} \pi_1^n(S), \pi_2^n(S)=L.
$$
For $i=1,2$, going to a subsequence if necessary, $\pi_i^n$ converges to a net path $\pi_i^\infty$
starting from $(L,S)$ and passing trough the point $z_i$. We will now show that $\pi_i^\infty$
connects $z_0$ to $z_i$ (in order to get the desired contradiction). In order to do so,
it is sufficient to prove that $\pi_i^\infty$ 
does not encounter any element in $L^\kappa(S,T)\cup R^\kappa(S,T) \cup K^\kappa({S,T})$ before entering $z_i$.

Again, we argue by contradiction. 
Let us assume that $\pi_i^\infty$
connects $(L,S)$ to a point $\bar z=(\bar x, \bar t)$ before entering $z_i$. Since $z\to\bar z$, by the second item of Lemma \ref{lem-lem}, 
there must exist a left-right pair
$(l,r)$ 
 such that $l\leq \pi_i^\infty \leq r$ on $[0,\bar t]$. Let $q$  be any rational number 
 in $(S,\bar t)$.
By Proposition \ref{bc-set},
the set $\xi^S(q)$ is locally finite a.s., 
and since $\pi_1^n$ converges to $\pi_1^\infty$, $\pi_1^n(q)$ and $\pi_1^\infty(q)$ coincide for $n$ large enough.
This implies that $\pi_1^n$ is ``trapped'' between $l$ and $r$ after a finite rank, i.e., that 
$l\leq \pi_i^n \leq r$ on $[q,\bar t]$. 
(Here we used the property that a net path can not cross a left-most path from the right, and a rightmost path from the left -- see \cite{SS07}).
This implies that $\pi_i^n$ enters the point $\bar z$ for $n$ large enough, which is the desired contradiction. This completes the proof of item 4 of Proposition \ref{reduced-graph-det-times}.
\end{proof}



Let us take a point $x$ 
in $S_2(S,T)$. By Proposition \ref{reduced-graph-det-times}, $(x,S)$ must be of type $(o,pp)$, and  connects to two distinct points. By Lemma \ref{lem-lem}(3), 
there exist two points $z',z''$ with $z'\neq z''$,
such that 
if $(l,r')$ and $(l',r)$ are the two distinct outgoing equivalent paths in $(\Wl,\Wr)$
exiting the point $(x,t)$, ordered from left to right,
then
$$
(x,S)\rightarrow_{l,r'} z'=(x',t')  \ \ \ (x,S)\rightarrow_{l',r} z''=(x'',t'').
$$

One can always decompose $G^{(x,S)}$ into 
two simply rooted graphs, a left and a right component  $G^{(x,S)}_l(T)$ and $G^{(x,S)}_r(T)$: We
define 
$G^{(x,S)}_l(T)$
as the subgraph whose set of vertices is the union 
of $(x,S)$ with the vertices of the (directed) subgraph originated from $z'$.
$G^{(x,S)}_r(T)$
is defined analogously using $z''$. For a pictorial representation 
of the next proposition, we refer the reader to Fig. \ref{reduced-graph}.

\bprop
\label{double-graphs}
Let  $x_l,x,x_r$ be three consecutive points 
in $S_2(S,T)$ and let $G^{(x,S)}_l(T)$
and $G^{(x,S)}_r(T)$ be respectively the left and right component
of the graph $G^{(x,S)}(T)$.
\beqnn
\forall y \in (x_l,x), & \text{$G^{(y,S)}(T)$ and $G^{(x,S)}_l(T)$ only differ by their root}, \\
\forall y \in(x,x_r),  & \text{$G^{(y,S)}(T)$ and $G^{(x,S)}_r(T)$ only differ by their root}.
\eeqnn
\eprop

\begin{proof}
We only prove the first relation. The second relation
can be proved along the same lines. 
In the following, as before, $(l,r')$ and $(l',r)$ will denote the two outgoing equivalent paths in $(\Wl,\Wr)$
exiting the point $(x,S)$, ordered from left to right. In particular, it is not hard to see from Proposition \ref{special-net},
that $l$ (resp., $r$) is the leftmost (resp., rightmost) path starting from $(x,S)$.

Take $x_n\uparrow x$, and for every $n$, let $(l_n,r_n)$
be an equivalent pair of outgoing paths starting from $(x_n,S)$. 
By Proposition \ref{reduced-graph-det-times}(4),
all the points in $(x_l,x)$ have the same finite graph
representation, up to relabeling of their root. As a consequence,
we need to prove that 
if $x_n\uparrow x$ then 
$G^{(x_n,S)}(T)$ and  $G^{(x,S)}_l(T)$ coincide up to their root, after some rank, i.e., we need 
to show that 
$(x_n,S)\rightarrow_{(l_n,r_n)}(x',t')$ (after some rank),
where $z'=(x',t')$
is the only point in $G_l^{(x,S)}(T)$
connected to $(x,S)$.

According to Lemma \ref{lem-lem}(3),
and since $U_{l_n,r'}\geq U_{l_n,r_n}\wedge U_{l,r'}$  (as defined in (\ref{Ulr})),
we need to prove that $l_n$ and $r'$ meet
before time $U_{l_n,r'}$.
Let us first show that $l_n$ coalesces with $l$ at some time $\tau_n\rightarrow S$.
By compactness of the Brownian web, $l_n$ converges to some path $\bar l\in \Wl$ starting from $(x,S)$. 
Since $l_n$ does not cross the path $l$, we have $l_n \leq l $, implying that
$\bar l \leq l$. Since $l$ is the leftmost path
starting from $(x,S)$, this yields $\bar l = l$, and thus $l_n$ converges to $l$. Using an argument analogous to the one given in the last paragraph of the proof of Proposition \ref{reduced-graph-det-times},
it is not hard to show 
that $l_n$ coalesces with $l$ at some $\tau_n\rightarrow S$, as claimed earlier.
Next,
since $r_n$ can not cross $r'$, we have
$r_n \leq r'$. This  implies that
for any $s \geq \tau_n$, 
$$
l(s) = l_n(s)\leq r_n(s)\leq r'(s).
$$
On the other hand, since
$l \sim_{out}^z r'$, we can choose a time $\tilde t$ arbitrarily close to $S$,
such that $l(\tilde t)=r'(\tilde t)$. In particular, for  $n$ 
large enough such that 
$\tau_n \leq \tilde t$,
the previous inequality implies that
$l_n$ and $r'$ must have met before time $\tilde t$.
Finally, since the set 
of killing points $\M^\kappa(S)$ is locally finite,
and since $\tilde t$ can be chosen arbitrarily close to $S$,
we can choose $n$ large enough such that 
$l_n$ and $r'$
meet
before any path squeezed between those two paths encounters a killing point, i.e., before time $U_{l_n,r'}$.
As explained earlier, by  Lemma \ref{lem-lem}(3), this implies that for $n$ large enough, the graphs
$
G^{(x_n,S)}(T)
$
and
$
G_l^{(x,S)}(T)
$
coincide modulo their root. This ends the proof
of Proposition \ref{double-graphs}.

\end{proof}

\subsubsection{Proof of Theorem \ref{teo1} (1)--(2)}

Let $t$ be a deterministic time.
Proposition  \ref{finite-3.3} implies that for any point $(x,t)$ (with $x$ deterministic or random) is ``nice''
and has a well defined pre-coloring $\bar \theta(x,t)$. 

\blem\label{pr-eq-def}
At any deterministic time $t$, for almost every realization of the CVMP,
for every $x$ (deterministic or random), 
the pre-coloring $\bar \theta(x,t)$
coincides with the coloring $\theta(x,t)$.
\elem
\begin{proof}
The key point
of our proof
is that if two graphs only differ by their
root,
they must have the same pre-coloring 
(by the very definition of the coloring algorithm as defined 
in Section \ref{coloring-algo}).
We denote by $\hat S_2(0,t)$
the backward analog of $S_2(0,t)$ as defined
in Proposition \ref{reduced-graph-det-times}(4), i.e.
$$
\hat S_2(0,t) \ = \ \{x \ : \ \hat G^{(x,t)}(0) \ \ \mbox{is not simply rooted.}\}
$$

We first show that Lemma \ref{pr-eq-def}  holds for 
$z=(x,t)\notin \hat S_2(0,T)$. 
Since $\hat S_2(0,T)$ is a.s. locally finite, Proposition
\ref{reduced-graph-det-times}(4) --- in its backward formulation ---
implies that
there exists $\delta>0$
small enough s.t.
$$
\forall x', \ \ \text{s.t.}  \ \ |x'-x|<\delta, \ \ \  \hat G^{(x,t)}(0) \ \ \text{and} \ \ \hat G^{(x',t)}(0) 
\ \ \ \text{only differ by their root}
$$
As a consequence, for any $x_n\rightarrow x$,  $\bar \theta(x_n,t)$ is stationary after a certain rank, and thus
$$
\theta(x,t) = \lim_{n} \bar \theta(x_n,t) = \bar \theta(x,t).
$$

Let us now consider
$z\in \hat S_2(0,T)$. 
First, by definition
of the precoloring $\bar\theta(\cdot)$, we must have
$\bar \theta(z)=\{c_l\}\cup\{c_r \}$, with 
$c_l$ (resp., $c_r$)
being the color given by our coloring algorithm
applied on the graph $\hat G_l^{x,t}(0)$ (resp., $\hat G_r^{x,t}(0)$) --- those two
graphs being defined as in Proposition \ref{double-graphs} in the backward context. 
(Note that those two colors are potentially equal.)
Let us now consider
a subsequence $x_n\rightarrow x$, with $x_n\neq x$.
We can decompose this sequence into two subsequences $x_n^l<x$
and $x_n^r>x$. For the sequence $x_n^l$,
Proposition \ref{double-graphs}, implies that
for $n$ large enough, $\bar \theta(x_n^l,t)= c_l$,
while for the sequence $x_n^r$, we must have 
$\bar \theta(x_n^r,t)= c_r$. 
As a consequence, we get
$
\theta(z) = \bar\theta(z).  
$
\end{proof}

Theorem \ref{teo1} (1)--(2) is then a direct consequence of the following proposition.

\bprop\label{prop1-3}
At any deterministic time $t$, for almost every realization of the CVMP:
\ben
\ii 
$|\theta(x,t)|=1$ a.s., for any deterministic point $(x,t)$.
Furthermore, the coloring $\theta(x,t)$
is obtained by applying the coloring algorithm to the reduced graph starting from $(x,t)$ (see Section \ref{coloring-algo} for a definition of the coloring algorithm).
\ii For every $x$, $|\theta(x,t)|\leq 2$.
\ii The set of points such that $|\theta(x,t)|  \ = \ 2$ is locally finite. Moreover,
the color between two consecutive points of this set
remains constant. 
\een
\eprop

\begin{proof}
The first item of Proposition \ref{prop1-3} is a direct consequence of Proposition \ref{reduced-graph-det-times}(1) and Lemma \ref{pr-eq-def}. 
For the second item, we first note that Proposition \ref{reduced-graph-det-times} implies that for any point $(x,t)$ (with $x$ deterministic or random), the out-degree of the root of $\hat G^{(x,t)}(0)$ is either $1$ or $2$. Furthermore, the definition of our pre-coloring
implies $|\bar \theta(x,t)|$ is less or equal to the out-degree of the reduced  $\hat G^{(x,t)}(0)$. 
The second item then follows by a direct application of Lemma \ref{pr-eq-def}.

\medskip

We now turn to the third item of Proposition 3.
Since 
the $\hat G^{(x,t)}(T)$'s for the $x$'s between two consecutive
point of $\hat S_2(0,T)$
only differ by their root,  
$\bar\theta(x,t)$ (and hence $\theta(x,t)$) must remain constant between two elements of $\hat S_2(0,t)$.
Since  $\hat S_2(0,t)$ is locally finite (Proposition \ref{reduced-graph-det-times} (4)),
this implies that the set of points such that $|\theta(x,t)|  \ = \ 2$ is locally finite.

Finally,
$|\bar \theta(x,t)|=2$ (and hence $\theta(x,t)$)
iff $(x,t)\in\hat S_2(0,t)$
and
the colors induced by
the graphs $\hat G_l^{(x,t)}(0)$
and $\hat G_r^{(x,t)}(0)$ differ (where $\hat G_l^{(x,t)}(0)$
and $\hat G_r^{(x,t)}(0)$ are defined as in Proposition \ref{double-graphs}).
By Proposition \ref{double-graphs},
for every $x_n^l\uaw x$
(resp., $x_n^r\daw x$)
the color $\bar \theta(x_n^l,S)$ (resp., $\bar \theta(x_n^r,S)$)
must coincide with the color induced by  
$\hat G_l^{(x,t)}(0)$
(resp., $\hat G_r^{(x,t)}(0)$) after a certain rank. This
implies that $|\theta(x,t)|=2$ iff
$x$ delimits two color intervals. 

\end{proof}


\section{Convergence to the continuum model. Proof of Theorem \ref{teo1} (3)}
\label{Invariance::Principle}
Let $\{\eps_n\}$ be a sequence of positive numbers such that $\eps_n\to 0$. As in Theorem \ref{teo1}, we assume that there exists $b,\kappa\geq0$,
and two positive sequences $\{b_n\}$ and $\{\kappa_n\}$ such that $b_n/\eps_n \to b$ and $\kappa_n/\eps_n^2\to \kappa$. We also assume that the boundary and bulk mechanisms $\{g_{i,j}^{n}\}_{i,j\leq q}$ and $p^{n}$ 
converge to a sequence of limiting probability distributions  $\{g_{i,j}\}_{i,j\leq q}$ and $p$  when $n\to\infty$.

Let $z_n^1,\cdots, z_n^k\in S_{\eps_n}(\Z_{odd}^2)$ be such that for every $i=1,\cdots,k,  \ \lim_{n\to\infty} z_n^i=z^i$ for some $z^i\in\R^2$. 
Our goal is to show that the distribution of the coloring of the points $(z^1_n,\cdots, z^k_n)$
for the discrete model labelled by $n$
converges to the distribution of the coloring of $(z^1,\cdots, z^k)$
at the continuum. 

By Proposition \ref{prop1-3}(1), the coloring of a deterministic point $z$ in the CVMP 
is obtained by applying the coloring algorithm to 
the reduced graph $\hat G^{z}(0)$ (the backward of the object introduced in Definition \ref{def:reduced-graph}). By Proposition \ref{reduced-graph-df2},
the color of a point in the discrete model can be recovered by applying the same procedure
to the discrete backward net with killing.

From the previous observation, and since the same property holds at the discrete level  (see Proposition \ref{reduced-graph-df2}), 
our convergence result easily
follows from the convergence (in law) of $\cup_{i=1}^k \hat G^{z^i_n}_n(0)$ to
$\cup_{i=1}^k \hat G^{z^i}(0)$, where  $\hat G^{z^i_n}_n(0)$ is the backward reduced graph
starting from $z_n^i$
with  time horizon $0$ and constructed from the (backward) discrete net with killing $S_{\eps_n}(\hat \cU^{b_n,\kappa_n})$
restricted to the upper half-plane. In this section, and for the sake of clarity, we will only prove
the result for $k=1$ (i.e., the convergence of one dimensional marginals of the coloring). The general case
involves more cumbersome notation, but can be treated along exactly the same lines.

Let $z_n\in S_{\eps_n}(\Z_{odd}^2)$ such that $z_n$ converges to $z=(x,T)$. We need to show that  $\hat G^{z_n}_n(0)$ converges to 
$\hat G^{z}(0)$. By reversing the direction of time and using translation invariance (in the $x$ and $t$ directions), 
it is equivalent to show the following result in the forward-in-time setting:

\bprop\label{mm1}
Let $z_n\in S_{\eps_n}(\Z_{odd}^2)$ such that $z_n\to (0,0)$. For every $T>0$,
$$
G_n := G^{z_n}_n(T ) \ \to \ G:= G^{(0,0)}( T ) \ \  \mbox{in law,}
$$
where $G^{z_n}_n(T )$ is the reduced graph at $z_n$
with time horizon $T$ for the (rescaled) forward discrete net with killing $S_{\eps_n}(\cU^{b_n,\kappa_n})$.
\eprop

\subsection{ Proof of Proposition \ref{mm1}}   

We start by introducing some notation. Define
$$
\Net^{b,\kappa}(\Sigma_0) \ := \  \{\pi\in\Net^{b,\kappa} \ : \ \sigma_\pi=0 \}
$$ 
and for any path $\pi\in\Net^{b,\kappa}(\Sigma_0)$, 
let $r(\pi)$ be the (a.s. finite) sequence 
of points in
\be\label{def:V}
L^\kappa({0,T}) \bigcup K^\kappa({0,T}) \ \bigcup \ R^\kappa(0,T)
\ee
(as defined in (\ref{def-kl}))
visited by the path $\pi$, the points in the sequence being listed in their order of visit. (In particular, the final point of $r(\pi)$
must be a leaf, and every other point is a relevant separation point). In the following,
$r(\pi)$ will be referred to as the finite graph representation of $\pi$.

Analogously to the continuum level, we assume that 
the discrete net with killing $\cU^{b_{n},\kappa_n}$ is coupled with 
a discrete net $\cU^{b_n}$ (where $\kappa_n=0$). (See Section \ref{onenet} for more details).
$S_{\eps_n}(\cU^{b_{n},\kappa_n}(\Sigma_0))$ will denote the discrete analog of $\Net^{b,\kappa}(\Sigma_0)$, and for 
any path $\pi_n\in S_{\eps_n}(\cU^{b_n,\kappa_n}(\Sigma_0))$,
$r_n(\pi_n)$ will denote the finite graph representation at the discrete level.

\medskip
For any point $(x,t)$ and $s\leq t$, define
$$
d^s((x,t)) \ := \ \inf\{|y-x| \ : \ y\neq x \mbox{ and } y\in\xi^s(t) \}.
$$
and from the Brownian net (with no killing), define the point process 
\be\label{bR}
\bar R(s,t) \ := \ \sum_{z\in R(s,t)} \delta_{z,d^s(z)}.
\ee
In words, $\bar R(s,t)$ records the location of the $(s,t)$--relevant separation points (in the standard Brownian net (with no killing)), together with a measure of
``isolation'' around each of those points. Analogously, we define $\bar R_n(s,t)$ for the rescaled discrete net $S_{\eps_n}(\cU^{b_n})$.

\medskip

Let $L>0$. Recall the definition of $\theta^{i}$ given in (\ref{thetai}) for $S=0$, i.e. $\theta^0=0$ and
for every $i\geq0$
$$
\theta^{i+1} \ := \ \inf\{t > \theta^i \ : \ \exists x\in(-L,L), \ (x,t)\in{\cal M}^\kappa(\theta^i)\}.
$$
where we recall that ${\cal M}^\kappa(S)$ is the set of killing points
attained by some path in $\Net^{b,\kappa}$ starting at time $S$. Define 
\beqn\label{def:N}
N=\inf\{i\geq 1: \theta^i\geq T\}.
\eeqn
For every $s<t$,  define
$$
\square_{s,t}^L\ := \ (-L,L)\times [s,t],
$$
and for every $i\geq 0$, 
\beqnn
{\cal R}^{i+1} & :=   & R(\theta^{i}\wedge T,\theta^{i+1}\wedge T)\cap \square_{0,T}^L \\
\\
{\cal L}^{i+1} & := & \left\{
\begin{array}{cc}
\{(x,\theta^{i+1})\in{\cal M}^{\kappa}(\theta^i) \ : \  x\in\xi^0(\theta^{i+1})\cap(-L,L)\} & \mbox{if $\theta^{i+1}<T$} \\
\{(x,T) \ : \ x\in\xi^0(T)\cap(-L,L)\} & \mbox{otherwise}.
\end{array} \right.
\eeqnn
Note that when $\theta^{i+1}<T$, ${\cal L}^{i+1}$ is a single killing point a.s.. Analogously to (\ref{bR}),
for every $i\geq 1$, we also define
\begin{eqnarray*}
\bar {\cal R}^{i}  \ := \ \sum_{z\in   {\cal R}^i}  \ \delta_{z,d^0(z)}, \ \ 
\bar {\cal L}^{i} \ := \ \sum_{z\in {\cal L}^i} \  \delta_{z, d^0(z)}. 
\end{eqnarray*}
(To ease the presentation, $L$ and $T$ are not explicit in the notation). 
Finally,  $\theta^i_n, N_n, {\cal R}_n^{i} ,\bar {\cal R}_n^{i}$ and ${\cal L}_n^{i}, \bar {\cal L}_n^{i}$ will denote the analogous quantities for the rescaled 
process $S_{\eps_n}(\cU^{b_n}, \cU^{b_n,\kappa_n})$.  
As in the proof of Proposition \ref{finite-SP}, the previous definitions are motivated by the fact that if
the paths starting from $(0,0)$ of the Brownian net (with no killing) between time $0$ and $T$ 
are all contained in the box $\square_{0,T}^L$, i.e.,
$$
\{(\pi(s),s) \ : \ s\in[0,T], \ \pi\in\Net^b((0,0)) \}\subset \square_{0,T}^L,
$$
then  any intermediary point  (resp., leaf) in $G$ must be a point in
$\cup_i {\cal R}^i$ (resp., $\cup_i {\cal L}^i$). See Fig. \ref{desin}. 
The analogous property holds at the discrete level.

\medskip

The next two results will be instrumental in the proof of Proposition \ref{mm1}.
Informally, Proposition \ref{isolation} states that relevant separation points are isolated 
from the left and from the right.
   Proposition  \ref{prop-principal} implies that the same holds at the discrete level;
furthermore it shows that each continuum separation point is the limit of a single discrete separation point,
which is essential for the proof of  Proposition \ref{mm1}, the convergence of reduced graphs.

\bprop\label{isolation}
For almost every realization of  $\Net^{b,\kappa}$, $d^0(z)>0$ for every $z$ in $\cup_{i=1}^{N} ( {\cal R}^{i} \cup {\cal L}^{i})$.
\eprop
\begin{proof}
By Proposition \ref{special-net} and \ref{geo-net-sp} and the construction of the killing points, $z$ is of type $(p,*)$ -- with $*$ being equal either to $p$ and $pp$. Proposition
\ref{isolation} is then the content of Proposition 3.11(c) in \cite{SSS08a}.
\end{proof}

\bprop\label{prop-principal}
Let $\bar {\cal P}$ be the space of Radon measures on $\R\times\R^+\times\R^+$ equipped with the vague topology. 
For every $L>0$ and $T>0$, the random variable $N$ is finite and for every $1\leq i\leq N$, 
$\bar {\cal R}^{i}$ and  $\bar {\cal L}^{i}$ are finite measures a.s.. Further,   
$$
\left\{\left(S_{\eps_n}(\cU^{b_n,\kappa_n}),S_{\eps_n}(\cU^{b_n,\kappa_n}(\Sigma_0)),\  \sum_{i=1}^{N_n} \bar {\cal R}_n^{i}, \ \sum_{i=1}^{N_n} \bar {\cal L}_n^{i} \right) \right\}_{n\geq1} \ \to\ \ 
\left(\Net^{b, \kappa}, \Net^{b, \kappa}(\Sigma_0), \  \sum_{i=1}^N \bar {\cal R}^{i}, \ \sum_{i=1}^N \bar {\cal L}^{i}\right) \ \ \mbox{in law,}
$$
where the convergence is meant in the product topology     ${\cal H} \times  {\cal H}\times \bar \cP\times \bar \cP$.    
\eprop

\bprop\label{st2}
Let $z_n=(x_n,t_n)\in S_{\eps_n}(\Z_{odd}^2)$ such that $z_n\to (0,0)$. As $n\to \infty$,
$$
\P\left(\mbox{$G^{z_n}_n( T )$ and $G^{(0,0)}_n(T)$ are isomorphic}\right) \to 1
$$
\eprop

The proofs of Propositions \ref{prop-principal}  and Proposition \ref{st2} are rather technical, and for the sake of clarity  they are postponed until Sections \ref{proof-of} and 
\ref{proof-of2}.

\bigskip
  
\begin{proof}[Proof of Proposition \ref{mm1} (assuming Propositions~\ref{prop-principal}  and \ref{st2})]
By Proposition \ref{st2}, it is enough to show Proposition \ref{mm1}  for $z_n=(0,0)$.

Let $A_L$ be
the set of realizations such that $\{(\pi(s),s) \ : \ s\in[0,T], \ \pi\in\Net^b((0,0)) \}\subseteq \square_{0,T}^L$. 
By equicontinuity of the net paths, the probability of $A_L$ goes to $1$ as $L\to\infty$. Thus, to prove Proposition \ref{mm1}, it is enough to show that for every $L>0$, there exists a coupling
such that for almost every realization in $A_L$, $G_n$ and $G$ are isomorphic for $n$ large enough.

By the Skorohod Representation Theorem and Proposition \ref{prop-principal}, for every $L>0$,
there exists a coupling between the discrete and continuum levels such that 
$$
\left\{\left(S_{\eps_n}(\cU^{b_n,\kappa_n}(\Sigma_0)),\  \sum_{i=1}^{N_n} \bar {\cal R}_n^{i}, \ \sum_{i=1}^{N_n} \bar {\cal L}_n^{i} \right) \right\}_{n\geq1} \ \to\ \ 
\left(\Net^{b, \kappa}(\Sigma_0), \  \sum_{i=1}^N \bar {\cal R}^{i}, \ \sum_{i=1}^N \bar {\cal L}^{i}\right) \ \ \mbox{a.s.}
$$
  
Until the end of the proof, we  work under this coupling. 
For any $n\in\N$ and any realization in $A_L$, for any discrete
point $z_n\in \cup_{i\geq1} \ ({\cal R}_n^{i} \cup \ {\cal L}_n^{i})$, define $\psi_n(z_n)$ to be the closest point in 
$\cup_{i\geq1} \ ({\cal R}^{i} \cup \ {\cal L}^{i})$
 to $z_n$
(it is easy to see that $\psi_n(z_n)$ is well
defined a.s.).
As argued above,
we need to show that for almost every
realization in $A_L$,
for $n$ large enough,
the graphs $G_n$ and $G$ are 
isomorphic under this mapping.

\medskip

We start by deriving a criterium for having asymptotic isomorphism
of $G_n$ and $G$ by $\psi_n$. Let us assume that for every $M\in\N$,
there is $n\geq M$ such that
the image of $G_n$ by $\psi_n$
does not coincide with $G$.
Going to a subsequence if necessary, 
for every $n\in\N$, there exists $\pi_{n}\in S_{\eps_n}(\cU^{b_n,\kappa_n}((0,0)))$
such that the image of the finite graph representation $r_n(\pi_n)$
does not match any path of the continuum
graph $G$ (where $\cU^{b_n,\kappa_n}((0,0))$ denotes the set of discrete paths starting from the origin).
Since $\{S_{\eps_n}(\cU^{b_n,\kappa_n}(\Sigma_0))\}_{n\geq0}$
converges to $\Net^{b,\kappa}(\Sigma_0)$ a.s. in the $(\h,d_{\cal H})$ topology,
the sequence $\{S_{\eps_n}(\cU^{b_n,\kappa_n}(\Sigma_0))\}_{n\geq0}$
is tight. This implies
that
$\cup_{n}  \ S_{\eps_n}(\cU^{b_n,\kappa_n}(\Sigma_0))$
is a compact set of paths. Thus, again going
to a subsequence if necessary,
$\{\pi_{n}\}$
converges to a path $\pi\in\Net^{b,\kappa}((0,0))$, 
and the image of the finite graph representation $r_n(\pi_n)$ by the map $\psi_n$
is not equal to $r(\pi)$. 

As a consequence, in order to
prove  
Proposition \ref{mm1}, it is sufficient to prove 
that conditional on $A_L$, for any subsequence of discrete paths $\{\pi_n\}_{n\geq1}$ with $\pi_n \in S_{\eps_n}(\cU^{b_n,\kappa_n}((0,0)))$ converging to
 $\pi\in\Net^{b,\kappa}((0,0))$, the discrete and continuum paths have the same finite graph 
representation (up to $\psi_n$) after a finite rank. 

\medskip

Let us now consider a sequence $\pi_{n}\in S_{\eps_n}(\cU^{b_n,\kappa_n}((0,0)))$ converging to $\pi$ (up to a subsequence), and let us
show that for $n$ large enough,
for any $z_n=(x_n,t_n) \in \cup_{i\geq1} \ ({\cal R}_n^i \cup {\cal L}_n^i)$
then $\pi_n$ hits $z_n$  if and only if $\pi$ hits the point $(x,t):=\psi_n(z_n)$. 
First,
\beqnn
\forall (x_n,t_n) \in \cup_{i\geq1} \ ({\cal R}_n^i \cup {\cal L}_n^i), \ (x,t)=\psi_n(z_n): & \\
 |\pi(t)-x| & \leq & |\pi(t)-\pi_n(t_n)| \ + \ |\pi_n(t_n)-x_n| \ + \ |x-x_n|, \\ 
 |\pi_n(t_n)-x_n| & \leq & |\pi_n(t_n)-\pi(t)| \ + \ |\pi(t)-x| \ + \ |x-x_n|.
\eeqnn
Define
$$
\gamma_n := \min_{z\in  \cup_{i=1}^N \ ({\cal R}^i \cup {\cal L}^i)} d^0(z) \ \wedge \ \min_{z_n\in  \cup_{i=1}^{N_n} \ ({\cal R}^i_n \cup {\cal L}^i_n)} d^0_n(z_n)
$$ 
On the one hand, since $N$ (as defined in (\ref{def:N})) is finite and the ${\cal R}^i$'s and ${\cal L}^i$'s 
are made of finitely atoms, Proposition \ref{isolation} implies that under our coupling
$\liminf \gamma_n =  \min_{z\in  \cup_{i=1}^N \ ({\cal R}^i \cup {\cal L}^i)} d^0(z) >0$ a.s..

On the other hand, 
$\{\pi_n\}$ converges to $\pi$ and $\max_{z_n} \rho(z_n,\psi_n(z_n))\to 0$ 
by Proposition \ref{prop-principal} (where the max is taken over $\cup_{i=1}^{N_n} \ ({\cal R}^i_n \cup {\cal L}^i_n)$).
Using the equicontinuity of the paths in the sequence $\{\pi_n\}$, 
we can choose $n$ large enough such that the first and last terms of the two latter inequalities
are less than $\gamma_n/3$, i.e., such that 
\beqnn
\forall (x_n,t_n)\in \cup_{i\geq1} \ ({\cal R}_n^i \cup {\cal L}_n^i), \ (x,t)=\psi_n(z_n): & \\
|\pi(t)-x| & \leq & \ |\pi_n(t_n)-x_n| \ + \ \frac{2\gamma_n}{3} \\ &\leq &    |\pi_n(t_n)-x_n| \ + \ \frac{2}{3} d^0(z)  \\ 
|\pi_n(t_n)-x_n| & \leq &   |\pi(t)-x| + \frac{2}{3} d^0_n(z_n).
\eeqnn
From the definition of $d^0$ and $d^0_n$, this easily implies that for $n$ large enough (so that the previous inequalities are satisfied), the following statement holds:
$$
\mbox{
for any $z_n\in \cup_{i\geq1} \ ({\cal R}_n^i \cup {\cal L}_n^i)$,
$\pi_n$ hits $z_n$  if and only $\pi$ hits the point $\psi_n(z_n)$,}
$$ 
as claimed earlier. Finally, for every realization in $A_L$,
and for $n$ large enough, we already argued that 
the set of vertices of $G$ (resp., $G_n$) is a subset of $\cup_{i\geq1} \ ({\cal R}^i \cup {\cal L}^i)$ (resp., $\cup_{i\geq1} \ ({\cal R}^i_n \cup {\cal L}^i_n)$).
Combining this with the previous statement implies that $n$ large enough, the finite graph representation of $\pi_n$ and $\pi$ coincides (up to $\psi_n$).
Proposition \ref{mm1} then follows by applying the criterion
for the asymptotic isomorphism derived earlier in the proof.

\end{proof}

It remains to prove Proposition \ref{prop-principal} and Proposition \ref{st2}.

\subsection{Preliminary work for the proof of Proposition \ref{prop-principal}}

 For the time being, we only consider the Brownian net $\Net^b$ (with no killing). 
For every $t>0$, let $\xi^0_n(t)$ and $R_n(0,t)$ be the discrete counterpart
(in the rescaled  net $S_{\eps_{n}}({\cal U}^{b_n})$)
of $\xi^0(t)$ and $R(0,t)$. Finally, $\bar R_n(0,t)$
is the discrete analog of $\bar R(0,t)$. 
The rest of this section
is dedicated to proving the following result, which is closely related to Proposition \ref{prop-principal}.
\bprop\label{prop:dc1}
For every $T>0$, $(S_{\eps_n}(\cU^{b_n}), \bar R_n(0,T)) \ \to \ (\Net^b, \bar R(0,T)) \ \ \mbox{in law}$, 
where the convergence in the product topology ${\cal H}\times\bar {\cal P}$.
\eprop

In order to prove Proposition \ref{prop:dc1}, we will need the following lemma.

\blem\label{lem-si}
For almost every realization,  $d^0(z)>0$ for every $z\in V:=R(0,T) \cup \{(x,T) \ : \ x\in\xi^0(T)\}$.
\elem

\begin{proof}
This a special case of Proposition~\ref{isolation} when $k=0$.    
\end{proof}

\bcor\label{a-set}
For every $(z^1,\cdots,z^k)\in (\R^2)^n$, let 
\beqnn
A(z^1,\cdots,z^k) & := & \{\pi \in\Net^b(\Sigma_0) \ : \  \ r(\pi)=(z^1,\cdots,z^k)\} \\
B(z^1,\cdots,z^k) & := & \Net^b(\Sigma_0) \setminus A(z^1,\cdots,z^k).
\eeqnn
If $A(z^1,\cdots,z^k)\neq\emptyset$,  then $A(z^1,\cdots,z^k)$ is a    
compact set    and  $d_{\cal H}\left(A(z^1,\cdots,z^k), \ B(z^1,\cdots,z^k) \right)>0$.
\ecor

\begin{proof}
We first show that $A(z^1,\cdots,z^k)$ is a compact set of paths. 
Since $\Net^b$ is compact, it is is sufficient to prove that for any $\{p_l\}_{l\geq1}$ in $A(z^1,\cdots,z^k)$ converging to $\pi$ in $\Net^b(\Sigma_0)$, 
the finite graph representation $r(\pi)$ is given by $z^1,\cdots,z^k$. On the one hand, it is clear that  $\pi$
must hit the points $z^1,\cdots,z^k$. On the other hand, if $\pi$   touches    some $\bar z=(\bar x, \bar t)\in V$ distinct from $z^1,\cdots,z^k$,
we would have
$p_l(\bar t)\neq \bar x$, but $p_l(\bar t)\to \bar x$ as $l\to\infty$,
which would  imply that $d^0(\bar z)=0$.  
However, by Lemma \ref{lem-si},  $d^0(\bar z)>0$ a.s..
This shows that $A(z^1,\cdots,z^k)$ is a compact set of paths a.s.

   Next, since sets of the form $A(\bz^1,\cdots,\bz^k)$ with $(\bz^1,\cdots,\bz^k)\in \cup_{l\in\N^*} (\R^{2})^l$ are pairwise disjoint,    the distance between $A(z^1,\cdots,z^k)$
and
any finite union of sets of the form $A(\bz^1,\cdots,\bz^m)\neq \emptyset$ with $(\bz^1,\cdots,\bz^m)\neq (z^1,\cdots,z^k)$ is     strictly positive.
Furthermore,
by equicontinuity of the paths in $\Net^b$ and the local finiteness of $V$, the same property holds for any infinite union of those sets since one can restrict attention to a finite space-time box. 
Combining this with
the fact that $B(z^1,\cdots,z^k)$
can be written as the union of all the non-empty $A(\bz^1,\cdots,\bz^m)$'s with $(\bz^1,\cdots,\bz^m)\neq (z^1,\cdots,z^k)$
completes the proof of   
Corollary \ref{a-set}.

\end{proof}

Let ${\cal P}$ be the space of Radon measures on $\R\times\R^+$ equipped with the vague topology. From Proposition 6.14 in \cite{ScheSS08b}, it is already known that 
\be\label{cv-r}
(S_{\eps_n}(\cU^{b_n}), R_n(0,T))  \ \to \ (\Net^b ,R(0,T)) \ \ \mbox{in law,}
\ee
where the convergence is meant in the product topology ${\cal H}\times{\cal P}$. 
\blem\label{cv-xi-lemma}
For every deterministic $t>0$,
\be\label{cv-xi}
(S_{\eps_n}(\cU^{b_n}), \xi_n^0(t)) \ \to  \ (\Net^b, \xi^0(t)) \ \mbox{in law.}
\ee 
\elem
\begin{proof}
From the Skorohod representation theorem, there exists a coupling between the discrete and continuum levels such that 
$$
S_{\eps_n}(\cU^{b_n}) \ \to \ \Net^b \ \ \ \mbox{a.s..}
$$
Let $a<b\in\R$. The previous convergence statement entails that under our coupling:
$$
\liminf_{n\to\infty}  \ |\xi_n(t)\cap [a,b]| \ \geq \  |\xi(t)\cap [a,b]| \ \mbox{a.s.}
$$
On the other hand, using the same technic as in Proposition 4.3. in \cite{NRS15}, one can show that  
$$
\limsup_{n\to\infty}  \ \E(|\xi_n(t)\cap [a,b]|) \ = \  \E(|\xi(t)\cap [a,b]|).
$$
Combing the two previous results, for every interval $[a,b]$, we get 
$$
\lim_{n\to\infty}  \ |\xi_n(t)\cap [a,b]| \ = \  |\xi(t)\cap [a,b]| \ \mbox{a.s.,}
$$
which ends the proof of the lemma.
\end{proof}

We will need one more ingredient that is a simple extension of
Lemma 6.7 in \cite{ScheSS08b}.   

\blem\label{lem-tech} 
 Let $E$ be a Polish space. Let $\{F_i\}_{i\in I}$ and $\{G_i\}_{i\in I}$ be two finite or countable collections of Polish spaces and for each $i\in I$, let $f_i : E \to F_i$ 
 and $g_i : E \to G_i$ be a measurable functions. Let $X,X_k,Y,Y_k$ and $U_{k,i},V_{k,i}$ be random variables $\{k \geq 1, i \in I\}$ such that $X,X_k,Y,Y_k$ take values in $E$ and $U_{k,i},V_{k,i}$ takes values in $F_i$ and $G_i$ respectively. Then
 \begin{eqnarray*}
 \P((X_k, U_{k,i}) \in \cdot) \ \to   \P((X, f_i(X)) \in \cdot), \ \P((Y_k, V_{k,i}) \in \cdot) \ \to   \P((Y, g_i(Y)) \in \cdot) \ \forall i\in I \\
 \mbox{and }  \P((X_k, Y_{k}) \in \cdot) \ \to   \P((X, Y) \in \cdot) \\
 \mbox{implies} \ \  \P((X_k, Y_k, U_{k,i}, V_{k,i})_{i\in I} \in \cdot) \ \to \   \P((X, Y, f_i(X), g_i(Y))_{i\in I} \in \cdot)
 \end{eqnarray*}
 where the convergence is meant in the product topology.
\elem 
\begin{proof}
The convergence of $\{(X_k, U_{k,i})\}_{k\geq1}$ and $\{(Y_k, V_{k,i})\}_{k\geq1}$
implies that the sequence $\{(X_k, (U_{k,i})_{i\in I},Y_k, (V_{k,i})_{i\in I})\}_{k\geq1}$ is tight. Let 
$(X_\infty, (U_{i})_{i\in I},Y_\infty, (V_{i})_{i\in I})$ be a subsequential limit. First,
$(X_\infty, Y_{\infty})=(X,Y)$ in law. Secondly, for every $i\in I$, $(X_\infty, U_i)$ is identical in law with $(X_\infty, f_i(X_\infty))$
which implies that $U_i=f_i(X_\infty)$ a.s. Likewise, $V_i=f_i(Y_\infty)$ a.s.. This completes the proof of the lemma.
\end{proof}
By Lemma 6.12 in \cite{ScheSS08b}, $(S_{\eps_n}(\cU^{b_n}), S_{\eps_n}(\cU^{b_n}(\Sigma_0)))$ converges to $(\Net^b, \Net^b(\Sigma_0))$ in $({\cal H},d_{\cal H})$.
Using Lemma \ref{lem-tech} (in the spacial case $X=Y=\Net^b$, and $U_{k,i}=V_{k,i}$) and the Skorohod Representation Theorem, (\ref{cv-r}) and (\ref{cv-xi}) 
imply that
there exists a coupling between the discrete and the continuum levels such that   
$$
\{\left( S_{\eps_n}(\cU^{b_n}(\Sigma_0),\  R_n(0,T), \  \{\xi_n^0(q)\}_{q\in\Q^+} , \  \xi_n^0(T)  \right)\}_{n\geq1}  \ \to \ \left(\Net^{b}(\Sigma_0), R(0,T),  \{\xi^0(q)\}_{q\in\Q^+} , \xi^0(T)  \right) \  \ {a.s.}. 
$$
From now until the end of this subsection, we   
will work under this coupling. 
In particular, if $V_n$ denotes the analog of $V$ (as defined in Lemma \ref{lem-si}),
then $V_n\Longrightarrow V$ a.s. under our coupling.

In the next lemma, $A_n(z^1,\cdots,z^k)$ will refer to the discrete analog of $A(z^1,\cdots,z^k)$ (i.e., the set of paths in the rescaled net $S_{\eps_n}(\cU^{b_n}(\Sigma_0))$
with discrete finite graph representation $(z^1,\cdots,z^k)$).

\blem\label{lem:zx}
   Let $(z^1,\cdots,z^k)$ be such that $A(z^1,\cdots,z^k)\ \neq \ \emptyset$. For almost every realization of our coupling, there exists 
$(z^1_n,\cdots, z^k_n)$ such that   
$$
\forall i\in\{1,\cdots,k-1\},  \ z^i_n \in R_n(0,T) \rightarrow z^i 
$$
and $x^k_n\in\xi^0_n(T)$ such that
$
z^k_n=(x^k_n,T) \rightarrow z^k. 
$
Further,
the set of paths
$
A_n(z_n^1,\cdots,z_n^k) 
$
converges to $A(z^1,\cdots,z^k)$ in $({\cal H},d_{\cal H})$ a.s..
\elem
\begin{proof}
The existence of the $z_n^i$'s is a direct consequence of the fact that $V_n\Longrightarrow V$ a.s.. The rest of the proof (i.e., the convergence
$A_n(z_n^1,\cdots,z_n^k)$  to $A(z^1,\cdots,z^k)$) is decomposed into two steps.

\medskip

{\bf Step 1.} Let $\delta(z^1,\cdots,z^k) :=d_{{\cal H}}\left( A(z^1,\cdots,z^k), B(z^1,\cdots,z^k)\right)$. Define
$$
\bar A_n(z^1,\cdots,z^k) \ = \ \{\pi_n\in S_{\eps_n}(\cU^{b_n}(\Sigma_0)) \ : \ \ \min_{\pi\in A(z^1,\cdots,z^k)} d(\pi_n,\pi) \ \leq \ \frac 1 3 \delta(z^1,\cdots,z^k)  \}, 
$$
and 
$$
\bar B_n(z^1,\cdots,z^k) \ = \ \{\pi_n\in S_{\eps_n}(\cU^{b_n}(\Sigma_0)) \ : \ \ \ \min_{\pi\in B(z^1,\cdots,z^k)} d(\pi_n,\pi) \ \leq \ \frac 1 3 \delta(z^1,\cdots,z^k)  \}. 
$$
In this step, we will show that  after a finite rank, for any $\pi_n\in \bar A_n(z^1,\cdots,z^k)$,
the finite graph representation of $\pi_n$ is a subsequence of 
$(z^1_n,\cdots,z^k_n)$, and that $A_n(z^1_n,\cdots,z^k_n)\subseteq \bar A_n(z^1,\cdots,z^k)$.

By Corollary \ref{a-set}, $\delta(z^1,\cdots,z^k)>0$ a.s.. From there, it is straightforward to check from the definition of $\delta(z^1,\cdots,z^k)$ and the fact that under our coupling $S_{\eps_n}(\cU^{b_n}(\Sigma_0))$ converges to $\Net^b(\Sigma_0)$ a.s.
that:
\be\label{dh}
d_{{\cal H}}(\bar A_n(z^1,\cdots,z^k), A(z^1,\cdots,z^k))\to 0, \ \ d_{{\cal H}}(\bar B_n(z^1,\cdots,z^k), B(z^1,\cdots,z^k)) \to 0 \ \mbox{a.s.}
\ee
as $n\to\infty$, and that $\bar A_n(z^1,\cdots,z^k)$ and $\bar B_n(z^1,\cdots,z^k)$ form a partition of the set  $S_{\eps_n}(\cU^{b_n}(\Sigma_0))$ for large enough $n$.
Let us now derive two consequences of this result.

First, Corollary \ref{a-set} easily implies that 
$$
\min_{\bz\in V,  \ \bz\notin\{z^1,\cdots,z^k\} } \ \ \ \  
\min_{\pi\in A(z^1,\cdots,z^k), s\geq0} \rho\left((\pi(s),s), \bar z\right) >0. \   
$$
Further, under our coupling, $V_n\Longrightarrow V$ a.s., and since $\bar A_n(z^1,\cdots,z^k)\to A(z^1,\cdots,z^k)$ (in $({\cal H},d_{\cal H}))$, the equicontinuity of the paths in $\cup_n S_{\eps_n}(\cU^{b_n})$ combined with the previous limit implies that
$$
\liminf_{n\to\infty} \ \  \min _{\bz_n\in V_n, \ \bz_n\notin\{z^1_n,\cdots,z^k_n\}} \  \ \ \min_{\pi_n\in \bar A_n(z^1,\cdots, z^k), s\geq0} \ \rho\left((\pi_n(s),s), \bz_n\right)>0.  
$$
As a consequence, after a finite rank, for any $\pi_n\in \bar A_n(z^1,\cdots,z^k)$, 
the (discrete) finite graph representation of $\pi_n$ is a subsequence of 
$(z^1_n,\cdots,z^k_n)$. (In words, some $z_n^i$'s can be ``missed'', but any point in the finite graph representation of $\pi_n$ must be 
in $(z^1_n,\cdots,z^k_n)$). 

Secondly, by reasoning along the same lines, one can deduce that 
after a finite rank, the finite graph representation of any path $\pi_n \in \bar B_n(z^1,\cdots,z^k)$ must be distinct from $(z^1_n,\cdots,z^k_n)$.
Since
$\bar A_n(z^1,\cdots,z^k)$ and $\bar B_n(z^1,\cdots,z^k)$  form a partition of the set of discrete net paths starting at time $0$ (for large enough $n$),
it follows that $A_n(z^1_n,\cdots,z^k_n)\subseteq \bar A_n(z^1,\cdots,z^k)$  after a finite rank.

\medskip

{\bf Step 2.} We now show that $\bar A_n(z^1,\cdots,z^k)=A_n(z^1_n,\cdots,z^k_n)$ for $n$ large enough.
Since we showed in Step 1 that after a finite rank, for any $\pi_n\in \bar A_n(z^1,\cdots,z^k)$,
the finite graph representation of $\pi_n$ is a subsequence of 
$(z^1_n,\cdots,z^k_n)$, and since $A_n(z^1_n,\cdots,z^k_n)\subseteq \bar A_n(z^1,\cdots,z^k)$, 
it is sufficient to show that 
for $n$ large enough, any $\pi_n\in \bar A_n(z^1,\cdots,z^k)$
must pass through $z^1_n,\cdots,z^k_n$.

First, for $n$ large enough, is is easy to see that any $\pi_n\in \bar A_n(z^1,\cdots,z^k)$ must go through $z^k_n$. This is a direct consequence of the fact that $\xi^0(T)$ is locally finite, $\xi^0_n(T)\Longrightarrow \xi^0(T)$ a.s. under our coupling, and the fact $\bar A_n(z^1,\cdots,z^k)$ converges to $A(z^1,\cdots,z^k)$ a.s. (see Step 1). \

Let us now argue that after some rank, any $\pi_n\in \bar A_n(z^1,\cdots,z^k)$
must also go through the point $z_n^i$ for $i<k$. In fact, we will show a little more: We will show by induction on $i\in\N$, that 
for any $m>i$ and any $\bz^1,\cdots\bz^m$ such that $A(\bz^1,\cdots\bz^m)\neq\emptyset$, for $n$ large enough, any $\pi_n\in\bar A_n(\bz^1,\cdots,\bz^m)$
must go through $\bz^1_n,\cdots,\bz^i_n$, where $\{\bz^i_n\}$ is defined analogously to $\{z^i_n\}$ (i.e., as  a sequence of points 
in $V_n$ approximating $\bz^i$).  

   {\bf The case \bf i=1: }    
Let us first show the property for $i=1$. We argue by contradiction. Going to a subsequence if necessary, 
let us assume that there exists 
an infinite sequence $\{\pi_n\}_n$ with $\pi_n\in \bar A_n(\bz_1,\cdots,\bz^m)$  (with $m>1$)
and such that $\pi_n$ does not go through the point $\bz^1_n$
(i.e., if $\bz_n^1=(\bar x^1_n, \bar t^1_n)$ then $\pi^1_n(\bar t^1_n)\neq \bar x^1_n$). 
Going to a further subsequence,
we can assume w.l.o.g. that $\{\pi_n\}$ converges 
to $\pi\in A(\bz^1,\cdots,\bz^m)$. (The fact that $\pi\in A(\bz^1,\cdots,\bz^m)$
easily follows from Corollary \ref{a-set}.)
Since $m>1$, the point $\bz^1=(\bar x^1, \bar t^1)$ is a relevant separation point.
From this observation,
one can construct
a path $\pi'\in\Net^b(\Sigma_0)$  such that $\pi=\pi'$ up to $\bar t^1$ and the paths $\pi$ and $\pi'$ separate after $\bar t^1$. (This is done by concatenating $\pi$
with one of the paths in $(\Wr, \Wl)$ starting from $\bz^1$: for instance, following the notation of Proposition \ref{geo-net-sp} and Fig. \ref{1-2}, 
if $\pi$ is squeezed between $l$ and $r'$, we concatenate $\pi$ with $r$. The constructed path belongs to the net 
using the stability of the net under hopping -- See \cite{SS07}.)

Since  $S_{\eps_n}(\cU^{b_n}(\Sigma_0))$ 
converges to $\Net^b(\Sigma_0)$, there is a sequence  $\{\pi_n'\}_n$ 
in $S_{\eps_n}(\cU^{b_n}(\Sigma_0))$ 
converging to $\pi'$. Let $q\in\Q$, with $q<\bar t^1$. $\xi^0(q)$ is a.s. locally finite and under our coupling $\xi^0_n(q) \Longrightarrow \xi^0(q)$ a.s.. Since 
$$
\lim_{n\to\infty} |\pi_n(q)-\pi_n'(q)| \ = \ |\pi(q)-\pi'(q)|=0,
$$ 
it follows that $\pi_n(q)=\pi_n'(q)$ for $n$ large enough.
Finally, since $\pi$
and $\pi'$ separate at $\bar z^1$, $\pi_n'$ and $\pi_n$ must separate at some    point  $(x_n,t_n)$    with
$\limsup t_n\leq \bar t^1$. However, in Step 1, we showed that after a finite rank,
the finite graph representation of $\pi_n$ is a subset of    $\{\bz^1_n,\cdots,\bz^k_n\}$. 
Since $\limsup t_n\leq \bar t^1$, it follows that for large $n$ , $(x_n,t_n) = \bz^1_n$ and so
$\pi_n$ must go through $\bar z^1_n$ (for $n$ large enough), which yields a contradiction.
This shows the induction hypothesis for $i=1$.

{\bf The induction step from i-1 to i:} Let us now assume that our property holds at rank $i-1$ with $i\geq 2$. Again we argue by contradiction and we assume
that there exists a subsequence $\pi_n\in \bar A_n(\bz^1,\cdots,\bz^m)$ (with $m>i$) such that $\pi_n(\bar t^i_n)\neq \bar x^i_n$ (where $\bz_n^i=(\bar x_n^i, \bar t_n^i)$).
By arguing as in the case $i=1$, we can assume w.l.o.g. that $\pi_n$ converges to $\pi\in A(\bz^1,\cdots,\bz^m)$, and that there exists 
$\pi'\in\Net^b$ such that $\pi$ and $\pi'$ coincide up to $\bar t^i$ and separate afterwards. Let $(\tilde z_1,\cdots, \tilde z^l)$
be the finite graph representation of $\pi'$. By construction, $l>i$ and $\tilde z^j=\bz_j$ for $j\leq i$. Since 
$\bar A_n(\tilde z^1,\cdots, \tilde z^l)\to A(\tilde z^1,\cdots, \tilde z^l)$ (by Step 1),
there exists $\pi_n'\in \bar A_n(\tilde z^1,\cdots, \tilde z^l)$ such that $\{\pi_n'\}$
converges to $\pi'$.

By using our induction hypothesis, $\pi_n$ and $\pi_n'$ must go through the point $\bz^{i-1}_n=\tilde z_n^{i-1}$ after a finite rank. Further, since $\pi$ and $\pi'$
separate at $\bz^i$, the two paths $\pi_n$ and $\pi_n'$ must separate at some time $t_n$ with $\limsup t_n \leq \bar t^{i} $.
However, for $n$ large enough,
the finite graph representation of $\pi_n$ is a subset of $(\bz^1_n,\cdots,\bz^k_n)$, and since $\pi_n$
and $\pi_n'$ coincide at $\bar t_n^{i-1}$, 
$\pi_n$ and $\pi_n'$ must separate at $\bar z_n^i$.
As a consequence,
$\pi_n$ goes through the point $\bz_n^{i}$, which yields the desired contradiction. This  ends  the proof of Lemma
\ref{lem:zx}.
\end{proof}

\begin{proof}[Proof of Proposition \ref{prop:dc1}]
Under our coupling, we already know that $R_n(0,T)\Longrightarrow R(0,T)$ a.s. It is enough to show that
for every $z=(x,t)\in V$, and for  every $z_n=(x_n,t_n)\in V_n$ with $z_n\to z$, we have $\lim_{n\to\infty} \ d^0_n(z_n) = d^0(z)$ a.s.. Note that
if $z=(x,t)\in R(0,T)$ then
$$
d^0(z) \ = \ \inf\{ |x-y| \ : \ \exists (z^1,\cdots,z^k) \ \mbox{and } \pi\in A(z^1,\cdots,z^k), \ \mbox{s.t. $y=\pi(t)$ and} \ \forall i\leq k, \ z\neq z^i \}.
$$
Thus, if we write
$$
m\left((z^1,\cdots,z^k), z\right) \ = \ \inf\{ |x-y| \ : \ \exists \pi\in A(z^1,\cdots, z^k) \ \mbox{s.t. } y=\pi(t) \}
$$
$d^0(z)$ is the infimum of all the $m\left((z^1,\cdots,z^k), z\right)$'s such that  $\forall i\leq k, \ z\neq z^i$. The analogous statement holds at the discrete level.

Let $m_n$ denote the discrete analog of $m$, and let $z_n,z^1_n,\cdots,z^k_n\in V_n$ converging to $z,z^1,\cdots, z^k\in V$
such that $A(z^1,\cdots,z^k)\neq \emptyset$. Lemma \ref{lem:zx} implies that $m_n\left((z^1_n,\cdots,z^k_n), z_n\right)$ converges to $m\left((z^1,\cdots,z^k), z\right)$. The latter convergence statement extends by taking the infimum over finitely many $m\left((z^1,\cdots,z^k), z\right)$'s
and $m_n\left((z^1_n,\cdots,z^k_n), z_n\right)$.
Proposition \ref{prop:dc1} is then a   
direct    consequence of the equicontinuity of  the paths in $\Net^b$ and $\cup_n S_{\eps_n}(\cU^{b_n})$.

\end{proof}

\subsection{Proof of Proposition \ref{prop-principal}}
\label{proof-of}

In the course of the proof of Proposition \ref{finite-SP}, we already proved that
$N<\infty$ (where $N$ is defined as in (\ref{def:N})) and for every $1\leq i\leq N$, the random point measure    $\bar{\cal R}^{i}$ is finite a.s..
When $\theta^i<T$, $\bar {\cal L}^i$ is a singleton a.s., and when $\theta^{i}>T$,
the number of atoms in $\bar {\cal L}^i$ coincides with the cardinality of $\{x \ : \ x\in\xi^0(T)\cap (-L,L)\}$. 
By Proposition \ref{bc-set}, this shows that 
that $\bar{\cal L}^i$ is also a finite measure a.s..
  In this section, to prove Proposition \ref{prop-principal}, our main task will be to prove  the following two lemmas:
\blem\label{conv-0}
$(S_{\eps_n}(\cU^{b_n,\kappa_n}), \bar {\cal L}_n^{1})\ \to \
(\Net^{b,\kappa}, \bar {\cal L}^1)$ in law.
\elem

\blem\label{conv-1}
$
(S_{\eps_n}(\cU^{b_n,\kappa_n}),  \bar {\cal R}_n^{1} ) 
\ \to \ ( \Net^{b,\kappa}, \bar {\cal R}^{1})
$ 
in law.
\elem
We note that in both lemmas, the convergence is meant in the product topology ${\cal H}\times{\cal P}$.

By using  the strong Markov property and the stationarity in time of the killed net
(both at the discrete and continuum level), Lemmas \ref{conv-0} and \ref{conv-1}   will  imply
that 
\beqnn
\{(S_{\eps_n}(\cU^{b_n,\kappa_n}), \{\bar {\cal R}_n^{i}\}_{i \geq 1})  \}_{n\geq1} \ \to\ \  (\Net^{b,\kappa}, \{\bar {\cal R}^{i}\}_{i\geq 1}) \ \mbox{and  } \\
\{(S_{\eps_n}(\cU^{b_n,\kappa_n}), \{\bar {\cal L}_n^{i}\}_{i \geq 1})  \}_{n\geq1} \ \to\ \  (\Net^{b,\kappa}, \{\bar {\cal L}^{i}\}_{i\geq 1}) \ \ \mbox{in law,}
\eeqnn
where the convergence is again meant in the product topology. 
Since $N$ (resp., $N_n$) coincides with the first $i$ such that $\bar {\cal R}^{i}$ (resp., $\bar {\cal R}^{i}_n$) is empty,
this implies
$$
\left\{ \left(S_{\eps_n}(\cU^{b_n,\kappa_n}), \sum_{i=1}^{N_n} \bar {\cal R}_n^{i}\right)   \right\}_{n\geq1} \ \to\ \  \left(\Net^{b,\kappa},  \sum_{i=1}^{N} \bar {\cal R}^{i}\right)$$\ \ \mbox{and  }  \ \
$$\left\{ \left(S_{\eps_n}(\cU^{b_n,\kappa_n}), \sum_{i=1}^{N_n} \bar {\cal L}_n^{i}\right)  \right\}_{n\geq1} \ \to\ \  \left(\Net^{b,\kappa},  \sum_{i=1}^{N} \bar {\cal L}^{i} \right) \ \ \mbox{in law}.
$$
Using Theorem \ref{onewebteo}, 
it is not hard to prove that $S_{\eps_n}(\cU^{b_n,\kappa_n}, \ \cU^{b_n,\kappa_n}(\Sigma_0))$
converges (in law) to $(\Net^{b,\kappa}, \Net^{b,\kappa}(\Sigma_0))$ (this is done in Lemma~6.19 in \cite{ScheSS08b} for the case $\kappa_n,\kappa=0$. The same type of argument carries over for the killed Brownian net).
Proposition \ref{prop-principal} then follows by a direct application of Lemma \ref{lem-tech} (taking $X,Y=\Net^{b,\kappa}$).
The rest of the section is
dedicated to the proofs of Lemmas  \ref{conv-0} and \ref{conv-1}.

\begin{proof}[Proof of Lemma \ref{conv-0}]

In Lemma \ref{cv-xi-lemma}, we showed that 
\begin{enumerate}
\item[(1)] $(\xi_n^0(T), S_{\eps_n}(\cU^{b_{n}})$ converges  to $(\xi^0(T), \Net^b)$ in law.
\end{enumerate}

In \cite{NRS15}, we showed that  the following statements hold in distribution: 
\begin{enumerate}
\item[(2)] $\{(x,\theta^1_n)\ \ \mbox{killing point} \ : \ x\in\xi_n^1(\theta_n^1)\cap (-L,L)   \}$ converges to
the singleton $(x,\theta^1) \in {\cal M}^k(0)$,
\item[(3)] and finally\begin{eqnarray}
\theta_n^1 \ \to \ \theta^1, \ \mbox{and} \ \xi^{0}_n(\theta_n^1) \ \Longrightarrow 
\ \xi^{0}(\theta^1).  \label{111}
\end{eqnarray}
\end{enumerate}
The second item can be found in Proposition 4.4(1) and  Lemma 4.4 in \cite{NRS15}. 
The third item is the content of Corollary 4.2 in  \cite{NRS15}.
Further, analogously to Theorem \ref{onewebteo},
a closer look at the proofs shows that
the latter results were shown by constructing a coupling 
between 
$
\{(S_{\eps_n}(\cU^{b_n}), S_{\eps_n}( {\cal U}^{b_{n},\kappa_{n}} ))\}_{n>0}$
and
$(\Net^{b}, \Net^{b,\kappa})$
such that  under this coupling,
the convergence of the random variables above
and the convergence of $S_{\eps_n}(\cU^{b_n,\kappa_n})$ to $\Net^{b,\kappa}$ hold jointly in probability. 

Using Lemma \ref{lem-tech}, 
this shows the convergence (in law) of $\{(S_{\eps_n}(\cU^{b_n,\kappa_n}),\theta_n^1,\xi_n^0(\theta_n^1))\}$
to its continuum counterpart, and from there it is easy to see that $(S_{\eps_n}(\cU^{b_n,\kappa_n}), \bar {\cal L}_n^{1})\to
(\Net^{b,\kappa},\bar {\cal L}^1)$. 
\end{proof}

\begin{proof}[Proof of Lemma \ref{conv-1}] 
Combing Proposition \ref{prop:dc1}, Lemma \ref{conv-0}, together with the convergence of
$\{(S_{\eps_n}(\cU^{b_n,\kappa_n}),\theta_n^1,\xi_n^0(\theta_n^1))\}$ to $(\Net^{b,\kappa}, \theta^1, \xi^0(\theta^1))$ --- 
see previous lemma --- we have
\beqnn
\forall t\geq0, \ \{(S_{\eps_n}(\cU^{b_n}), \bar R_n(0,t))\}  \ \to \ (\Net^b ,\bar R(0,t)),  \  \ \  \
\ \ \ (S_{\eps_n}(\cU^{b_n,\kappa_n}), \bar {\cal L}_n^{1})\ \to \
(\Net^{b,\kappa}, \bar {\cal L}^1), \ \
\\ \mbox{and}   \ \ \ \ \{(S_{\eps_n}(\cU^{b_n,\kappa_n}),\theta_n^1,\xi_n^0(\theta_n^1))\} \to (\Net^{b,\kappa}, \theta^1, \xi^0(\theta^1))  
 \mbox{          in law}.
\eeqnn
Furthermore, by Theorem \ref{onewebteo}, $S_{\eps_n}(\cU^{b_n}, \cU^{b_n, \kappa_n})$
converges to $(\Net^b, \Net^{b,\kappa})$.
By the Skorohod representation theorem and Lemma \ref{lem-tech} with $(X,Y)=(\Net^b, \Netbk)$,
this implies that there exists a coupling between the discrete and continuum levels such that 
(1)    $S_{\eps_n}(\cU^{b_n})$ converges to $\Net^b$    a.s., 
(2) $\bar {\cal L}_n^1$ converges to $\bar {\cal L}^1$ a.s., 
(3) the convergence statement $\{(S_{\eps_n}(\cU^{b_n,\kappa_n}),\theta_n^1,\xi_n^0(\theta_n^1))\}$  hold a.s., 
and (4) such that 
$
\bar R_n(0,t) \Longrightarrow \bar R(0,t) 
$
a.s. for any positive rational value of $t$ and for $t=T$. We will work under this coupling 
for the rest of the proof.

Next, we first note that
\be\label{bound-case}
\forall t\in\Q^+, t=T, \ \ R(0,t)\cap \{(x,s)\ : \ s\in[0,t], \ x=\pm L\} = \emptyset \ \mbox{a.s.}
\ee
(This is a direct consequence of the fact that $R(0,t)$ is locally finite a.s. and translation invariance of the Brownian net along
the $x$-axis). 
Since $\forall t\in\Q^+, t=T,  \ \bar R_n(0,t)\ \Longrightarrow \bar R(0,t)$ a.s., this implies that
\be\label{bvc}
\forall t\in\Q^+, t=T, \ \sum_{z\in R_n(0, t) \cap \square_{0,t}^L} \ \delta_{z,d^0_n(z)}  \ \Longrightarrow \sum_{z \in  R(0,t) \   \cap \square_{0,t}^L} \ \delta_{z,d^0(z)}  \ \mbox{a.s..} 
\ee
(\ref{bvc}) when $t=T$ entails that Lemma \ref{conv-1} holds for almost every realization on $\{\theta^1>T\}$.
In order to complete the proof of Lemma \ref{conv-1}, 
it remains to show that  
\be\label{targ}
\sum_{z\in R_n(0,\theta^1_n) \cap \square_{0,\theta^1_n}^L} \ \delta_{z,d^0_n(z)}  \ \Longrightarrow \sum_{z \in  R(0,\theta^1) \   \cap \square_{0,\theta^1}^L} \ \delta_{z,d^0(z)}  \ \mbox{a.s..} 
\ee 
We will proceed in three steps.

{\bf Step 1.} 
Since $\Net^b$ is a set of equicontinuous paths a.s., for almost every realization of our coupling,
there exists $L'\in\Q^+$ with $L'>L$    (and $L'$ random)    such that any path in $\Net^b(\Sigma_0)$
hitting a point in $R(0,\theta^1)\cap \square_{0,\theta^1}^L$ 
remains in the box $\square_{0,\theta^1}^{L'}$
between time $0$ and $\theta^1$, and further, under our coupling, the same holds at the discrete level for $n$ large enough.
In the following, we show that for almost every realization of our coupling,
there exists a (random) rational $q<\theta^1$ such that 
\beqn
& R(0,\theta^1) \cap \square_{q,\theta^1}^{L'}   = &  \emptyset, \label{empty0} \\
\mbox{and for $n$ large enough }, & R_n(0,\theta^1_n) \cap \square_{q,\theta^1_n}^{L'}   = &  \emptyset. \nonumber
\eeqn
We first note that if $x<x'$, $R(0,\theta^1)\cap \square_{x',\theta^1}^{L'}\subseteq R(0,\theta^1)\cap \square_{x,\theta^1}^{L}$, 
and that
the same property holds at the discrete level. Since under our coupling, $\theta^1_n\to\theta^1$ a.s., it is enough to show that 
\beqn
\inf\{x<\theta^1\ : \  R(0,\theta^1) \cap \square_{x,\theta^1}^{L'} = \emptyset\} & < & \theta^1,\nonumber \\ 
\limsup_{n\to\infty}   \ \  \inf\{x<\theta^1_n\ : \  R_n(0,\theta^1_n) \cap \square_{x,\theta^1_n}^{L'} = \emptyset\} & < & \theta^1 \ \ \mbox{a.s.} \label{empty1}
\eeqn
The first property simply follows from the fact that $R(0,\theta^1)$ is locally finite (which was established in the proof of Proposition \ref{finite-SP}). 
Let us consider a realization such that the 
$\limsup$ above is equal to $\theta^1$ and let us show that it must belong to a set of measure zero. Going to a subsequence if necessary,
there is a sequence $q_{n}\uparrow \theta^1$ such that
$$ 
R_n(0,\theta^1_n) \cap \square_{q_n,\theta^1_n}^{L'} \neq \emptyset.
$$
In particular,
there must exist a sequence of discrete  separation points $\{z_n\}$ with $z_n\in R(0,\theta_n^1) \cap \square_{q_n, \theta_n^1}^{L'}$, and $\pi^1_n, \pi^2_n \in S_{\eps_n}(\cU^{b_n}(\Sigma_0))$
such that $\pi^1_n$ and $\pi^2_n$ 
pass through the point $z_n$ but attain two distinct points 
$(x^1_n,\theta_n^1)$ and $(x^2_n,\theta_n^1)$ at time $\theta_n^1$.
Recall that $\xi^0(\theta^1)$ is locally finite a.s. and that under our coupling
$$
\xi^0_n(\theta^1_n) \ \Longrightarrow \ \xi^0(\theta^1)  \ \mbox{a.s.}
$$
In particular, the sets $\xi^0_n(\theta^1_n)$ are sparse, in the sense that 
distinct points must remain at a macroscopic distance from each other.
Let us assume w.l.o.g. that the realization under consideration belongs to the set under which the previous limiting statement holds. Since the two distinct points $x_n^1,x_n^2$ alluded to earlier belong to $\xi^0_n(\theta^1)$, this entails 
that the family of paths $\{\pi^1_n,\pi^2_n\}_n$ is not equi-continuous (the two paths $\pi_n^1$ and $\pi_n^2$
separate at time arbitrarily close to $\theta^1_n$ and take two macroscopically distinct values at time $\theta_n^1$).
We already argued that $\cup_n S_{\eps_n}(\cU^{b_n})$ is a set of equicontinuous paths a.s. 
As a consequence, the event
$
\{\limsup_{n\to\infty}   \ \  \inf\{x<\theta^1_n\ : \  R_n(0,\theta^1_n) \cap \square_{x,\theta^1_n}^{L'} = \emptyset\} = \theta^1\} 
$ has measure $0$. It follows that
(\ref{empty1}) (and thus (\ref{empty0})) holds a.s., as claimed earlier. 

\medskip

{\bf Step 2.} For every $q,L'\in\Q^+$,
define $E_{q,L'}$ to be the set of realizations such that (1) for $n$ large enough, any path in   $S_{\eps_n} (\cU^{b_n})(\Sigma_0)$   
hitting a point in $R(0,\theta^1_n)\cap \square_{0,\theta^1}^L$ 
remains in the box $\square_{0,\theta^1}^{L'}$ and (2) $\theta^1>q$, and (3)
 (\ref{empty0}) holds. Note that for any realization in $E_{q,L'}$, any point in $R(0,\theta^1)\cap\square ^L_{0,\theta^1}$
must belong to the rectangle $\square ^L_{0,q}$.

In Step 2, our goal is to show that 
for almost every realization in $E_{q,L'}$,
\be\label{targ2}
\sum_{z\in  R_n(0,\theta^1_n) \ \cap \square_{0,q}^L} \ \delta_{z,d^0_n(z)}  \ \Longrightarrow \sum_{z \in  R(0,\theta^1) \   \cap \square_{0,q}^L} \ \delta_{z,d^0(z)}    \ \ \ \mbox{a.s.}.
\ee
First, under our coupling,
\begin{eqnarray}\label{targ3}
\sum_{z\in R_n(0, q) \cap \square_{0,q}^L} \ \delta_{z,d^0_n(z)}  \ \Longrightarrow \sum_{z \in  R(0,q) \   \cap \square_{0, q}^L} \ \delta_{z,d^0(z)}  \ \mbox{a.s..} 
\end{eqnarray}
Next, on $E_{q,L'}$, since $q<\theta^1$, $R(0,\theta^1)    \cap \square_{0,q}^L$ is a subset of $R(0,q) \   \cap \square_{0,q}^L$: if $t<q$, any two paths separating at $t$ up to time $\theta^1$
must also separate up to time $q$. Let us now show
that for almost every realization in $E_{q,L'}$:
\beqn
\forall z_n=(x_n,t_n) \ \in \ R_n(0,q)  \mbox{ s.t. } z_n\ \to \  z=(x,t)\in R(0,q), \ \mbox{then} \nonumber \\
 \ z\ \in  \ R(0,\theta^1) \ \mbox{ iff } \ z_n\in R(0,\theta^1_n) \ \mbox{for $n$ large enough.}  \ \ \label{eq-step2}
\eeqn
Combining (\ref{targ3}) with (\ref{eq-step2})  yields (\ref{targ2}). 
The rest of Step 2 is dedicated to the proof of (\ref{eq-step2}) (thus showing (\ref{targ2})).

First, let us assume that $z\in R(0,q)$ and that $z$ also belongs to $R(0,\theta^1)$. We now show that $z_n$ (as defined in (\ref{eq-step2}))
belongs to $R(0,\theta^1_n)$ for $n$ large enough. Let $\pi^1,\pi^2\in\Net^b(\Sigma_0)$
passing through $z$ and separating up to $\theta^1$. Since $S_{\eps_n}(\cU^{b_n}(\Sigma_0))$ converges to $\Net^b(\Sigma_0)$ a.s., 
there are two sequences $\pi_n^1$ and $\pi_n^2$ in $S_{\eps_n}(\cU^{b_n}(\Sigma_0))$ such that 
$\pi_n^i$ converges to $\pi^i$, for $i=1,2$. By Lemma \ref{lem-si},  $d^0(z)>0$ a.s.
and further, under our coupling
$d^0_n(z_n)$ converges to $d^0(z)$ a.s.: in other words, $z_n$ must remain macroscopically isolated from the left and from the right. 
Since $\pi_n^1(t_n)-\pi_n^2(t_n)\to 0$, this implies that 
$$
\pi^1_n(t_n)=\pi^2_n(t_n) \ \ \mbox{for $n$ large enough.}
$$
Next, recall that $\pi^1$ and $\pi^2$ separate at $z$ up to time $\theta^1$. By the latter  property,
$\pi^1_n$ and $\pi^2_n$ separate up to $\theta^1_n$    for large enough $n$    at a point $z_n'$ with $\rho(z_n',z_n)\to 0$ as $n\to\infty$.
However, because $R(0,q)$ is locally finite and $R_n(0,q)\Longrightarrow R(0,q)$ a.s.,
$z_n'$ and $z_n$ must coincide after a finite rank. This proves that if $z\in R(0,q)$, but $z$ also belongs to $R(0,\theta^1)$, then $z_n$ also belongs to $R(0,\theta^1_n)$ for $n$ large enough.

Next, let us assume that $z_n\in R_n(0,q)$ and also belongs to $R_n(0,\theta_n^1)$, 
and let $\pi^1_n$ and $\pi^2_n$ in $S_{\eps_n}(\cU^{b_n}(\Sigma_0))$ separating at $z_n$ up
to time $\theta^1_n$. In order to prove (\ref{eq-step2}), we need to show that $z$
belongs to $R(0,\theta^1)$. Going to a subsequence if necessary, $\pi^i_n$ converges to $\pi^i\in\Net^b(\Sigma_0)$ for $i=1,2$.
Furthermore, since $d^0(z)>0$ a.s., arguing as in the previous paragraph, the $\pi^i$'s must go through the point $z$. It remains to prove 
that $\pi^1$
and $\pi^2$ separate up to time $\theta^1$ after passing through $z$. First, since $\xi^0(\theta^1)$
is locally finite a.s., and since $\xi^0_n(\theta^1_n)$ converges to $\xi^0(\theta^1)$ a.s.,
the set of $\xi^0_n(\theta^1_n)$ remains ``sparse'' as $n$ goes to $\infty$.
Since $\pi^1_n(\theta^1_n)\neq\pi^2_n(\theta^1_n)$, this entails
$$
\lim_n \pi^1_n(\theta^1_n) - \pi^2_n(\theta^1_n) >0,
$$
which is equivalent to $\pi^1(\theta^1)\neq \pi^2(\theta^1)$.

We now need to prove that
$\pi^1$ and $\pi^2$ never coincide on the time interval $(t,\theta^1)$ (where $t$ is the time coordinate of $z$).   
Let us assume by contradiction that $\pi^1$ and $\pi^2$ meet again on this time interval.
Since $\pi^1(\theta^1)\neq \pi^2(\theta^1)$, the two paths must separate at another point $z'=(x',t')$, with $t<t'<\theta^1$.
Since we are only considering realizations in $E_{q,L'}$, we must have $t'<q$ (as already mentioned at the beginning of Step 2). 
Let $z_n'\in R_n(0,q)$ converging to $z'$.
Since $d^0(z')>0$ and $d^0_n(z'_n)$ converges to $d^0(z')$, this implies that $\pi_1^n$ and $\pi^n_2$
must also go through $z_n'$ after a finite rank. But this would contradict the fact that $\pi^1_n$ and $\pi^1_n$ 
separate at $z_n$. This ends the proof of (\ref{targ2}).

\medskip

{\bf Step 3.} We are now ready to prove (\ref{targ}). For every  $q,L'\in\Q^+$   , and 
for almost every realization    in  $E_{q,L'}$:
\beqnn
\sum_{z\in R(0,\theta^1) \   \cap \square_{0,\theta^1}^L }  \delta_{z,d^0(z)}
& = &  \sum_{z\in R(0,\theta^1) \   \cap \square^L_{0,q}} \  \delta_{z,d^0(z)} \  + \  \ \sum_{z\in  R(0,\theta^1) \   \cap \square^L_{q,\theta^1}} \  \delta_{z,d^0(z)} \\
& = &   \sum_{z\in R(0,\theta^1) \   \cap \square^L_{0,q}} \  \delta_{z,d^0(z)} \, .
\eeqnn
      Since the analogous identity holds at the discrete level for $n$ large enough, 
(\ref{targ2}) implies that for almost every realization in $E_{q,L'}$
$$
\sum_{z\in R(0,\theta^1_n) \   \cap \square_{0,\theta^1_n}^L }  \delta_{z,d^0_n(z)} \ \Longrightarrow \ \sum_{z\in R(0,\theta^1) \   \cap \square_{0,\theta^1}^L }  \delta_{z,d^0(z)} \ \ \mbox{a.s.}
$$
Finally, since $\P(\cup_{q,L'\in\Q^+} E_{q,L'})=1$ by Step 1,
(\ref{targ}) follows. This   completes   the proof of Lemma \ref{conv-1}.

\end{proof}

\subsection{Proof of Proposition \ref{st2}}
\label{proof-of2}

Going to a subsequence, we can assume w.l.o.g. that the sign of $t_n$ remains constant as $n$ varies.
In the following, we will treat the case $\forall n\in\N, \ t_n\leq0$. The other case can be handled along the same lines.
We decompose the proof of Proposition \ref{st2} into three lemmas.

\blem
Let $\gamma>0$ and define the event
$$
D_{\gamma,n} \ = \ \{ \forall k\in\Z \mbox{ s.t. } \ |(2k+1)\eps_n|\leq\gamma, \ \ G_{n}^{((2k+1)\eps_n,0)}(T) \ \mbox{and } G_{n}^{(0,0)}(T)\  \mbox{only differ by their root}  \}.
$$
Then $\liminf_{\gamma\to0} \liminf_{n\to\infty} \P\left(D_{\gamma,n}\right) = 1$.
\elem
\begin{proof}
Define 
$$
\gamma_n \ := \ \sup\{(2k+1) \eps_n: k\in\N, \ (2k+1)\eps_n \leq \gamma\}.
$$
and let $l_{\gamma,n}$ (resp., $r_{\gamma,n}$) be the leftmost (resp., rightmost) path in $S_{\eps_{n}}(\cU^{b_n})$
starting from $-\gamma_n$ (resp., $\gamma_n$). Finally, we set
$$
\tau_{\gamma,n}:=\inf\{t>0\ : \ r_{\gamma,n}(t)=l_{\gamma,n}(t) \}.
$$
Let $L=1$ in the definition of $\theta_n^1$. Any path in $S_{\eps_n}(\cU^{b_n,\kappa_n})$ starting on $[-\gamma_n,\gamma_n]$ at time $0$ remains squeezed between $l_{\gamma,n}$ and $r_{\gamma,n}$, which easily implies that
$$
D_{\gamma,n}'=\{\tau_{\gamma,n}\leq \theta^1_n, \ \ |r_{\gamma,n}(u)|, |l_{\gamma,n}(u)|< 1, \forall u\in[0,\tau_{\gamma,n}]\}
$$
is a subset of $D_{\gamma,n}$. 
(Indeed, the properties characterizing $D_{\gamma,n}'$
ensure that any path in the killed Brownian net starting at time $0$ at $\bar x_n\in [-\gamma_n,\gamma_n]$
connects the root $(\bar x_n,0)$ to the point $(r(\tau_{\gamma,n}),\tau_{\gamma,n})= (l(\tau_{\gamma,n}),\tau_{\gamma,n})$).
On the other hand, as $n$ go to $+\infty$, $l_{\gamma,n}$ and $r_{\gamma,n}$
converge to drifted Brownian motions starting from $-\gamma$ and $\gamma$ respectively,
those two Brownian motions being independent until they meet.
This easily implies:
$$
\forall  \delta>0, \ \lim_{\gamma\to0} \ \lim_{n\to\infty} \ \P(\tau_{\gamma,n} \leq \delta, \  \ |r_{\gamma,n}(u)|, |l_{\gamma,n}(u)|< 1, \forall u\in[0,\tau_{\gamma,n}]) \ = \ 1,
$$
whereas $\theta^1_n$ converges in distribution to $\theta^1$, with $\theta_1>0$ a.s.. This implies that
$$
\lim_{\gamma\to0} \lim_{n\to\infty} \ \P(D_{\gamma,n}')\ =\ 1,
$$
and since $D_{\gamma,n}'\subseteq D_{\gamma,n}$, this completes the proof of the lemma.

\end{proof}

\blem
Let $\gamma>0$ and define 
$
F_{\gamma,n} \ = \ \{ \forall \pi \in S_{\eps_n}(\cU^{b_n})(z_n),\  \ |\pi(0)|<\gamma/2  \}.
$
Then $\liminf_{n\to\infty} \P\left(F_{\gamma,n}\right)~=~1 $.
\elem
\begin{proof}
This is a direct consequence of the equicontinuity of the paths in
$\cup_{n}S_{\eps_n}(\cU^{b_n})$.
\end{proof}

\blem
Define 
$
H_{n} \ = \ \{ \forall \pi\in S_{\eps_n}(\cU^{b_n,\kappa_n})(z_n), \ e_{\pi}>0     \}.
$
Then $\liminf_{n\to\infty} \P\left(H_{n}\right) = 1$.
\elem
\begin{proof}
Again, we let $L=1$ in the definition of $\theta_n^1$.
Using translation invariance (both in time and space), we have
\beqnn
\P(H_{n}) 
& =  & 
\P(\forall \pi\in S_{\eps_n}(\cU^{b_n,\kappa_n})(0,0), \   e_{\pi}>-t_n) \\
& \geq &
\P(\forall \pi\in S_{\eps_n}(\cU^{b_n,\kappa_n})(0,0), \ |\pi(u)|\leq 1 \ \forall u\in[0,-t_n], \ \theta^1_n>-t_n).
\eeqnn
Finally, the conclusion follows from the equicontinuity of $\cup_{n}S_{\eps_n}(\cU^{b_n,\kappa_n})$
and the fact that $\theta_1^n$ goes to $\theta_1$, with $\theta^1>0$ a.s..
\end{proof}

We are now ready to complete the proof of Proposition \ref{st2}. It is not hard to check that for every $\gamma>0$, 
$$
D_{\gamma,n}\cap F_{\gamma,n}  \cap H_{n} \subset \{\mbox{$G_n^{z_n}(T)$ and $G_n^{(0,0)}(T)$ are isomorphic}  \} 
$$
and thus, 
$$
\P( \{\mbox{$G_n^{z_n}(T)$ and $G_n^{(0,0)}(T)$ are isomorphic}  \}^c ) \ \leq \  \P(D_{\gamma,n}^c) \ + \  \P(F_{\gamma,n} ^c) \ + \ \P(H_{n}^c). 
$$
Letting successively $n$ and $\gamma$ go to $\infty$ and $0$ respectively completes the proof of Proposition \ref{st2}.

\section{Appendix}

In this section we present three models that can be written as a voter model perturbation (VMP). In the following, 
we use the same notation as in Section \ref{decomposition-1}.

\subsubsection{Spatial Stochastic Lotka-Volterra model.} 

As in \cite{CDP11}, we consider  a discrete version of the stochastic Lotka-Volterra model
as introduced by Pacala and Neuhauser \cite{NP00}. It is a spin system
with each site taking value in $\{0,1\}$.
At each time step, the dynamics can easily be
described by a two steps procedure (between $t$
and $t+1$). First,
the particle at $x$ dies with probability $\alpha(1-m_\eps(1-\eta(x))\  f_{x,\eta}(1-\eta(x)))$,
where $m_\eps$ is a function from $\{0,1\}$ to $(0,1)$ and $\alpha\in[0,1]$. 
Secondly,
if it dies, it is replaced by the type of one 
its neighbor chosen according to
the transition kernel $K$. 

For the sake of illustration, we show that when $m_\eps(1)=m_\eps(0)=\eps>0$ (symbiotic
symmetric case) the Lotka-Volterra model
can be written as a perturbation of the voter model. 
Other cases (i.e. when $m_\eps$ is not constant and might be negative)
can be treated in a similar fashion.

\bprop\label{symbiotic}
The symmetric symbiotic Lotka---Volterra model
can be written as a perturbation of the voter model.
\eprop
\begin{proof}
We will show that the symmetric symbiotic Lotka---Volterra model
can be written as a perturbation of the voter model 
\be\label{check1}
P_{x,\eta}^{\eps} \ = \ (1-\eps)\bar P^{v}_{x,\eta}  \ +  \eps B_{x,\eta} 
\ee
with
$$
\bar P_{x,\eta}^{v} \ = \ (1-\alpha) \delta_{\eta(x)} + \alpha\sum_{i} f_{x,\eta}(i) \delta_i
$$
and $B_{x,\eta}$ is a boundary noise with
\begin{eqnarray*}
Q & = & \delta_0 \times K(0,dy) \times K(0,dy) \\ 
g_{0,1,0} = g_{0,0,1} & = &  \half\alpha\delta_1 +(1-\half\alpha)\delta_0 \ \ \ \text{and $g_{0,i,j}=\delta_0$
otherwise.}  \\ 
g_{1,1,0} = g_{1,0,1}  & = & \half\alpha\delta_0 + (1-\half\alpha)\delta_1 \ \ \ \text{and $g_{1,i,j}=\delta_1$
otherwise}
\end{eqnarray*}
(Note that there is no bulk noise in this model.)
For $\eta(x)=0$, we have 
\begin{eqnarray*}
P_{x,\eta}^\eps(1) & = &  \alpha(1-\eps f_{x,\eta}(1)) f_{x,\eta}(1) \\
P_{x,\eta}^\eps(0) &  = &  (1-\alpha(1-\eps f_{x,\eta}(1))) +  
\alpha(1-\eps f_{x,\eta}(1))  f_{x,\eta}(0) 
\end{eqnarray*}
which can be rewritten as
\begin{eqnarray*}
P^\eps_{x,\eta}(1) & = & (1-\eps) [ \alpha f_{x,\eta}(1) ] \ + \ \eps [ \alpha f_{x,\eta}(1) - \alpha f_{x,\eta}(1)^2] \\
                     & = & (1-\eps) [ \alpha f_{x,\eta}(1)  ] \ + \ \eps [ \alpha f_{x,\eta}(0) f_{x,\eta}(1) ]   \\
P^\eps_{x,\eta}(0) & = & (1-\eps) [ \alpha f_{x,\eta}(0) + (1-\alpha) ] \ + \ \eps [ 1- \alpha f_{x,\eta}(0)f(1)].
\end{eqnarray*}
One can get a symmetric expression for $\eta(x)=1$. From there,
one can directly check that (\ref{check1}) holds.
\end{proof}

\subsubsection{The $q$ colors Potts model.}\label{Potts} 
\label{Potts-sect}

In this example, we show how our particular scaling for the voter model
perturbation emerges naturally when considering a
stochastic Potts model at low temperature.
 Stochastic one-dimensional Potts model are  
Markov processes $\{\n_t\}$ in discrete or continuous time taking values
in the state space ${\{1,...,q\}}^\Z$ with an invariant distribution equal (or closely related) to  
the Gibbs measure (at inverse temperature $\beta$) 
with formal Hamiltonian $H$ given
by
\beq
H(\n)= \sum_{x\in\Z} \delta_{\n(x)\neq\n(x+1)}, \label{Ham} 
\eeq
where $\n(x)$ is the $x$ coordinate of the configuration $\n$. Interpreting $\{1,...,q\}$ as a 
set of colors, $H$ simply counts the number of boundaries between the 
color-clusters of the system. In dimension $1$,
it is well known that for $\beta<\infty$ the Gibbs measure 
is unique, and that there is no phase transition for 
the (static) Potts model.

\bigskip


We will primarily be concerned with the following discrete time model in which at each integer time, the values of  $ \n_t(x)$ for $x \in \Z$ all update simultaneously with the following probabilities
\[\
w_\beta= \frac{2 (e^\beta-1)}{q+2(e^\beta-1)}, b_\beta= \frac{(e^\beta-1)^2 q}{((e^{2\beta }-1)+q)(q + 2 (e^{\beta}-1))}, \kappa_\beta = \frac{q}{e^{2 \beta} +(q-1)}
\]

This particular dynamics qualifies
as a perturbation of the voter model, with the parameter $\epsilon$
corresponding to $e^{-\beta}$ and further
$$
k_\beta e^{2\beta}\to q \ \ \ \mbox{and}  \ \ \ b_\beta e^{\beta} \to q/2.
$$

\begin{rmk}\label{rmk2}
The scaling of the branching and killing parameters coincides with the scaling of Theorem \ref{teo1}(3)(i).
\end{rmk}

Later in this section, we will discuss the relation of the Potts model Gibbs measure to this discrete time process. But first, we motivate the choice of transition probabilities by considering a continuous time voter model perturbation with transition rates given by $w_\beta,b_\beta$ and $k_\beta$

\bprop
\label{exact}
Let us consider the continuous time Markov process with transition rates $(w_\beta,b_\beta,\kappa_\beta)$ -- the transitions corresponding to
a walk, branch and kill move respectively (see end of Section \ref{decomposition-1}) --
and let $\{B_{x,\eta}\}$ for $\eta(x-1)\neq \eta(x+1)$ and $p$ be the uniform distribution on $\{1,\cdots,q\}$
and $B_{x,\eta}=\delta_{\eta(x-1)}$ when $\eta(x-1)=\eta(x+1)$. 
This model defines a
contiuous time $q$-states Potts model 
at inverse temperature which is reversible with respect to Gibbs measure.
\eprop

\begin{proof}
When a Poisson clock rings at site $x$ at time $t$ that site is updated based on the values of spins at the two nearest neighbor sites. Therefore in the detailed
balance equations needed for reversibility  $\eta(x-1)$
and $\eta(x+1)$ are fixed while $\eta_{t_-}(x)$ makes a transition to
$\eta_{t}(x)$. 
For  $\overline{c}_i = (c_\ell, c_i,c_r)$ the color configuration in
$\{1,\dots,q\}^3$ of $(\eta_{t-}(x-1),\eta_{t-}(x),\eta_{t-}(x+1))$, $c_f$
the color of $\eta_{t}(x)$,  with  $\overline{c}_f=(c_\ell, c_f,c_r)$, the detailed balance
equations for the rates $q$  of the transitions 
$\overline{c}_i \to \overline{c}_f$ and $\overline{c}_f \to \overline{c}_i$ are
\begin{equation}
\label{detailedbal}
q(\overline{c}_i \to \overline{c}_f)/q(\overline{c}_f \to \overline{c}_i)
\, = \, \exp({-\beta \Delta H (\overline{c}_i \to \overline{c}_f)}) \, ,
\end{equation}
where
\begin{equation}
\label{energydiff}
\Delta H (\overline{c}_i \to \overline{c}_f) \, = \, H(\overline{c}_f) - H(\overline{c}_i)
\, = \, (\delta_{c_f \neq c_\ell} - \delta_{c_i \neq c_\ell}) \, + \,
(\delta_{c_f \neq c_r} - \delta_{c_i \neq c_r}) \, .
\end{equation}
One can readily check that (\ref{detailedbal})
is satisfied for our particular choice of the triplet $(w_\eps,b_\eps,\kappa_\eps)$.
\end{proof}


Before returning to our discrete time model where all vertices update simultaneously, we consider a model in which, say, even (resp., odd) vertices update at even (resp., odd) integer times. We leave it as an exercise for the reader to convince herself that the detailed balance calculation in the proposition above for continuous time model implies the same for each time step in this (alternating) discrete time model.

For the (simultaneously updating) discrete model, we note two elementary properties: (A) if one restricts attention to the even (resp., odd) space-time sublattice of $\Z \times \N$, these two restricted stochastic processes are independent of each other. (B) The odd sublattice restriction of this dynamics is identical to the odd sublattice restriction of the alternating discrete model. These two properties imply that an invariant measure for simultaneously updating process is $P \times P'$ where $P$ (resp., $P'$) is the Gibbs measure on $\Z$ restricted to $\Z_{even}$ ($\Z_{odd}$).
\subsubsection{Noisy biased voter model.}\label{Biased-voter}
We start with a general description
of this model.
Let $(\kappa_\eps,b_\eps,\alpha)$ be three 
non-negative numbers with $\alpha\in[0,1]$. In a 
continuous time setting, the noisy biased voter model
is a spin system
with transition rates at $x$ :
\beqnn
\text{transition to 1}     &  & (1+b_\eps) f_{x,1}(\eta) + \kappa_\eps (1-\alpha), \\
\text{transition to  2}    &  & f_{2,\eta}(x) + \kappa_\eps \alpha. 
\eeqnn
A discrete time analog of such system
consists in updating each site $x$ at each discrete time step according to the following rule
\beqn\label{expand}
\text{take on color $1$ with probabilty}  \  \ P^\eps_{x,\eta}(1) = \frac{(1+b_\eps) f_{1,\eta}(x) + (1-\alpha)\kappa_\eps }{1+b_\eps f_{1,\eta}(x)+ \kappa_\eps}, \nonumber\\ 
\text{take on color $2$ with probability}   \  \ P^\eps_{x,\eta}(2) = \frac{f_{2,\eta}(x)+\alpha \kappa_\eps}{1+b_\eps f_{1,\eta}(x)+ \kappa_\eps}, 
\eeqn
\bprop\label{biased-voter}
Let $b,k\geq0$, and let $\kappa_\eps = \eps^2 \kappa $ and $b_\eps = \eps b$. In the case $d=1$,
the previous model is a voter model perturbation. 
\eprop

\begin{proof}
We show that the biased voter model is a voter model perturbation
with a boundary nucleation part given by
\beqn
& Q  =  K(0,dy)\times K(0,dy) \times K(0,dy)\no  \\
& g_{2,2,2}  = \delta_2, \ g_{1,1,1}=\delta_1 \\ 
& g_{1,2,2}  =  g_{2,1,2} = g_{2,2,1} = \frac{1}{3} \delta_2 + \frac{2}{3} \delta_{1} \no \\
& g_{2,1,1}  =  g_{1,2,1} = g_{1,1,2} = \frac{b_\eps}{3} \delta_2 + (1-\frac{b_\eps}{3}) \delta_{1}  \label{peew}
\eeqn
and a bulk nucleation part given by
$$
p^\eps(2)= \alpha, \ \ \mbox{}, p^\eps(1)=1-\alpha,
$$
while the rate of branching is $b_\eps$ and the killing rate is $\kappa_\eps$.
Expanding (\ref{expand}) in the $\eps$ parameter, we get
that 
\beqn
P^\eps_{x,\eta}(1) & = &   (f_{x,\eta}(1) + r_\eps(f_{x,\eta}(1))) + b_\eps f_{x,\eta}(1)\left(1-(1+b_\eps)f_{x,\eta}(1) + b_\eps f_{1,\eps} \right) + \kappa_\eps ( (1-\alpha)-f_{x,\eta}(1)) \nonumber \\
P^\eps_{x,\eta}(2) & = &   (f_{x,\eta}(2) - r_\eps(f_{x,\eta}(1)))   + b_\eps f_{2,\eps} ( b_\eps f_{x,\eta}(1)^2  -f_{x,\eta}(1)) + \kappa_\eps ( \alpha -f_{x,\eta}(2)).
\eeqn
where $r_\eps$ is a function such that $r_\eps(f_{x,\eta}(1)) =  O(\eps^3)$ and for $\eps$ small enough, for every value of $f_{x,\eta}(1)\in[0,1]$:
$$
f_{x,\eta}(1) + r_\eps(f_{x,\eta}(1))/w_\eps, \ \  f_{x,\eta}(2) - r_\eps(f_{x,\eta}(1))/w_\eps \geq 0. 
$$ 
After some straightforward manipulations, this can be rewritten under the following form
\beqn
P^\eps_{x,\eta}(1) & = & w_\eps \left(f_{x,\eta}(1) +  r_\eps(f_{x,\eta}(1))/w_\eps \right) 
+ 
b_\eps \left( f_{x,\eta}(1)^3 \ + \  3 f_{x,\eta}(2)^2 f_{x,\eta}(1)  \times  \frac{2}{3} \ + \ 3 f_{x,\eta}(1)^2 f_{x,\eta}(2) \times (1-\frac{b_\eps}{3}) \right)
 \nonumber \\ & &+ 
\kappa_\eps ( 1-\alpha) \nonumber \\
P^\eps_{x,\eta}(2) & = & w_\eps \left(f_{x,\eta}(2)  - r_\eps(f_{x,\eta}(1))/w_\eps\right) 
+ 
b_\eps \left( f_{x,\eta}(2)^3 +   3 f_{x,\eta}(2)^2 f_{x,\eta}(1)  \times  \frac{1}{3}  + \ 3 f_{x,\eta}(1)^2 f_{x,\eta}(2) \times \frac{b_\eps}{3} \right) 
+ 
\kappa_\eps \alpha, \nonumber
\eeqn
where $w_\eps=1-b_\eps-\kappa_\eps$.

Under this condition,
one can readily check that the noisy voter model 
can be written as a 
voter model perturbation with 
a boundary and bulk nucleations chosen as 
described above.
\end{proof}

\bigskip

{\bf Acknowledgments}  The research of K.R. was partially supported by Simons Foundation
Collaboration grant 281207. The research of C.M.N. was supported in part by U.S. NSF
grants DMS-1007524 and DMS-1507019.

\end{document}